\documentclass[11pt, reqno]{amsart}
\ifdefined\pdfminorversion\pdfminorversion=7\fi
\usepackage{amsmath,amsthm,amsfonts,amssymb,mathrsfs,bm,graphicx,stmaryrd,caption}
\usepackage{subcaption}
\usepackage{epstopdf}
\usepackage{mathtools}
\usepackage{dsfont}
\usepackage{multicol}
\usepackage{adjustbox}
\usepackage[colorlinks=true,linkcolor=blue,citecolor=blue,breaklinks]{hyperref}
\usepackage{url}

\usepackage{breakurl}
\usepackage{bbm}
\usepackage{upgreek}
\newcommand{\scT}{\mathscr{T}}

\newcommand{\cR}{\mathcal{R}}

\newcommand{\cQ}{\mathcal{Q}}

\newcommand{\cX}{\mathcal{X}}

\newcommand{\RR}{\mathbb{R}}

\newcommand{\hor}{\mathrm{hor}}

\newcommand{\PP}{\mathbb{P}}

\newcommand{\cB}{\mathcal{B}}

\newcommand{\EE}{\mathbb{E}}

\newcommand{\ZZ}{\mathbb{Z}}

\newcommand{\cI}{\mathcal{I}}

\newcommand{\hitset}{\mathrm{HitSet}}

\newcommand{\disvol}{\mathrm{DisVol}}

\usepackage[letterpaper,hmargin=1.0in,vmargin=1.0in]{geometry}
\parskip 	\smallskipamount

\newtheorem{theorem}{Theorem}
\newtheorem{lemma}[theorem]{Lemma}

\newtheorem{question}[theorem]{Question}

\newtheorem{corollary}[theorem]{Corollary}
\newtheorem{proposition}[theorem]{Proposition}

\theoremstyle{definition}
\newtheorem{remark}[theorem]{Remark}
\newcommand{\ind}{\mathbbm{1}}
\newcommand{\wgt}{\mathrm{Wgt}}

\def\Var{{\rm Var}}

\newcommand{\cA}{\mathcal{A}}

\newcommand{\cD}{\mathcal{D}}
\newcommand{\cE}{\mathcal{E}}
\newcommand{\cN}{\mathcal{N}}

\newcommand{\cM}{\mathcal{M}}

\newcommand{\Cov}{\mathrm{Cov}}

\newcommand{\bn}{\mathbf{n}}

\newcommand{\0}{\mathbf{0}}

\newcommand{\coarse}{\mathrm{Coarse}}

\newcommand{\nearmax}{\mathrm{NearMax}}

\usepackage[backend=biber,style=alphabetic,doi=false,maxalphanames=10,maxnames=50]{biblatex}

\AtBeginBibliography{\small}

\usepackage{longtable}
\usepackage{xcolor}

\usepackage{tikz}
\usetikzlibrary{arrows.meta,calc,decorations.pathreplacing}
\definecolor{clusterblue}{RGB}{39,88,135}
\definecolor{refineteal}{RGB}{18,112,104}
\definecolor{witnessgold}{RGB}{174,112,25}
\tikzset{>=Stealth, every picture/.style={font=\small},
  guide/.style={draw=black!35,densely dashed,line width=.45pt},
  cluster/.style={draw=clusterblue,fill=clusterblue!10,line width=.75pt},
  child/.style={draw=refineteal,fill=refineteal!18,line width=.7pt}}

\newcommand{\norm}[1]{\left\lVert#1\right\rVert}

\newcommand{\dint}[2]{[\![#1,#2]\!]}
\newcommand{\slab}[2]{\dint{#1}{#2}_{\RR}}

\begin{document}
\title[]{Geodesic traces in dynamical Brownian last passage percolation}
\author[]{Manan Bhatia}
\address{Manan Bhatia, Department of Mathematics, University of Geneva, Geneva, Switzerland}
\email{mananbhatia1701@gmail.com}
\date{}
\begin{abstract}
We consider Brownian last passage percolation (BLPP) in which the Brownian
increment process on each unit horizontal interval is independently
resampled at rate one. {By combining strong passage-time
stability estimates with a multiscale analysis of static near-optimal
paths, we show that, for every $\varepsilon>0$, the}
{union of all geodesics between two KPZ-scale rectangles of
transverse width of order $n^{2/3}$ and longitudinal length of order
$n$, separated by a distance of order $n$, visits at most
$n^{1+\varepsilon}$ unit horizontal cells in the bulk during the
critical time interval $[0,n^{-1/3}]$, both in}
expectation and with stretched-exponentially high probability.
{We also obtain the quantitative bound
$n\exp\{C(\log\log n)^2\}$ on the expected hitset size, with a
corresponding failure probability at most $Ce^{-c(\log n)^2}$.}
Using this, we establish that the set of times admitting a
non-trivial bigeodesic has almost surely zero Hausdorff measure for
the subpolynomially decaying gauge $H(r)=\exp\{-L(r)^2\log L(r)\}$,
where $L(r)=\log\log(1/r)$, as $r\downarrow0$.
In particular, this set almost surely has Hausdorff dimension zero. For each fixed
deterministic non-axial direction, we further show that almost
surely no time admits a bigeodesic in that direction.

\end{abstract}
\maketitle
\setcounter{tocdepth}{1}
\tableofcontents

\section{Introduction}

Last passage percolation (LPP) is a model of random geometry in which
directed paths are chosen to maximise the weight they collect. A basic example
is exponential LPP: independent exponential weights are assigned to
the vertices of $\ZZ^2$, and the weight of an up-right lattice path is
the sum of the weights of its vertices. The passage time between two
ordered vertices is the largest weight of such a path, and a path
attaining this maximum is called a geodesic. LPP models are believed to belong to
the Kardar--Parisi--Zhang universality class \cite{KPZ86}, in which
geodesics and passage times exhibit characteristic fluctuations on large scales.
Under the corresponding rescaling, their passage-time fields are
conjectured to converge to the directed landscape, constructed in
\cite{DOV22}, and this has been shown for integrable models \cite{DOV22, DV21}.

Instead of working with a ``static'' LPP model, one might also evolve the vertex weights in a dynamical fashion and study the behaviour of this richer model. For example, by attaching
an independent rate-one Poisson clock to each vertex and resampling
its weight whenever the clock rings, one obtains a stationary
dynamical LPP model. The associated geodesics now change with time, and a central question is how a small perturbation of the environment
affects the geometry of a long geodesic. A natural route to study the above phenomenon is to compare the structure of a geodesic at two times in the dynamics, and this has been investigated \cite{Cha14,GH24}. One may also ask about the evolution of a geodesic over a dynamical time interval: how often does it change, and how large is the trace it sweeps out in the plane?

Such questions about finite paths are closely related to the possible
existence of exceptional times for infinite paths. A bigeodesic is a
bi-infinite directed path every finite portion of which is a geodesic.
For static LPP models, non-trivial bigeodesics are expected to be absent;
this has been proved for exponential LPP in \cite{BHS22,BBS20}.
Here the qualification ``non-trivial'' excludes the entirely horizontal
and entirely vertical paths, which are always bigeodesics.
In planar first passage percolation, the bigeodesic question is also
connected to the existence of non-constant ground states of the
ferromagnetic Ising model with random coupling constants; see
\cite[Section 4.5.2]{ADH17}. Evolving the passage weights corresponds
in this connection to evolving the coupling constants.
In a stationary dynamical model, an event that
has probability zero at each deterministic time may nevertheless
occur at random times. It is therefore natural to ask whether
bigeodesics can appear at such exceptional times, and, if they do,
how large this set of times can be.

In this work, we address the latter question for Brownian last passage
percolation (BLPP) with the discrete resampling dynamics introduced
in \cite{B25}. BLPP has a continuous horizontal coordinate, and its
paths collect increments of independent Brownian motions on successive
rows. The dynamics resample these increments on unit horizontal
intervals. We prove that the set of times at which a non-trivial
bigeodesic exists has almost surely Hausdorff dimension zero, and in fact, its Hausdorff measure
vanishes for the subpolynomially decaying gauge function $H$ in Theorem~\ref{thm:main}.
\begingroup
Our approach is to estimate directly the region swept out by finite
geodesics. A geodesic of length of order $n$ visits, in expectation,
at most $n\exp\{C(\log\log n)^2\}$ unit horizontal cells in its bulk
as time varies over $[0,n^{-1/3}]$.
\par\endgroup

We also rule out exceptional times admitting a bigeodesic
in any fixed deterministic non-axial direction. Whether such times
actually exist when the direction is allowed to be random still remains
an open question; see \cite[Question 2]{BE25} and the discussion following Conjecture 4 therein for the corresponding existence question and
its predicted behaviour in dynamical LPP.

This work is part of the program initiated in \cite{BE25,B25} to
develop a theory of exceptional times at which bigeodesics exist
in dynamical LPP,
motivated by the theory of exceptional infinite clusters in dynamical
percolation \cite{HPS97}. For critical site percolation on the
triangular lattice, there is almost surely no infinite open cluster
at any fixed time. Nevertheless, when the states of the sites are
independently refreshed, Schramm and Steif~\cite{SS10} proved that
exceptional times with an infinite open cluster exist, and Garban,
Pete and Schramm~\cite[Theorem 1.4]{GPS10} showed that their set has
Hausdorff dimension $31/36$. In dynamical LPP, the corresponding
questions concern the appearance of bigeodesics and the size of the
set of times at which they occur. The passage-time and hitset
estimates developed here provide finite-scale inputs to this program.

\subsection{Previous work}

The study of noise sensitivity initiated by \cite{BKS99} asks whether a small resampling of a random
environment can substantially change its macroscopic features.
For Gaussian optimisation problems, \cite{Cha14} developed
the connection between unusually small fluctuations of the optimal
value and instability of its optimiser under perturbations.
In LPP, this instability can be measured through the overlap of
geodesics at two times. For BLPP under the Ornstein--Uhlenbeck dynamics,
\cite{GH24} identifies $n^{-1/3}$, up to subpolynomial corrections,
as the critical time scale for length-$n$ geodesics: the overlap
changes from order $n$ to a smaller order as this scale is crossed.
These dynamics evolve the Brownian environment continuously.
For independent resampling of vertex weights in discrete LPP models,
\cite{ADS24} relates the transition between stability and chaos to
$\Var(T)/n$, under assumptions on the vertex weight distribution.
If the passage-time variance has the conjectured order $n^{2/3}$,
this again gives the scale $n^{-1/3}$.
\begingroup
For dynamical critical first-passage percolation on the triangular and square
lattice, \cite{DHHL23,DHHL26} investigate exceptional times with atypical passage-time growth. {The recent work \cite{GGH26} studies a form of chaos for the
continuum directed random polymer \cite{AKQ14} under changes in
temperature. As a further consequence, they prove that the directed
landscape is a two-dimensional black noise in the sense of
Tsirelson and Vershik~\cite{TV98}. This extends the earlier temporal
black-noise result of \cite{HP24} to a decomposition into independent
local randomness in space--time rectangles. It can be interpreted
as an intrinsic form of noise sensitivity for the directed landscape:
if each rectangle's local randomness is independently resampled
with a fixed positive probability, however small, then
square-integrable observables of the original and resampled
landscapes decorrelate as the rectangular mesh tends to zero.}
\par\endgroup

The work \cite{B25} introduced a discrete dynamics on BLPP and
{investigated the total number of changes accumulated by a geodesic
as time proceeds, a quantity referred to therein as \emph{geodesic switches}.}
In the bulk of a length-$n$ geodesic, the expected number of these
coarse-grained changes over an interval of length $h$ is at most
$n^{5/3+o(1)}h$. A {particular consequence of this} is an upper bound of $1/2$ for the
Hausdorff dimension of exceptional times. For bigeodesics restricted
to a deterministic direction fixed in advance, the upper bound in
that work is already zero. Here we use
a different strategy, estimating the size of the trace directly
without counting switches. A cell is counted only once, even if the
geodesic leaves it and returns many times. Repeated switches between
previously explored portions can therefore contribute to the switch
count without enlarging the trace, and this is why we estimate the size
of the trace directly. The stronger trace estimate also
improves the fixed-direction conclusion from dimension zero to
non-existence; see Theorem~\ref{thm:fixeddirection}.

The companion work~\cite{BE25} approaches the existence question in
dynamical exponential LPP. For suitable opposite $\Theta(n)$ length segments
at distance of order $n$, it gives a lower bound $c/\log n$ for the
probability that some geodesic between the segments passes through
the origin at a time in $[0,1]$. The crucial feature is that this
probability decays subpolynomially: it is larger than $n^{-a}$ for
every fixed $a>0$ and sufficiently large $n$, and this is the sense in
which the result gives near-existence of bigeodesics. The same work
conjectures that, for natural dynamical LPP models, the set of
exceptional times has Hausdorff dimension zero, even if
it is nonempty; see \cite[Conjecture 4]{BE25}.

The discussion following Question~19 in \cite{BE25} suggests a route
to such an upper bound: if a geodesic visits many locations while the
environment changes only slightly, one should be able to detect many
near-optimal alternatives in the static environment. In this paper, we implement
this idea in the discrete BLPP model using a new passage-time
stability estimate. Theorem~\ref{thm:stability} gives a discrete-BLPP
counterpart of the concentration requested in \cite[Question 20]{BE25};
its relation to the $\sqrt{nt}$ scale is explained below.
The geometric difficulty is that the locations visited by a geodesic
as time varies over an interval of length $n^{-1/3}$ need not be well
separated. They can form a hierarchy of clusters, with smaller
excursions along each macroscopic alternative and still smaller
excursions within those. Figure~\ref{fig:blppclusters} depicts this
picture using BLPP staircases.

To control this hierarchy, we apply stability and static near-maximum
estimates at successively smaller scales. The static input bounds
the number of well-separated near-maximisers by comparing BLPP
weight profiles with Brownian motion. It combines
Brownian comparison for the routed profile with an estimate controlling the number of near-maximisers for Brownian motion \cite{CHH23,GH23}.
At each scale, stability places every crossing of a shorter geodesic
segment within unit distance of a near-maximiser of the corresponding
static profile, with a tolerance determined by the segment's own
length.
Its endpoints, however, are selected by the evolving geodesic and
are therefore random. The method of \cite[Section 5]{B25} addresses
such random endpoints by sampling an independent Poisson cloud of
endpoint pairs: with high probability, every central path portion
of interest agrees with a portion of a geodesic between one of the
sampled pairs. We refer to this procedure as \textit{Poisson capture}.
Its geometric foundation is {a} one-sided volume-accumulation
estimate in \cite[Section 5]{BB23}, proved for
exponential LPP and adapted to BLPP in \cite[Section 11]{B25}.
This estimate gives large sets of starting points whose geodesics
merge quickly with a given path; it supplies the basin-volume
bound that allows us to find representatives from the independently sampled Poisson cloud.
We perform this construction simultaneously in every possible region
containing a crossing and at every scale used in the refinement.
For every sampled endpoint pair, passage-time stability relates its
dynamic crossings to static near-maximisers. The static estimate then
bounds the number of separated near-maximisers in the appropriate
window, using the length scale of that pair, and provides a cover by
few intervals. We first obtain one event on which {the \textit{Poisson capture} statement} and these
covers hold simultaneously for all the sampled pairs in all the
regions and scales under consideration. On this event, we refine the
intervals containing the actual crossings from large scales to small
scales and count only the intervals retained. This multiscale
recursion gives the required bound on the entire dynamical trace.
Section~\ref{sec:proofoutline} explains the argument in more detail.

\begin{figure}[tbp]
\centering
\includegraphics[width=\textwidth]{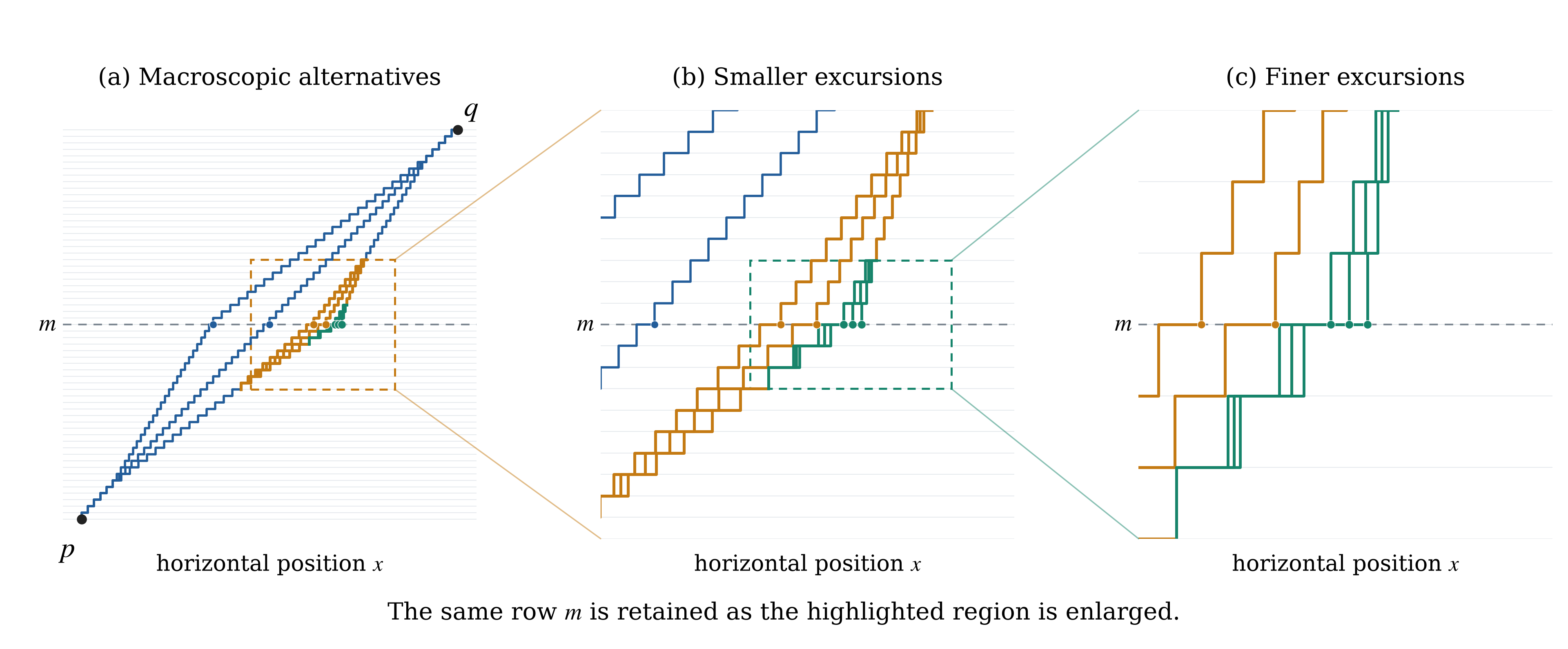}
\caption{The cluster hierarchy in the trace of the evolving geodesic
$\Gamma_p^{q,t}$ as $t$ ranges over a short dynamical interval.
The staircases represent different positions taken by this geodesic
during that interval. Blue paths show alternatives on the original
scale; orange portions show smaller excursions along one such
alternative, and green portions show still smaller excursions.
Panels (b) and (c) enlarge the indicated regions of the same drawing.
Intersecting with the fixed row $m$ reveals groups of hits within
groups of hits. To count all these hits, we must estimate the
probability cost of the alternatives within each cluster using the
length scale of that cluster. This is why the proof applies its
near-maximum estimate at every level of the hierarchy, rather than
only at the largest scale.}
\label{fig:blppclusters}
\end{figure}

\subsection{Dynamical Brownian last passage percolation}\label{sec:dynamics}

We recall the notation and model from \cite{B25}. For
$A,B\subseteq\RR$, write $B_A=A\times B$, and let
$\dint ab=[a,b]\cap\ZZ$. {Thus,} $\ZZ_{\RR}=\RR\times\ZZ$ and
$\slab ab=\RR\times([a,b]\cap\ZZ)$. For points
$p=(x_0,j_0)$ and $q=(x_1,j_1)$, we write $p\le q$ if both
$x_0\le x_1$ and $j_0\le j_1$.

We first define static BLPP. Let $(W_j)_{j\in\ZZ}$ be independent
two-sided standard Brownian motions, with $W_j(0)=0$. For
$p=(x_0,j_0)\le q=(x_1,j_1)$ in $\ZZ_{\RR}$, a staircase
$\xi:p\to q$ is specified by
\[
 x_0=\xi(j_0-1)\le\xi(j_0)\le\cdots\le\xi(j_1)=x_1.
\]
It contains the horizontal segment
$[\xi(j-1),\xi(j)]\times\{j\}$ on row $j$ and, for $j<j_1$, the
upward segment joining $(\xi(j),j)$ to $(\xi(j),j+1)$. We call
$\xi(j)$ the \emph{exit} from row $j$. In particular, the row
coordinate $j$ and the dynamical time introduced below are different
parameters. The weight of this staircase is
\[
 \wgt(\xi)=\sum_{j=j_0}^{j_1}
             \bigl(W_j(\xi(j))-W_j(\xi(j-1))\bigr).
\]
The passage time from $p$ to $q$ is defined by
\[
 T_p^q=\max_{\xi:p\to q}\wgt(\xi),
\]
where the maximum is over all staircases from $p$ to $q$. Any
staircase attaining this maximum is called a geodesic and is denoted
by $\Gamma_p^q$. On the probability-one event that all driving
Brownian paths are continuous, compactness gives a maximiser
simultaneously for every $p\le q$; see Section~\ref{sec:pathclasses}.
Further, for each deterministic endpoint pair $p\le q$, the geodesic
is almost surely unique; see \cite[Lemma B.1]{Ham19}.

We now discuss the discrete dynamics introduced in \cite{B25}.
Starting with the static
environment at time zero, split each Brownian motion into its
unit-interval increment processes:
\[
 X_{i,j}^0(u)=W_j(i+u)-W_j(i),\qquad
 (i,j)\in\ZZ^2,\quad 0\le u\le1.
\]
These are independent standard Brownian motions on $[0,1]$. Attach
an independent rate-one Poisson clock to each $(i,j)\in\ZZ^2$,
independently of the initial environment. Whenever this clock rings,
replace the entire path $X_{i,j}^t$ by a fresh standard Brownian
motion on $[0,1]$, independently of all previous samples and all
other blocks. Between rings, the block remains unchanged. This
defines a stationary process for $t\ge0$, which has a stationary
extension to $t\in\RR$. Figure~\ref{fig:brownianblocks} shows the
decomposition into independently refreshed blocks.

\begin{figure}[tbp]
\centering
\begin{subfigure}{.59\textwidth}
\centering
\adjustbox{valign=c}{\includegraphics[width=.92\linewidth]{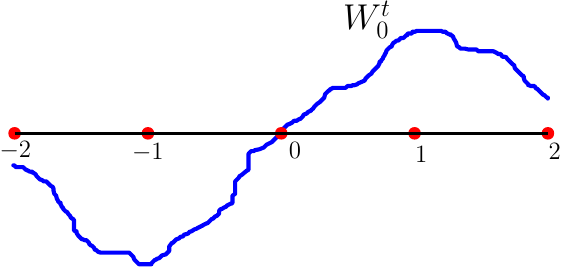}}
\end{subfigure}\hfill
\begin{subfigure}{.38\textwidth}
\centering
\adjustbox{valign=c}{\includegraphics[width=.94\linewidth]{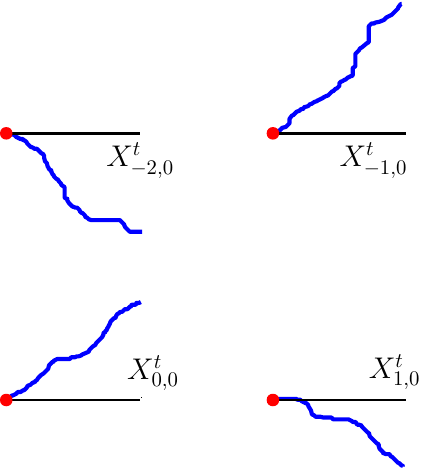}}
\end{subfigure}
\caption{The Brownian-block decomposition used to define the dynamics
(figure from \cite{B25}). On the left is a portion of $W_0^t$.
On the right are the corresponding unit-interval increment processes,
each translated to start at zero:
$X_{i,0}^t(u)=W_0^t(i+u)-W_0^t(i)$ for $0\le u\le1$.
Each block has its own rate-one clock; a ring replaces that entire
increment process by an independent Brownian motion.}
\label{fig:brownianblocks}
\end{figure}

At each dynamical time, the Brownian motions $W_j^t$ are reconstructed
from these increment processes by the conditions
\[
 W_j^t(0)=0,\qquad
 W_j^t(i+u)-W_j^t(i)=X_{i,j}^t(u),
 \quad i,j\in\ZZ,\quad 0\le u\le1.
\]
In particular, $W_j^0=W_j$, and for every deterministic $t$ the
collection $(W_j^t)_{j\in\ZZ}$ has the usual static Brownian law.
Replacing $W_j$ by $W_j^t$ in the preceding definitions gives the
dynamical weights $\wgt^t$, passage times $T_p^{q,t}$ and geodesics
$\Gamma_p^{q,t}$. At an update time, $t$ refers to the updated
configuration and $t^-$ to the preceding one. We refer to
\cite[Section 1.1]{B25} for further details of the construction.

\subsection{Main results}\label{sec:statements}

Our main result concerns exceptional times for infinite paths.
A bigeodesic is a doubly infinite
staircase every finite portion of which is a geodesic between its
endpoints. Entirely horizontal and entirely vertical paths are always
examples, and are called trivial. We are interested in the random set
\[
 \scT=\{t\in\RR:\text{there exists a non-trivial bigeodesic in }T^t\}.
\]
The following theorem, which is the central result of this paper, gives an upper bound on the size of this set.
Throughout, $\dim$ denotes Hausdorff dimension, with the convention
$\dim\varnothing=0$.

\begingroup
We state the result using Hausdorff measures with a general gauge.
For a continuous nondecreasing function $H:[0,\infty)\to[0,\infty)$
with $H(0)=0$, write
\begin{equation}\label{eq:gaugemeasure}
 \mathcal H^H(E)
 =\lim_{\rho\downarrow0}\inf\left\{
       \sum_i H(|J_i|): E\subseteq\bigcup_i J_i,
       \quad |J_i|\le\rho\right\},
\end{equation}
where the infimum is over countable families of intervals and
$|J_i|$ denotes interval length.
\begingroup
For $0<r\le e^{-e}$, write $L(r)=\log\log(1/r)$, so that
$L(r)\ge1$, and define
\begin{equation}\label{eq:hausdorffgauge}
 H(0)=0,\qquad H(r)=\exp\{-L(r)^2\log L(r)\}
                    \quad(0<r\le e^{-e}).
\end{equation}
Set $H(r)=1$ for $r>e^{-e}$. Since $L(e^{-e})=1$, this
extension is continuous; only the behaviour near zero matters.
\par\endgroup

\begingroup
\begin{samepage}
\begin{theorem}[Exceptional times]\label{thm:main}
Almost surely,\footnote{\label{fn:generalgauge}The same conclusion holds with $\log L$
replaced by any fixed positive, continuous, nondecreasing function
$\phi(L)\to\infty$ as $L\to\infty$.
In this case, set $H(r)=e^{-\phi(1)}$ for $r>e^{-e}$
so that the extension remains continuous.}
\[
 \mathcal H^H(\scT)=0.
\]
In particular, $\dim\scT=0$.
\end{theorem}
\end{samepage}
The gauge $H$ tends to zero more slowly than $r^s$ for
every $s>0$, so the measure assertion gives finer information
than Hausdorff dimension zero alone.
\par\endgroup
\par\endgroup
This proves Conjecture~4 of \cite{BE25}, also stated there as
Question~19, for the discrete dynamical BLPP model.
For a prescribed direction the conclusion is stronger. We call a
non-trivial bigeodesic $\theta$-directed if it is unbounded in both
row directions and $x_k/j_k\to\theta$ for every sequence $(x_k,j_k)$
on the path with $|j_k|\to\infty$.
For each deterministic $\theta\in(0,\infty)$, let
\[
 \scT^\theta
 =\{t\in\RR:T^t\text{ admits a non-trivial $\theta$-directed bigeodesic}\}.
\]
\begin{theorem}[No exceptional times in a fixed direction]
\label{thm:fixeddirection}
For every fixed deterministic $\theta\in(0,\infty)$,
$\PP(\scT^\theta=\varnothing)=1$.
\end{theorem}
Note that $\theta$ in the above is fixed. In particular,
the theorem allows the possibility of exceptional times whose
bigeodesic directions are random. Its proof, in
Section~\ref{sec:fixeddirection}, bounds the probability that a
geodesic in the prescribed direction visits a given unit cell during
a short time interval.

The proof of Theorem~\ref{thm:main} proceeds through estimates for
finite geodesics. We control the size of the region they visit over a
short dynamical time interval, uniformly as their endpoints vary in
suitable windows. To state these estimates, we recall the notion
of a hitset introduced in \cite[Section 1.2]{B25}.

The hitset records the trace of a family of geodesics at unit
horizontal resolution. A \emph{cell} is the segment
$[i,i+1]\times\{m\}$, indexed by $(i,m)\in\ZZ^2$. The Brownian
\emph{block} carried by this cell is the increment process $X_{i,m}^t$
from Section~\ref{sec:dynamics}; the geometric cell and its random
weight process thus have the same index. For a set
$A\subseteq\RR^2$ {we define its coarse-grained set by}
\[
 \coarse(A)=\{(i,m)\in\ZZ^2:
                   ([i,i+1]\times\{m\})\cap A\ne\varnothing\}.
\]
For endpoint sets $U,V\subseteq\ZZ_{\RR}$ and a time interval $J$, the
\emph{hitset} records all cells visited by the relevant geodesics at
any time in $J$. In a spatial region $A$, it is defined by
\begin{equation}\label{eq:hitdefinition}
 \hitset_U^{V,J}(A)=
 \bigcup_{t\in J}\ \bigcup_{p\in U,\,q\in V,\,p\le q}\
 \bigcup_{\Gamma\text{ geodesic from }p\text{ to }q\text{ in }T^t}
 \coarse(A\cap\Gamma).
\end{equation}
The union includes every geodesic when an endpoint pair has more
than one maximiser. For singleton endpoint sets we suppress the braces,
and we omit $A$ when
$A=\RR^2$. For bounded endpoint sets all paths and hitsets lie in a
deterministic rectangle. Note that a rectangle with side lengths $O(n)$ has
$O(n^2)$ coarse cells.

\begingroup
Proving Theorem~\ref{thm:main} requires accurate bounds on hitsets
when the endpoints vary in regions of transverse width $n^{2/3}$
at distance of order $n$ from each other.
\par\endgroup
\par\begingroup
Write $\0=(0,0)$ and $\bn=(n,n)$. We use the endpoint regions
\begin{equation}\label{eq:introregions}
 \mathscr R_n^\pm
 =\{(x,j)\in\ZZ_{\RR}: |j\mp n|\le n/32,\quad |x-j|\le n^{2/3}\}.
\end{equation}
These are rectangles in the coordinates $(x-j,j)$, centered at
$\pm\bn$, with transverse width $2n^{2/3}$ and longitudinal length
$n/16$. Their separation is of order $n$.
\par\endgroup

\begingroup
\begin{theorem}[Regional hitset bounds]\label{thm:regional}
Write
\[
 X_n=\left|\hitset_{{\mathscr R_n^-}}^{{\mathscr R_n^+},[0,n^{-1/3}]}
                         (\slab{-n/2}{n/2})\right|.
\]
{There are constants $A,C,c>0$, $\varepsilon_0\in(0,1)$
and an integer $n_0\ge1$ such that, for every}
integer $n\ge n_0$ and every
\[
 A\frac{(\log\log n)^2}{\log n}\le\varepsilon\le{\varepsilon_0},
\]
we have
\begin{equation}\label{eq:segmenthitsetprob}
 \PP(X_n>n^{1+\varepsilon})
 \le C\exp\!\left\{-c n^{\,c\varepsilon/\log(e/\varepsilon)}\right\}.
\end{equation}
The constants are independent of $n$ and $\varepsilon$.
Moreover, for every $n\ge n_0$,
\begin{equation}\label{eq:segmenthitset}
 \EE X_n\le n\exp\{C(\log\log n)^2\}.
\end{equation}
\end{theorem}

For each fixed {$0<\varepsilon\le\varepsilon_0$}, the theorem gives the
$n^{1+\varepsilon}$ bound with a stretched-exponential failure
probability in $n$. {Taking
$\varepsilon=D(\log\log n)^2/\log n$, with $D$ a sufficiently large
fixed constant, gives the smaller-threshold bound}
\begin{equation}\label{eq:quantitativehitsettail}
 \PP\bigl(X_n>n\exp\{C(\log\log n)^2\}\bigr)
 \le Ce^{-c(\log n)^2},
\end{equation}
\begingroup
{after increasing $C$ if necessary.}
All these bounds also hold with $X_n$ replaced by the point-to-point
hitset size
\[
 \left|\hitset_{-\bn}^{\bn,[0,n^{-1/3}]}
                    (\slab{-n/2}{n/2})\right|,
\]
since this hitset is contained in the regional one.

The proof first uses our passage-time stability estimate,
Theorem~\ref{thm:stability}, in the subcritical regime to cover the possible
geodesic exits by a small number {of} intervals. These intervals are still too
wide to count all their cells, so we examine shorter portions of
the geodesics and successively refine the cover inside each interval.
Section~\ref{sec:proofoutline} explains this construction.
Sections~\ref{sec:multiscale} and~\ref{sec:dimension} prove the
fixed-endpoint and regional estimates, respectively; the passage
from a slightly shorter time interval to $[0,n^{-1/3}]$ is given
in Lemma~\ref{lem:timesubdivision}.
\par\endgroup
\par\endgroup

\subsection{Passage-time stability}

Passage-time stability is a key input to the hitset argument:
over a short dynamical time interval, the optimized value changes
little even though its maximising path may change drastically.
We state the estimate here for the endpoints $\0=(0,0)$ and
$\bn=(n,n)$.

\begin{theorem}[Passage-time stability]\label{thm:stability}
There is an absolute constant $C$ such that, for every integer
$n\ge1$, every deterministic $t\ge0$, and every $u\ge2$,
\begin{equation}\label{eq:mainstabilitytail}
 \PP\left(\left|T_{\0}^{\bn,t}-T_{\0}^{\bn,0}\right|
       >C\left(\sqrt{ntu}+u\right)\right)\le2e^{-u}.
\end{equation}
\end{theorem}

Passage-time concentration on the $\sqrt{nt}$ scale was asked for in
the dynamical exponential LPP setting in \cite[Question 20]{BE25}.
For $nt\ge1$, Theorem~\ref{thm:stability} gives a discrete-BLPP
version of this estimate: setting $u=\alpha$ and increasing $C$ yields
\begin{equation}\label{eq:stabilityquestion20}
 \PP\left(\left|T_{\0}^{\bn,t}-T_{\0}^{\bn,0}\right|
                 >C\alpha\sqrt{nt}\right)\le2e^{-\alpha},
 \qquad \alpha\ge2.
\end{equation}
The Bernstein form in \eqref{eq:mainstabilitytail} also retains
Gaussian tails for moderate deviations. Indeed, setting
$u=\alpha^2$ yields
\begin{equation}\label{eq:stabilitymoderate}
 \PP\left(\left|T_{\0}^{\bn,t}-T_{\0}^{\bn,0}\right|
                 >C\alpha\sqrt{nt}\right)\le2e^{-\alpha^2},
 \qquad 2\le\alpha\le\sqrt{nt},
\end{equation}
after increasing $C$ if necessary.
When $nt\le1$, the theorem instead gives exponential tails for the
absolute change on a unit scale {as opposed to the $\sqrt{nt}$ scale}.\footnote{The exponential tail bound
\eqref{eq:stabilityquestion20} cannot hold uniformly as $nt\to0$
for discrete resampling: already for $n=1$, a full block refresh can cause an
order-one change with probability of order $t$. The additive term
in \eqref{eq:mainstabilitytail} accommodates these rare jumps.}

The scale $\sqrt{nt}$ is smaller than the static fluctuation scale
$n^{1/3}$ when $t\ll n^{-1/3}$. For example, setting $u=n^r$ in
\eqref{eq:mainstabilitytail}, for a fixed small $r>0$, makes the
failure probability at most $2e^{-n^r}$. This allows a union bound
over polynomially many choices of passage values. Section~\ref{sec:concentration} proves the
estimate by rotating the Brownian increments in the refreshed blocks.
Proposition~\ref{prop:generalstability} gives the version for arbitrary
deterministic endpoints and two times, together with moment bounds;
Corollary~\ref{cor:increments} treats restricted passage times.
The additive $+u$ in \eqref{eq:mainstabilitytail} is also used in the multiscale argument when the
product of the local length and the dynamical time is less than one.

\par\begingroup
The same rotation argument also applies to the Ornstein--Uhlenbeck
(OU) dynamics considered in \cite{GH24}. Let
$(B_j^r)_{j\in\ZZ,r\in\RR}$ be the stationary OU Brownian
environment: each $B_j^r$ is a two-sided standard Brownian motion,
the correlation of the same Brownian increment at times $r,r'$
is $e^{-|r-r'|}$, and different rows evolve independently.

{In the following theorem and its proof,
superscripts on passage values refer to this OU dynamics.}

\par\begingroup
\begin{theorem}[Passage-time stability under OU dynamics]\label{thm:oustability}
There is an absolute constant $C$ such that, for every integer
$n\ge1$, $t\ge0$ and $k\ge2$,
\begin{equation}\label{eq:oumainmoment}
 \left\|T_{\0}^{\bn,t}-T_{\0}^{\bn,0}\right\|_k
 \le C\sqrt{kn(1-e^{-t})}\le C\sqrt{knt}.
\end{equation}
In particular, for every $u\ge2$,
\begin{equation}\label{eq:oumaintail}
 \PP\left(\left|T_{\0}^{\bn,t}-T_{\0}^{\bn,0}\right|
          >C\sqrt{n(1-e^{-t})u}\right)
 \le 2e^{-u}.
\end{equation}
\end{theorem}

The theorem gives Gaussian tails on the $\sqrt{nt}$ scale even
when $nt<1$, without the additive term required for discrete
resampling. A proof sketch is given after
Corollary~\ref{cor:increments}. 
{Translation and Brownian scaling give the corresponding
bound for general deterministic endpoints. These estimates are used
in the companion paper~\cite{Dynamics26} to construct a non-trivial
dynamics on the directed landscape.}
\par\endgroup

For the hitset argument and the resulting bounds on exceptional
times, discrete updates provide a further advantage:
Lemma~\ref{lem:clocks} reduces uniformity over a time interval to a
union over finitely many environment configurations. We use this
in the regularity and \textit{Poisson capture} arguments.
OU dynamics visit infinitely many configurations, so extending
the hitset theorem to this setting would also require
uniform-in-time geometric estimates replacing these finite unions.

\par\endgroup

\subsection{The time scale and the dimension bound}

It is useful to compare Theorem~\ref{thm:regional} with the hitset
estimate obtained by counting switches in \cite{B25}. Because the
dynamics have discrete updates, for a fixed endpoint pair the
geodesic can change only when an update occurs. The switch count
defined in \cite[Equation (7)]{B25} adds, over these updates, the number
of cells entered by the new geodesic that were absent from the
preceding one. Every cell in the hitset is either visited at time
zero or accounted for by such an entry; revisits can therefore make
the switch count larger than the number of distinct new cells.
The estimate in \cite[Theorem 1]{B25} bounds this count in the bulk
by $n^{5/3+o(1)}h$ in expectation. \begingroup
\par\begingroup
Poisson capture transfers the fixed-endpoint bound to endpoint
regions. Taking $\gamma=1/2$ in \cite[Proposition 35]{B25}, the
endpoint parallelograms there contain $\mathscr R_n^-$ and
$\mathscr R_n^+$, respectively, for all sufficiently large $n$.
Consequently, that proposition gives
\[
 \EE\left|\hitset_{\mathscr R_n^-}^{\mathscr R_n^+,[0,h]}
                  \bigl(\slab{-n/2}{n/2}\bigr)\right|
 \le n^{1+o(1)}+n^{5/3+o(1)}h.
\]
\par\endgroup
Here the $o(1)$ notation means that each exponent loss may be
replaced by any fixed positive constant for all sufficiently large $n$.
\par\endgroup

For one direction window, spatial averaging divides by the volume
$n^{5/3}$ of a central tube. Covering a compact range of directions then
costs another factor $n^{1/3}$. The resulting probability that one of these geodesics visits a
given unit cell during $[0,h]$ is bounded by
{
\[
 n^{-1/3+o(1)}+n^{1/3+o(1)}h.
\]}
{For $h=n^{-2/3}$, the spatial averaging calculation appears in
the proof of \cite[Lemma 41]{B25}, and the union bound over directions
appears in \cite[proof of Proposition 44, Equation (154)]{B25}.}
The powers of $n$ in its two terms balance at $h=n^{-2/3}$,
giving the probability bound $h^{1/2+o(1)}$. This accounts for the dimension $1/2$ obtained from
that argument.

Our hitset estimates reach the longer time scale $n^{-1/3}$,
the predicted critical scale for natural LPP dynamics. For BLPP
under Ornstein--Uhlenbeck dynamics, this scale is identified, up to
subpolynomial corrections, in \cite{GH24}.
\begingroup
\begingroup
Our proof bounds the expected hitset size by
$n\exp\{C(\log\log n)^2\}$ over $[0,n^{-1/3}]$, also for the
endpoint families used in the exceptional-time argument.
The same spatial averaging and union over directions bound the
probability that one of these geodesics visits a given unit cell
during this interval by
$n^{-1/3}\exp\{C(\log\log n)^2\}$.
Divide $[0,1]$ into intervals of length $n^{-1/3}$ and retain
those during which such a visit occurs. Their expected number is
at most $\exp\{C(\log\log n)^2\}$.
Section~\ref{sec:gaugeproof} uses these intervals to construct
the Hausdorff cover proving Theorem~\ref{thm:main}.
\par\endgroup
\par\endgroup

\paragraph{\textbf{Acknowledgements.}}
The author thanks Riddhipratim Basu and Dor Elboim for the discussions and acknowledges the use of GPT-6 Astra at the level of a coauthor in developing the arguments and preparing this manuscript.

\paragraph{\textbf{Outline of the paper.}}
Section~\ref{sec:model} collects the geometric and probabilistic
preliminaries. Section~\ref{sec:proofoutline} gives the proof outline,
with particular attention to the need for more than one spatial scale.
The passage-time estimate, {cover of dynamic exits}, and local \textit{Poisson capture}
argument occupy Sections~\ref{sec:concentration}--\ref{sec:capture}.
{Section~\ref{sec:multiscale} combines
these estimates in one accelerated recursion with an adjustable
allowance. Section~\ref{sec:dimension} proves the regional theorem
and completes the exceptional-time arguments.}

\section{Preliminaries}\label{sec:model}

We now give the basic definitions and recall the estimates from
the literature that will be used in the proof.

\subsection{Coordinates, subpaths, and characteristic lines}

For endpoints $p=(x_0,j_0)\le q=(x_1,j_1)$ with $j_0<j_1$, we use
the slope convention
\[
 \lambda=\frac{x_1-x_0}{j_1-j_0}.
\]
{Thus,} slope means horizontal displacement per unit increase in row
number. The characteristic location on row $j$ is
$x_0+\lambda(j-j_0)$, and transverse distances are measured horizontally
from this line. Throughout, slopes lie in a fixed compact subset of
$(0,\infty)$. We shall informally call a geodesic a path at scale $N$ when its
horizontal and vertical displacements are comparable to $N$.
The corresponding KPZ transverse fluctuation scale is $N^{2/3}$. We use $n$ for the original scale and $N$ for a shorter
scale within the same path.

For $\beta\in(0,1/2)$, define the bulk slab
\begin{equation}\label{eq:bulk}
 \operatorname{Bulk}_\beta(p,q)
 =\slab{j_0+\beta(j_1-j_0)}{j_1-\beta(j_1-j_0)}.
\end{equation}
Its rows stay a fixed positive fraction of the vertical separation
away from both endpoints.

The notation $\Gamma_p^{q,t}(j)$ refers to the right endpoint of the
horizontal portion of the geodesic $\Gamma_p^{q,t}$ on row $j$.
That portion is
\[
 [\Gamma_p^{q,t}(j-1),\Gamma_p^{q,t}(j)]\times\{j\}.
\]

{We will repeatedly use the optimality of portions of a geodesic.
Fix an environment time $t$. If $a,b\in\Gamma_p^{q,t}\cap\ZZ_{\RR}$
occur in this order along $\Gamma_p^{q,t}$, then its portion from $a$
to $b$ maximizes weight among staircases with those endpoints.
Indeed, replacing that portion by a staircase of larger weight would
increase the weight of the path from $p$ to $q$.}
We also use the planar ordering of geodesics: paths
between ordered endpoints can be bounded by the corresponding extreme
geodesics. This permits confinement of an entire family between
endpoint windows by confinement of finitely many bounding paths.

\subsection{Staircase coordinates and compact path classes}\label{sec:pathclasses}

We identify a staircase from $p=(x_-,j_-)$ to $q=(x_+,j_+)$ with its
vector of exit coordinates: $z_j$ is the horizontal coordinate of
its upward step from row $j$ to row $j+1$. {Thus,} the space of all such staircases is
\begin{equation}\label{eq:staircasesimplex}
 \mathfrak S(p,q)=\{(z_{j_-},\ldots,z_{j_+-1}):
                 x_-\le z_{j_-}\le\cdots\le z_{j_+-1}\le x_+\}.
\end{equation}
We set $z_{j_--1}=x_-$ and $z_{j_+}=x_+$. If $j_-=j_+$, the space
has one element, corresponding to the horizontal segment from $p$ to
$q$. The set $\mathfrak S(p,q)$ is a compact simplex in the Euclidean
space of exit coordinates. A \emph{compact class of staircases}
means a nonempty compact subset of this space. For example, the
constraint $z_m\in I$, where $I$ is a closed interval, defines a
closed, and hence compact, subset whenever it is nonempty. These
coordinates will also be used in Section~\ref{sec:concentration} to
approximate compact classes by finite collections of staircases.
For a fixed Brownian environment, the weight of a staircase with
exit coordinates $(z_j)$ is
\[
 \sum_{j=j_-}^{j_+}\bigl(W_j(z_j)-W_j(z_{j-1})\bigr).
\]
The continuity of the Brownian paths makes this weight continuous
in the exit coordinates. By compactness, it therefore attains its
maximum on every nonempty compact staircase class.

\subsection{Routed weight profiles and near maximisers}\label{sec:profiles}

For endpoints $u\le(x,m)\le(x,m+1)\le v$, we define the
\emph{routed weight profile} by
\begin{equation}\label{eq:Zbullet}
 Z_u^{v,\bullet,t}(x,m)=T_u^{(x,m),t}+T_{(x,m+1)}^{v,t}.
\end{equation}
This is the maximum weight of a staircase whose upward step from
row $m$ to row $m+1$ occurs at $x$. {Therefore,}
\begin{equation}\label{eq:maxprofile}
 T_u^{v,t}=\max_x Z_u^{v,\bullet,t}(x,m),\qquad
 Z_u^{v,\bullet,t}(\Gamma_u^{v,t}(m),m)=T_u^{v,t}.
\end{equation}
At a fixed time, the two summands in \eqref{eq:Zbullet} are
independent because they use disjoint sets of Brownian rows. On a
compact interval strictly inside the admissible horizontal range,
each profile, after subtracting its value at the left endpoint, has
law absolutely continuous with respect to Brownian motion; see the
Brownian Gibbs {\cite{CH14}} description of BLPP profiles in
\cite[Sections 3.1--3.2]{CHH23}. Their sum therefore has local
Brownian behaviour, with variance rate two. The quantitative
comparison for the routed profile is given in
\cite[Theorem 1.2]{GH23}. Stronger comparisons for
the Airy line ensemble and BLPP profiles are available in
\cite{Dau24} and \cite[Appendix 2]{B25}; see
footnote~\ref{fn:strongercomparison} below for our choice of input.
 We derive the resulting
near-maximiser count in Lemma~\ref{lem:packing}.

If $I$ is a nonempty closed interval in the admissible horizontal
range, we shall write
\begin{equation}\label{eq:FIdefinition}
 F_I^t=\max_{x\in I}Z_u^{v,\bullet,t}(x,m).
\end{equation}
This is the best passage time among staircases whose exit from row $m$
belongs to $I$.

We will need to count separated locations at which a profile is close
to its maximum. Since a Brownian fluctuation of height $\alpha$ has
horizontal scale $\alpha^2$, it is natural to impose separation at
that scale. For a continuous function $f$ on a compact interval $J$
and $\alpha>0$, define
\begin{equation}\label{eq:NMdefinition}
 \nearmax^\alpha(f;J)=\max\left\{|S|:S\subseteq J,\quad
 \begin{array}{l}
 |x-y|\ge\alpha^2\quad(x\ne y\text{ in }S),\\
 f(x)\ge\max_Jf-\alpha\quad(x\in S)
 \end{array}\right\}.
\end{equation}
This is the number of near touches considered in
\cite[Section 1.2.3]{CHH23}, with separation parameter $\alpha^2$;
we use the notation of {\cite[Section 9.3, immediately before Lemma 80]{B25}}.
A maximal separated set covers the near-maximum set by intervals of
radius $\alpha^2$ about its points, since any uncovered near maximiser could
otherwise be added to the set. {In particular, a bound on
$\nearmax^\alpha$ gives a bound on the number of intervals needed
for a cover.}

\subsection{Probability conventions and clock configurations}

For a random variable $Y$ and $k\ge1$, we write
$\norm Y_k=(\EE|Y|^k)^{1/k}$. Constants denoted by $C,c$ may change
from line to line. The notation $a\asymp b$ means that their ratio
lies between two positive fixed constants, and $a\lesssim b$ means
$a\le Cb$ for such a constant. Statements written as $n^{o(1)}$ in
the discussion mean that the required bound is available with $n^\varepsilon$
for every fixed $\varepsilon>0$, for sufficiently large $n$.

An event holds with \emph{stretched-exponentially high probability}
in $n$ if its failure probability is at most $Ce^{-cn^\theta}$ for
some fixed $C,c>0$ and $\theta\in(0,1)$. Such an error is, in
particular, $O_A(n^{-A})$ for every fixed $A>0$.
Constants and decay exponents may depend on fixed slope and bulk
ranges and on small exponent parameters, but are uniform in
deterministic endpoints within those ranges. We allow the lower
threshold on $n$ to depend on these fixed parameters.

A useful feature of discrete resampling is that only finitely many
environments are seen in a bounded rectangle during a bounded time
interval. Although those environments are dependent, each retains the
static marginal law even after the clock times have been specified.
The following observation makes precise the union bound we will often use.

\begin{samepage}
\begin{lemma}[Static marginals conditional on the clocks]\label{lem:clocks}
Fix a deterministic rectangle meeting at most $C n^2$ update blocks.
Conditional on all its clock rings in $[0,1]$, the environment in each
successive configuration has the product Brownian marginal distribution.
Except on an event of probability $e^{-c n^2}$, the number of configurations
is at most $C' n^2$. {Thus,} a static event determined by this rectangle and
having failure probability $p_n$ holds in every configuration except with
probability at most $C'n^2p_n+e^{-c n^2}$.
\end{lemma}
\end{samepage}
\begin{proof}
Let $\cB$ be the set of update blocks meeting the rectangle, and let
$K_{\cB}$ be the number of their clock rings in $[0,1]$. Write
$0=\tau_0<\tau_1<\cdots<\tau_{K_{\cB}}\le1$ for time zero and the
ordered rings. Let $\mathscr C_{\cB}$ be the {$\sigma$-algebra} generated by
these clocks, including the identity of the block at each ring.
For each block $b$, write $(X_{b,k})_{k\ge0}$ for its initial and
successive replacement samples. Conditional on $\mathscr C_{\cB}$,
the configuration at $\tau_i$ selects one deterministic index from each
sequence $(X_{b,k})_{k\ge0}$. Since the samples are independent of the
clocks and independent across blocks, that configuration has the product
Brownian law.

Let $\cA$ be a static event determined by the rectangle, with
$\PP(\cA^c)=p_n$, and let $\cA^{\tau_i}$ denote its occurrence in the
configuration at $\tau_i$. Conditional on $\mathscr C_{\cB}$,
\[
 \PP\left(\bigcup_{i=0}^{K_{\cB}}(\cA^{\tau_i})^c
          \,\middle|\,\mathscr C_{\cB}\right)
 \le (K_{\cB}+1)p_n.
\]
Also $K_{\cB}$ is Poisson with mean $|\cB|\le Cn^2$. Choose
$C'>2C+2$ and set $\cE^{\mathrm{clock}}_{\cB}
=\{K_{\cB}+1\le C'n^2\}$. The Poisson exponential moment gives
$\PP((\cE^{\mathrm{clock}}_{\cB})^c)\le e^{-cn^2}$, after increasing
$C'$ if necessary. Integrating the preceding conditional bound on
$\cE^{\mathrm{clock}}_{\cB}$ proves
\[
 \PP\left(\bigcup_{i=0}^{K_{\cB}}(\cA^{\tau_i})^c\right)
 \le C'n^2p_n+e^{-cn^2}.
\]
Every configuration in $[0,1]$ is one of those indexed by $\tau_i$.
This proves the assertion, including its all-time formulation.
\end{proof}

For deterministic endpoint pairs, {uniqueness of geodesics in the static model and}
Lemma~\ref{lem:clocks} give uniqueness at all times in a bounded interval,
almost surely. The same holds for a finite or countable list of pairs
sampled independently of the dynamics, by conditioning on that list.

\begingroup
\subsection{Geometric regularity of one geodesic}\label{sec:regularity}

We use regularity to locate the cut endpoints of a geodesic and to
convert covers of its exits into covers of all visited cells. It is
useful to retain an adjustable allowance $\ell$ in these estimates.
Later, taking $\ell=\log n$ or a small power of $n$ will give the
two regimes of the hitset theorem from the same proof.

\begin{lemma}[Regularity with an adjustable allowance]\label{lem:regularity}
Fix $0<a<b<\infty$ and a compact interval $K\subset(0,\infty)$.
There exist $C_{\mathrm{reg}},C,c>0$ and $N_1<\infty$ such that
the following holds for every integer $N\ge N_1$ and
$\log N\le\ell\le N^{1/20}$. Let
$p=(x_0,j_0)\le q=(x_1,j_1)$ be deterministic with
\[
 aN\le j_1-j_0\le bN,\qquad
 \lambda=\frac{x_1-x_0}{j_1-j_0}\in K,
\]
and put
\begin{equation}\label{eq:regularitymargin}
 P=C_{\mathrm{reg}}\ell^2.
\end{equation}
Let $\cE^{\mathrm{reg}}_{p,q,N}(\ell)$ be the event on which,
simultaneously for all $t\in[0,1]$,
\begin{align}
 |\Gamma_p^{q,t}(j)-x_0-\lambda(j-j_0)|
       &\le PN^{2/3},\label{eq:rootTF}\\
 |\Gamma_p^{q,t}(j+M)-\Gamma_p^{q,t}(j)-\lambda M|
       &\le PM^{2/3},\label{eq:modulus}\\
 \Gamma_p^{q,t}(j)-\Gamma_p^{q,t}(j-1)
       &\le P.\label{eq:rowlength}
\end{align}
In the first and third inequalities $j_0\le j\le j_1$, with
$\Gamma_p^{q,t}(j_0-1)=x_0$; in the second, $j,M$ are integers
with $M\ge1$ and $j_0\le j<j+M\le j_1$. Then
\begin{equation}\label{eq:regularityprob}
 \PP\bigl((\cE^{\mathrm{reg}}_{p,q,N}(\ell))^c\bigr)
 \le Ce^{-c\ell^3}.
\end{equation}
All constants depend only on $a,b,K$ and are uniform in
$N,\ell,p,q$ in these ranges.
\end{lemma}

\par\begingroup
For example, suppose the exit on row $m$ lies in an interval of
width $N^{2/3}$ centered at $c$. The increment bound
\eqref{eq:modulus} places the exits on rows $m-N,m+N$, when these
rows lie between the endpoints, within distance
$(P+1)N^{2/3}$ of $c-\lambda N,c+\lambda N$, respectively.
These are the windows in which we will use Poisson capture.
Once the exits have been covered, the row-length bound
\eqref{eq:rowlength} lets us count all the cells visited by the
horizontal portions of the geodesic.
\par\endgroup

\begin{proof}
Write $D=j_1-j_0$. At a static time, translation and horizontal
Brownian scaling by $\lambda^{-1}$ reduce the pair to
$(0,0),(D,D)$. Apply \cite[Theorem 1.4(1), (2) and
Corollary 1.5]{GH23} with regularity parameter $r=\ell$.
The restriction $r\le D^{1/10}$ holds uniformly for large $N$
since $D\asymp N$ and $\ell\le N^{1/20}$.

\begingroup
{After this horizontal scaling, write $\gamma(i)$ for the exit
on row $i$, with $\gamma(-1)=0$. The cited estimates imply that
the following three bounds hold simultaneously outside an event of
probability at most $C(1+\log D)e^{-c\ell^3}$:}
\begin{align*}
 \max_{0\le i\le D}|\gamma(i)-i|
    &\le C\ell D^{2/3},\\
 |\gamma(i+M)-\gamma(i)-M|
    &\le C\ell M^{2/3}\bigl(\log(1+D/M)\bigr)^{1/3},\\
 \max_{0\le i\le D}\bigl(\gamma(i)-\gamma(i-1)\bigr)
    &\le C\ell(\log D)^{1/3}.
\end{align*}
{Here the second bound holds for all integers $0\le i<i+M\le D$.
To obtain the stated simultaneous probability bound, apply
Corollary 1.5 to the first inequality, with failure probability
$Ce^{-c\ell^3}$. The second follows from Theorem 1.4(1)}
on each dyadic range of $M/D$ to which it applies, with failure
$Ce^{-c\ell^3 k}$ on the $k$th range. Theorem 1.4(2) therein covers the
remaining bounded separations and horizontal row portions, with
failure $CD^{-c\ell^3}$; taking two points on the same row also
gives the third bound, including the first row. For $M$ comparable
to $D$, the first bound at both endpoints gives the second after
increasing $C$.

Since $\log D\le C\log N\le C\ell$, all the displayed margins
are bounded by $C'\ell^2$ times the indicated powers of $D$ or $M$.
Undoing the horizontal scaling introduced at the beginning of this proof multiplies them by $\lambda$.
As $D\asymp N$ and $\lambda\in K$, increasing
$C_{\mathrm{reg}}$ gives \eqref{eq:rootTF}, \eqref{eq:modulus}
and \eqref{eq:rowlength}, respectively. Summing the dyadic failure bounds shows that, at a fixed
time, the probability that any of \eqref{eq:rootTF},
\eqref{eq:modulus} or \eqref{eq:rowlength} fails is at most
$C(1+\log N)e^{-c\ell^3}$.
\par\endgroup

{Every staircase from $p$ to $q$ lies in
$[x_0,x_1]\times\dint{j_0}{j_1}$. This rectangle meets at most
$CN^2$ unit horizontal cells, each carrying one Brownian increment
process and its rate-one refresh clock.}
Lemma~\ref{lem:clocks} gives all-time failure at most
\[
 CN^2(1+\log N)e^{-c\ell^3}+e^{-cN^2}
 \le C'e^{-c'\ell^3}.
\]
The last inequality is uniform over the allowed $\ell$:
$\ell\ge\log N$ absorbs the polynomial factor and
$\ell\le N^{1/20}$ handles the second term.
\end{proof}
\par\endgroup

\subsection{Bigeodesics in static BLPP}
\label{sec:bigeodstatblpp}
For static BLPP, {\cite[Theorem 3.1(v), (vi)]{SS23}} proves nonexistence of non-trivial bigeodesics along any fixed direction.
Furthermore, \cite[Theorem 5.4]{RS26} rules out all non-trivial
locally leftmost or locally rightmost bigeodesics, meaning that their
finite portions are always the corresponding extremal geodesics
between their endpoints. Unrestricted static nonexistence
follows from {the dynamical BLPP paper} \cite{B25}: its bound $\dim\scT\le1/2$ implies
$\operatorname{Leb}(\scT\cap[0,1])=0$ almost surely, and stationarity
along with an application of Fubini's theorem yields
$\PP(0\in\scT)=\EE[\operatorname{Leb}(\scT\cap[0,1])]=0$.

Apart from the estimates recorded in these preliminaries, the
remaining inputs concerning geodesics will be stated where they
are used. Section~\ref{sec:crossings} derives a static
near-maximiser estimate on the crossing windows at our shorter
scales. Section~\ref{sec:capture} records the general-slope
version of the Poisson capture result in \cite[Proposition 30]{B25}.
The input about infinite geodesics used in
Section~\ref{sec:dimension} is \cite[Proposition 13]{B25}: every
non-trivial bigeodesic has an asymptotic direction and stays within
a suitable neighbourhood of the corresponding line.
\section{Outline of the proof}\label{sec:proofoutline}

The passage from finite geodesics to exceptional times at which
bigeodesics exist is governed by how many cells those finite geodesics
can visit during a short time interval. We begin with the estimate that makes this
connection, and then explain how passage-time stability and a hierarchy
of spatial covers lead to it.

\paragraph{\textbf{The target estimate and the reduction to one row.}}
\begingroup
We focus on the quantitative expectation bound in
Theorem~\ref{thm:regional}:
\[
 \EE\left|\hitset_{{\mathscr R_n^-}}^{{\mathscr R_n^+},[0,n^{-1/3}]}
                \bigl(\slab{-n/2}{n/2}\bigr)\right|
 \le n\exp\{C(\log\log n)^2\}.
\]
Since the slab contains $O(n)$ rows, it is enough to prove
\[
 \EE\left|\hitset_{{\mathscr R_n^-}}^{{\mathscr R_n^+},[0,n^{-1/3}]}
                 (\RR\times\{m\})\right|
 \le\exp\{C(\log\log n)^2\}
\]
uniformly over integers $|m|\le n/2$. We may therefore concentrate
on the locations visited in a single row.

For this outline we take $\ell=\log n$ and work first over
$[0,h]$, where $h=n^{-1/3}\ell^{-4}$. Subdividing the critical
interval $[0,n^{-1/3}]$ into $\lceil\ell^4\rceil$ pieces then
recovers the stated horizon; see Lemma~\ref{lem:timesubdivision}.
This costs only a fixed power of $\log n$, which is absorbed in
the bound above by increasing $C$. The proofs retain $\ell$ as
an adjustable parameter to obtain the full tail estimate in
Theorem~\ref{thm:regional} from the same construction.
For each fixed $n$, however, we choose $\ell$ once and keep it
unchanged throughout the proof, including at every smaller scale
$N$. Thus $\ell=\log n$ is not replaced by $\log N$ during the
recursion. We suppress this fixed allowance in notation such as
$A_N$ and $D_N$; their subscripts record the scale being refined.
\par\endgroup

\paragraph{\textbf{Replacing endpoint regions by a finite list of pairs.}}
There is still a continuum of possible endpoints in the two {regions} ${\mathscr R_n^-}$ and ${\mathscr R_n^+}$.
The \textit{Poisson capture} argument of \cite[Section 5]{B25}
represents the central portions of these geodesics using a much
smaller family with independently sampled endpoints.
More precisely, an independent Poisson cloud supplies a finite list
of endpoint pairs such that, on an event of high probability, the
central portion of every geodesic from ${\mathscr R_n^-}$ to ${\mathscr R_n^+}$, at every
time in $[0,h]$, agrees with a portion of a geodesic from this list
at the same time. {The number of sampled
pairs costs a fixed power of $\ell$.}

Conditional on the cloud, its endpoint pairs are deterministic and
the environment retains its original law. It therefore suffices to
establish
\begin{equation}\label{eq:outlinepointrow}
 \EE\left|\hitset_p^{q,[0,h]}(\RR\times\{m\})\right|
 \le\exp\{C(\log\log n)^2\}
\end{equation}
uniformly over deterministic endpoint pairs
$p,q$ at scale $n$ and bulk rows $m$, as specified in
Proposition~\ref{prop:generalpoint}.
For the rest of the outline, we fix such a pair $p,q$, a bulk
row $m$, and the time interval $[0,h]$. Write
\begin{equation}\label{eq:outlineexitset}
 E_m=\{\Gamma_p^{q,t}(m):0\le t\le h\}
\end{equation}
for the set of all exits of this evolving geodesic from row $m$.
We first cover $E_m$ by short intervals; at the end, a bound on the
length of each horizontal portion of the geodesic converts this exit
cover into a bound on all the cells visited in that row.

\paragraph{\textbf{Passage-time stability.}}
Theorem~\ref{thm:stability} controls the change of an optimised
passage time for diagonal endpoints. Its version for general
endpoints, Proposition~\ref{prop:generalstability}, gives, for the
deterministic pair under consideration,
\begin{equation}\label{eq:outlinestabilitytail}
 \PP\left(\left|T_p^{q,t}-T_p^{q,0}\right|
       >C\bigl(\sqrt{ntu}+u\bigr)\right)\le2e^{-u},
 \qquad t\ge0,\quad u\ge2.
\end{equation}
Corollary~\ref{cor:increments} gives the same bound for $F_I^t$, the
best passage time among staircases whose exit from row $m$ lies in a
deterministic unit interval $I$. The proof uses a rotation of the
Brownian increments in the refreshed blocks. These bounds, together
with the discrete update structure, also give simultaneous control
over the required intervals $I$ and all times in $[0,h]$; this is
established in Lemma~\ref{lem:timeuniform}.

\paragraph{\textbf{The first cover and its limitation.}}
To see why stability is useful, consider a unit interval $I$ containing
an exit of $\Gamma_p^{q,t}$ at some time $t\le h$. At that time,
$F_I^t=T_p^{q,t}${, with $F_I^t$ as in \eqref{eq:FIdefinition}}. If both passage values have changed by at most
$A$ from time zero, then
\begin{equation}\label{eq:introwitness}
 T_p^{q,0}-F_I^0\le 2A.
\end{equation}
This use of passage-time stability to produce near-optimal paths in a
static environment also appears in the work of Ganguly--Hammond
\cite[Section 3.3]{GH24}. We discuss the relation with their
argument in Remark~\ref{rem:ghstrategy} below.
{Thus,} every dynamically visited exit interval contains a near maximiser
of the static routed profile $Z_p^{q,\bullet,0}(\cdot,m)$. This gives
information about near-optimal paths in the time-zero environment;
the time-zero geodesic need not itself visit $I$.

The static near-maximum estimate in Section~\ref{sec:crossings} now
provides a spatial cover. It bounds the number of near maximisers
separated by order $A^2$, and hence covers the possible exits by a
small number of intervals of that length. We call these covering
intervals \emph{clusters}: each groups nearby candidate exits, but
{we do not assert that every point of the interval is ever visited
by $\Gamma_p^{q,t}$ as time evolves.}
\begingroup
At the largest scale the simultaneous stability bound uses
$A=C(\ell\sqrt{nh}+\ell^2)$. The factor $\ell$ allows the
estimate to hold for all the passage values needed in the argument.
Since $h=n^{-1/3}\ell^{-4}$ and $\ell=\log n$, {$A^2$,}
and hence the width of each covering interval, is of order
$n^{2/3}\ell^{-2}$, still much larger than a unit cell. 
\par\endgroup

This first cover leaves a substantial counting problem. Although
there are few clusters, each still contains many possible unit cells.
Counting every cell in even one cluster would cost
\begin{equation}\label{eq:onescalecost}
 {n^{2/3}\ell^{-2}}\quad\hbox{per row},\qquad
 {n^{5/3}\ell^{-2}}\quad\hbox{over $O(n)$ rows}.
\end{equation}
Our target is $n\exp\{C(\log\log n)^2\}$ visited cells in
total, so {this bound is far too large, recalling that $\ell=\log n$ in this outline}. We therefore need further
information about which cells inside each cluster can actually be
visited. Figure~\ref{fig:clusters} illustrates this obstruction.

\begin{figure}[tbp]
\centering
\begin{tikzpicture}[x=1cm,y=1cm]
  \fill[clusterblue!7] (0.55,2.3) rectangle (11.55,3.08);
  \draw[guide] (0.55,3.08)--(11.55,3.08);
  \draw[guide] (0.55,2.3)--(11.55,2.3);
  \node[anchor=west,fill=white,inner sep=2pt] at (8.65,3.25) {$\max Z^0$};
  \node[anchor=west,fill=white,inner sep=2pt] at (8.65,2.12) {$\max Z^0-A$};
  \draw[<->,witnessgold] (11.8,2.3)--node[right] {$A$}(11.8,3.08);
  \draw[->,black!65] (.45,.78)--(12.1,.78) node[right] {$x$};
  \draw[clusterblue,line width=1pt] plot coordinates {
    (.65,1.1)(1.0,1.45)(1.18,1.3)(1.38,1.85)(1.6,2.55)
    (1.74,2.43)(1.92,2.83)(2.05,2.65)(2.24,3.08)(2.4,2.77)
    (2.53,2.92)(2.71,2.58)(2.88,2.7)(3.08,2.31)(3.24,2.48)
    (3.42,2.01)(3.67,2.17)(3.94,1.65)(4.16,1.8)(4.45,1.32)
    (4.69,1.56)(4.9,1.37)(5.16,1.89)(5.4,1.66)(5.72,2.07)
    (5.91,1.9)(6.16,2.49)(6.31,2.36)(6.49,2.8)(6.65,2.61)
    (6.83,2.89)(7.0,2.67)(7.15,2.77)(7.3,2.4)(7.46,2.59)
    (7.65,2.04)(7.83,2.15)(8.07,1.62)(8.34,1.76)(8.64,1.39)
    (8.91,1.66)(9.17,1.3)(9.4,1.52)(9.64,1.09)(9.9,1.36)
    (10.15,1.15)(10.42,1.52)(10.69,1.27)(11.0,1.48)(11.5,1.04)};
  \node[anchor=west] at (.6,3.65) {\textbf{(a)} A static profile and its near-maximum band};
  \draw[guide] (1.4,.9)--(1.4,-.27);
  \draw[guide] (3.5,.9)--(3.5,-.27);
  \draw[guide] (5.98,.9)--(5.98,-.27);
  \draw[guide] (7.8,.9)--(7.8,-.27);
  \draw[cluster] (1.4,-.45) rectangle (3.5,-.12);
  \draw[cluster] (5.98,-.45) rectangle (7.8,-.12);
  \draw[<->] (1.4,-.75)--node[below] {$O(A^2)$}(3.5,-.75);
  \node[anchor=west] at (.6,-1.55) {\textbf{(b)} Few clusters, but many possible cells inside each};
  \node[anchor=west,align=left,font=\footnotesize] at (8.3,-.45)
    {Few intervals;\\many possible cells inside.};
\end{tikzpicture}
\caption{The first cover obtained from a static routed weight profile.
Near maximisers lie within $A$ of the maximum, and their
locations can be covered by a few intervals of length $O(A^2)$.
After unit enlargement, these intervals cover the dynamically visited
exits. Each interval still contains many candidate cells; the finer
cover must determine which of them can actually be visited.}
\label{fig:clusters}
\end{figure}
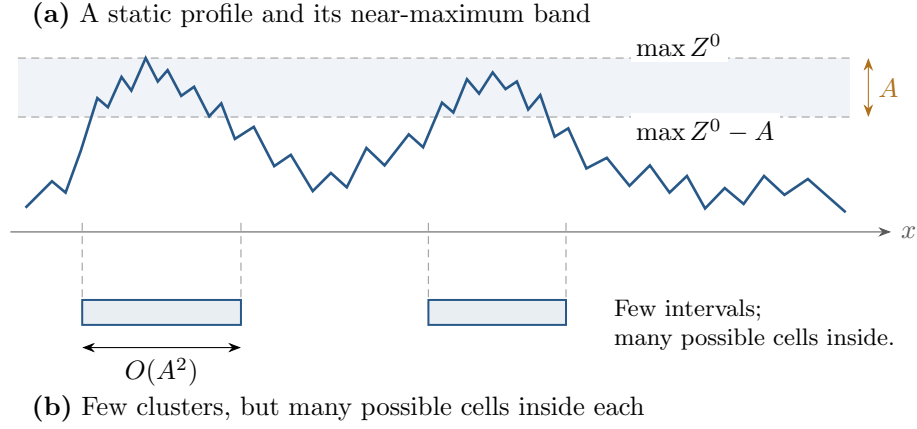

\paragraph{\textbf{A hierarchy of finer clusters.}}
We shall use a multiscale argument to obtain this information from shorter
geodesic segments. One can picture the evolving geodesic as making
different excursions on the original scale, with further, smaller
excursions along each of these paths. In the fixed row $m$, this
suggests clusters containing smaller clusters, and so on.

To examine a parent interval of width $N^{2/3}$, we look at the
portions of the paths between rows $m-N$ and $m+N$.
Figure~\ref{fig:cutting} shows such a portion between its two cut
endpoints, $a_t$ and $b_t$; it retains the original exit on row $m$.
These portions are themselves geodesics of length of order $N$.
Applying passage-time
stability and the static near-maximum estimate at this length gives,
over the same time interval $[0,h]$, a passage-time tolerance
$C(\ell\sqrt{Nh}+\ell^2)$ and covering intervals of length at most
\[
 D_N=C_{\mathrm{width}}(Nh\ell^2+\ell^4).
\]
Apart from the logarithmic factors and the final additive term,
the new width is $Nh$, with $N$ replacing the original length $n$.
Thus each
refinement uses the rarity estimate at the scale of the shorter
segment itself. Repeating the procedure reveals which parts of a
parent interval can be visited at finer scales.

\paragraph{\textbf{\textit{Poisson capture} for the random cut endpoints.}}
The preceding description hides a difficulty: the endpoints of a cut
segment depend on the evolving geodesic. Although the original
endpoints $p,q$ are deterministic, these new endpoints vary with the
environment and with time. We cannot simply regard them as a fixed
pair when applying the passage-time and near-maximum estimates.

\textit{Poisson capture} is used a second time to address this issue, now
{within the proof of the one-row point-to-point hitset estimate
\eqref{eq:outlinepointrow}.} Suppose that, at
time $t$, the exit coordinate $\Gamma_p^{q,t}(m)$ of the geodesic
from row $m$ lies in an interval $J$ of width $N^{2/3}$, centered at $c$.
Geodesic regularity places the two cut endpoints in deterministic
windows associated with $(N,c,m)$;
Figure~\ref{fig:windowsbasins} shows their placement.
We sample an independent cloud
$\cQ_{N;c,m}$ of endpoint pairs in slightly larger regions around
these windows. Its geodesics capture the central portions of the cut
segments: Figure~\ref{fig:cutting} shows a sampled pair whose geodesic
agrees with the cut segment around row $m$, preserving its exit.
The representative pair may change with the path and with time.

Conditional on the cloud alone, the sampled endpoint pairs are fixed
and the environment retains its original law. We apply the scale-$N$
passage-time and near-maximum estimates to every pair in the cloud,
obtaining a cover of all its exits during $[0,h]$. Taking the union
of these covers then covers $E_m\cap J$. {Thus,} the random cut
endpoints are handled by testing the whole independently sampled
list, without having to know which pair captures a given segment.
Section~\ref{sec:capture} proves this \textit{Poisson capture} statement, and
Figure~\ref{fig:cloudrefinement}(a) illustrates how the covers are combined.

\begingroup
\paragraph{\textbf{Preparing a finer cover for every possible interval.}}
\begingroup
The parent intervals will be selected using the environment.
To allow these random choices, we first prove estimates for every
interval that could be selected. We use a deterministic list of
scales, ending at {the cutoff $\ell^Q$ for a sufficiently large fixed $Q$}.
At scale $N$, a deterministic mesh of width $N^{2/3}$ supplies
the possible parent intervals. For every bulk row $m$ and every
such interval $J$, centered at $c$, we prepare its own windows
and independent cloud $\cQ_{N;c,m}$.

We construct one event $\cE_n^{\mathrm{ref}}$ on which the
following three properties hold simultaneously for all these
choices of $(N,m,J)$:
\begin{enumerate}
\item Whenever $\Gamma_p^{q,t}$ exits row $m$ in $J$, for any
$t\in[0,h]$, some pair in $\cQ_{N;c,m}$ has a geodesic that
agrees with its cut segment around row $m$, including the exit.
\item The cloud $\cQ_{N;c,m}$ contains at most a fixed power of
{$\ell$} endpoint pairs.
\item For each of these pairs, all its geodesic exits from row $m$
during $[0,h]$ admit a cover by at most a fixed power of {$\ell$}
intervals, each of length at most $D_N$.
\end{enumerate}
The first property lets us cover $E_m\cap J$ by taking the union
of the covers in the third. The second bounds the number of covers
being combined. Thus each parent has a finer cover with at most
a fixed power of {$\ell$} intervals. This is the construction
shown in Figure~\ref{fig:cloudrefinement}(a).

Lemma~\ref{lem:refinement} proves these properties by a union
bound over all possible boxes. There are polynomially many in $n$;
the local failure probabilities are small enough to absorb this
number because $\ell=\log n$. Stopping at a sufficiently large
power of {$\ell$ ensures that the estimates remain strong enough}
even at the smallest prepared scale. The distinction is essential:
every possible parent enters this probability bound, but only the
parents retained by the cover construction enter the count below.
Figure~\ref{fig:cloudrefinement}(b) illustrates these two steps.
\par\endgroup
\par\endgroup

\begin{figure}[tbp]
\centering
\begin{tikzpicture}[x=1cm,y=1cm,font=\footnotesize]
  \node[anchor=west,font=\small] at (0,9.15)
    {\textbf{(a)} One parent $J$: combine the covers from its sampled pairs};
  \node[anchor=north west,align=left,text=black,text width=4.5cm]
    at (.2,8.8) {Cloud $\cQ_{N;c,m}$:\\{a fixed power of $\ell$ pairs}};
  \node[anchor=north west,align=left,text=black,text width=7cm]
    at (5.0,8.8) {Each pair:\\{a fixed power of $\ell$ covering intervals}};
  \draw[rounded corners,draw=witnessgold,fill=witnessgold!6]
    (.15,5.55) rectangle (3.05,7.6);
  \foreach \y/\i in {7.2/1,6.5/2,5.8/3}{
    \fill[witnessgold] (.48,\y) circle (2pt);
    \node[anchor=west] at (.7,\y) {sampled pair $\i$};
    \draw[->,black!55] (3.2,\y)--(4.05,\y);
    \draw[black!25] (4.25,\y)--(12.25,\y);
  }
  \draw[child,line width=2pt] (4.8,7.2)--(5.3,7.2);
  \draw[child,line width=2pt] (8.1,7.2)--(8.65,7.2);
  \draw[child,line width=2pt] (5.05,6.5)--(5.55,6.5);
  \draw[child,line width=2pt] (10.8,6.5)--(11.35,6.5);
  \draw[child,line width=2pt] (8.35,5.8)--(8.85,5.8);
  \draw[child,line width=2pt] (10.95,5.8)--(11.45,5.8);
  \node[anchor=west,align=left] at (.2,4.9)
    {Take the union:\\a cover of $E_m\cap J$};
  \draw[->,black!55] (3.2,4.9)--(4.05,4.9);
  \draw[cluster] (4.25,4.7) rectangle (12.25,5.1);
  \foreach \a/\b in {4.8/5.3,5.05/5.55,8.1/8.65,8.35/8.85,10.8/11.35,10.95/11.45}{
    \draw[refineteal,line width=2pt] (\a,4.9)--(\b,4.9);
  }
  \node[anchor=north west,align=left,text=black,text width=8cm]
    at (4.25,4.45) {{A fixed power of $\ell$ intervals,}\\each of length at most $D_N$};

\begin{scope}[yshift=-9.5mm]
  \node[anchor=west,font=\small] at (0,3.95)
    {\textbf{(b)} Prepare clouds everywhere; then count the retained intervals};
  \foreach \k in {0,1,2}{
    \draw[draw=black!25,fill=black!3] ({4*\k},2.1) rectangle ({4*\k+4},2.42);
    \draw[black!30] ({4*\k+2},2.42)--({4*\k+2},2.61);
    \draw[draw=black!30] ({4*\k+2},2.82) ellipse (.48 and .2);
    \foreach \d in {-.2,0,.2}{
      \fill[black!35] ({4*\k+2+\d},2.82) circle (1.1pt);
    }
  }
  \draw[cluster] (4,2.1) rectangle (8,2.42);
  \node at (6,2.26) {$J$};
  \draw[witnessgold,fill=witnessgold!8] (6,2.82) ellipse (.48 and .2);
  \foreach \d in {-.2,0,.2}{\fill[witnessgold] ({6+\d},2.82) circle (1.3pt);}
  \node[anchor=west] at (8.7,3.4) {scale $N$: width $N^{2/3}$};

  \foreach \k in {0,...,11}{
    \draw[draw=black!25,fill=black!3] (\k,.2) rectangle ({\k+1},.52);
    \draw[black!30] ({\k+.5},.52)--({\k+.5},.68);
    \draw[draw=black!30] ({\k+.5},.85) ellipse (.32 and .16);
    \foreach \d in {-.12,.12}{
      \fill[black!35] ({\k+.5+\d},.85) circle (1pt);
    }
  }
  \foreach \k in {4,7}{
    \draw[child] (\k,.2) rectangle ({\k+1},.52);
    \draw[witnessgold,fill=witnessgold!8] ({\k+.5},.85) ellipse (.32 and .16);
    \foreach \d in {-.12,.12}{
      \fill[witnessgold] ({\k+.5+\d},.85) circle (1.2pt);
    }
  }
  \draw[->,refineteal,line width=.8pt] (5,2.02)--(4.5,1.08);
  \draw[->,refineteal,line width=.8pt] (7,2.02)--(7.5,1.08);
  \node[anchor=west] at (8.7,1.35) {next scale $N'<N$};
  \node[anchor=north,align=center] at (6,-.05)
    {Every mesh interval has its own cloud.\\Only the highlighted intervals are retained in this example.};
\end{scope}
\end{tikzpicture}
\caption{The local cover and its use at successive scales, for one
fixed row $m$ and the same dynamical time interval $[0,h]$.
\textbf{(a)} A dot denotes a sampled \emph{endpoint pair} in the cloud
for $J$, not a point on row $m$. Each pair has a cover of all its
geodesic exits during $[0,h]$; only the portions of those covers in
$J$ are drawn. \textit{Poisson capture} ensures that their union covers $E_m\cap J$.
The geometric agreement with a sampled geodesic is shown separately
in Figure~\ref{fig:cutting}.
\textbf{(b)} Each oval denotes the cloud belonging to the interval
beneath it. Before any intervals are retained, the union bound
establishes the properties in (a) for all intervals at both scales,
including the grey ones. The cover for the highlighted parent $J$
then determines which intervals of the finer mesh are retained.
Their own clouds provide the next refinement. All intervals enter
the failure-probability bound, but only retained intervals enter the
count of the final cover. The numbers and placements are schematic;
no independence between the successive covers is used.}
\label{fig:cloudrefinement}
\end{figure}
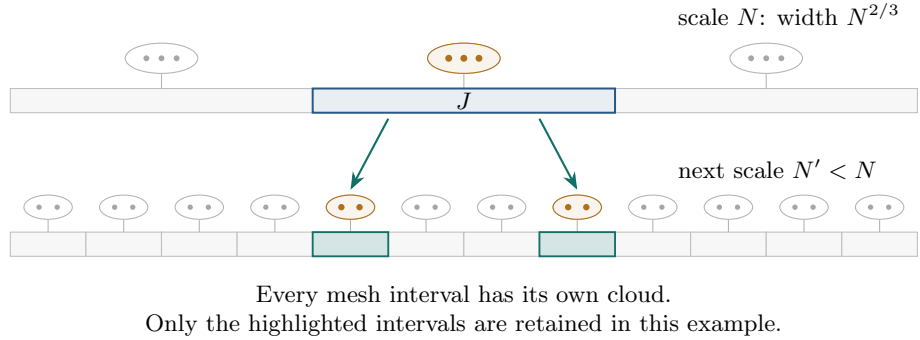

\paragraph{\textbf{Following the covers from coarse to fine.}}
On $\cE^{\mathrm{ref}}_n$, for every parent interval $J$, the
construction gives a cover of $E_m\cap J$ with the bounds just
described. We now use these simultaneous bounds to build successively
finer covers of $E_m$. At the largest scale, we keep the mesh
intervals needed to cover $E_m$. Suppose that $J$ is one of the
intervals kept at some scale $N$. We use its cloud to obtain the
finer cover of $E_m\cap J$, and keep the intervals of the next
{deterministic mesh that meet this cover within the horizontal
range of the original endpoint pair.} These are the
children of $J$. We repeat this operation for each interval kept
at the current scale. Every point of $E_m$ remains covered, although
some retained intervals may contain no actual exit.

For example, if the first cover uses just two mesh intervals, only
those two intervals contribute their finer covers to the next count.
The estimates prepared for all other intervals were needed to ensure
that these two choices were allowed, wherever the first cover happened
to place them. They do not increase the number of children we count.
Figure~\ref{fig:cloudrefinement}(b) illustrates the same distinction
for one parent. Independence of the clouds from the environment is
used when proving the estimates for all possible intervals. Once
$\cE^{\mathrm{ref}}_n$ holds, each step simply uses a cover already
known to be valid; it requires neither independence between successive
covers nor a new probability estimate conditional on the intervals
chosen at the previous step.

\begingroup
\begingroup
\paragraph{\textbf{Choosing the scales: why the refinement accelerates.}}
The next scale is determined by the intervals just produced: we
choose $N_{j+1}$ so that its mesh width $N_{j+1}^{2/3}$ is
comparable to $D_{N_j}$. We can then cover each of those intervals
by a bounded number of intervals in the next mesh. The next cut
is at rows $m-N_{j+1}$ and $m+N_{j+1}$, around the same row $m$;
Figure~\ref{fig:scaleselection} shows these nested geodesic portions.
The time interval $[0,h]$ remains unchanged throughout.

The shorter portions provide progressively better information
because the stability tolerance decreases relative to their own
fluctuations. {Before the additive $\ell^2$ term in the stability tolerance $C(\ell\sqrt{Nh}+\ell^2)$ becomes relevant,}
the two scales to compare are $\sqrt{Nh}$ and $N^{1/3}$. Their ratio is $\sqrt{hN^{1/3}}$, which decreases
as $N$ decreases. Equivalently, the local critical time $N^{-1/3}$
becomes larger, so the same interval $[0,h]$ lies further below
that critical time. The near-maximiser condition therefore becomes
more restrictive relative to the fluctuations at the scale where
it is applied. This is what improves the cover inside a parent;
reusing only the original length-$n$ estimate would not do so.

One can see the resulting acceleration directly in the widths.
While the term $N_jh\ell^2$ dominates the additive $\ell^4$
correction, the ratio of the new width to the parent width is,
up to constants,
\[
 q_j=\frac{N_jh\ell^2}{N_j^{2/3}}
     =\ell^2hN_j^{1/3}.
\]
Ignoring constants in this explanation, the scale rule gives
$q_{j+1}=q_j^{3/2}$. Initially $N_0\asymp n$, so
$q_0\asymp\ell^{-2}$. The factors by which the widths shrink
therefore improve at successive generations: in the idealised
recursion they are $\ell^{-2},\ell^{-3},\ell^{-9/2},\ldots$.
Thus the gain in the logarithm of the width grows geometrically.
Only $O(\log\log n)$ generations are needed to reach a scale
that is a fixed power of {$\ell$}. If the additive $\ell^4$ term dominates
the $N_jh\ell^2$ term
earlier, the next scale is already at most a fixed power of
{$\ell$}, and we can stop. This acceleration is the key to retaining a subpolynomial cover
size when $h=n^{-1/3}(\log n)^{-4}$ is only a logarithmic factor
below the critical time scale. It leads to the quantitative
hitset bound in Theorem~\ref{thm:regional} and, through the
exceptional-time cover, the subpolynomial gauge in
Theorem~\ref{thm:main}. Step 1 of the proof of
Proposition~\ref{prop:generalpoint}, especially
\eqref{eq:acceleration}--\eqref{eq:generations}, makes these
statements precise with the constants and rounding included.

\paragraph{\textbf{Counting the final cover.}}
The number of children per retained parent is at most a fixed
power of {$\ell=\log n$}, uniformly over the scales. Combining this
with the number of generations gives a total multiplication
factor bounded by
\[
 \bigl((\log n)^{C_1}\bigr)^{C_2\log\log n}
 =\exp\{C_1C_2(\log\log n)^2\}.
\]
This is the source of the squared $\log\log n$ in the hitset
bound: each generation costs $O(\log\log n)$ in the logarithm
of the interval count, and there are $O(\log\log n)$ generations.
The initial cover contains only a fixed power of $\log n$
intervals, so it does not change this order.

At the final scale, every retained interval has length at most
a fixed power of $\log n$. We no longer test for near maximisers
inside it: we count every possible unit cell there. Geometric
regularity bounds the horizontal length traversed on each row
by another fixed power of $\log n$. Enlarging each terminal
interval by this amount therefore covers the whole horizontal
path portion whenever its exit lies in that interval. Both costs
are absorbed by increasing the constant in
$\exp\{C(\log\log n)^2\}$. On the complement of
$\cE_n^{\mathrm{ref}}$, the deterministic $O(n^2)$ bound on the
entire hitset, multiplied by its failure probability
$Ce^{-c(\log n)^2}$, contributes negligibly to the expectation.
This proves the one-row bound \eqref{eq:outlinepointrow}.

Finally, the first Poisson capture transfers this fixed-endpoint
estimate to the endpoint regions. Its cloud size and the number
of time subintervals both cost only fixed powers of $\log n$.
Summing over the $O(n)$ rows contributes the factor $n$, giving
the target $n\exp\{C(\log\log n)^2\}$. The same proof, with
$\ell$ left adjustable, yields the tail estimate in
Theorem~\ref{thm:regional}.
\par\endgroup
\par\endgroup
\begingroup

\paragraph{\textbf{From the quantitative hitset to a Hausdorff gauge.}}
Spatial averaging and the unions over directions and time intervals
leave an expected $\exp\{O((\log\log n)^2)\}$ selected intervals
of length $n^{-1/3}$. The gauge $H$ makes their expected
total cost summable over dyadic $n$. Section~\ref{sec:gaugeproof} gives the complete covering
argument.

\par\endgroup

\begingroup
\begin{remark}[Relation to the stability argument of Ganguly--Hammond]
\label{rem:ghstrategy}
The strategy of converting a dynamical event into a rare event in a
static environment already appears in \cite[Section 3.3]{GH24}.
There, passage-time stability is used to construct competitive
time-zero paths reflecting the geometry of a time-$t$ geodesic, and
static near-optimality estimates then control its overlap with the
time-zero geodesic. That argument uses several excursion scales and
deals with endpoints selected by the geodesic; see also
\cite[Sections 6 and 7]{GH24}.

The passage-time input in \cite[Proposition 4.2]{GH24}, for OU
dynamics, is a second-moment bound. The subsequent endpoint-uniform
estimates have polynomial error bounds; see
\cite[Theorem 6.1 and the proof of Proposition 6.2]{GH24}.
{By controlling higher moments of the rotation derivative, we obtain
the exponential increment bound in Proposition~\ref{prop:generalstability}
and its constrained version, Corollary~\ref{cor:increments}.
The same method also gives Gaussian increment bounds for OU dynamics,
as recorded in {Theorem~\ref{thm:oustability}}.} These
stronger tails allow simultaneous control of the passage values
needed throughout the hitset argument.

The geometric construction here is also substantially different. The overlap
argument in \cite{GH24} compares the geodesics at times zero and
one fixed time $t$. Its excursion decomposition treats that pair at
several scales; it does not require a recursively refined cover of
all locations visited as time ranges over an interval. Here, on the other hand, we must
count this entire set of visited locations, and the clusters that
appear inside each earlier cluster lead to such a hierarchy of covers.
Each refinement uses the stability tolerance and the near-maximum
estimate at the length scale of that cluster, with a further use of
independent Poisson representatives for the random cut endpoints.
This recursive cover, together with the stronger passage-time
estimate, supplies the additional ingredients for our hitset bounds.
{Unlike the arguments in \cite{GH24,B25}, the proof here does not
require a twin-peaks estimate: the static input in
Lemma~\ref{lem:packing} controls the number of separated near maximisers
at each scale.}
\end{remark}
\par\endgroup

\section{Concentration for passage maxima}\label{sec:concentration}

We prove Theorem~\ref{thm:stability} and the corresponding estimate
for passage times of staircases restricted to exit a specified
horizontal row through a fixed interval. The latter will be
used alongside the unrestricted estimate when a dynamic crossing is
compared with a static near maximiser.

The following result Proposition~\ref{prop:generalstability} is a generalised version of Theorem~\ref{thm:stability} which gives the passage-time
estimate for arbitrary deterministic endpoints
$p=(x_-,j_-)\le q=(x_+,j_+)$.
Every staircase between these endpoints has the same total horizontal
length $L=x_+-x_-$. This is the quantity that determines the variance
of its Brownian weight, and it will also control the change of the
optimized passage value. For the endpoint pairs at scale $N$ used
later, $L\asymp N$, so the leading scale is $\sqrt{N|t-s|}$.

\begin{proposition}[Passage-time increments for general endpoints]\label{prop:generalstability}
There is an absolute constant $C$ such that the following holds. Let
$p=(x_-,j_-)\le q=(x_+,j_+)\in\ZZ_{\RR}$ be deterministic and set
$L=x_+-x_-$. For all deterministic $s,t\in\RR$ and $k\ge2$,
\begin{equation}\label{eq:mainstability}
 \norm{T_p^{q,t}-T_p^{q,s}}_k
 \le C\left(\sqrt{kL|t-s|}+k\right).
\end{equation}
For $u\ge2$,
\begin{equation}\label{eq:generalstabilitytail}
 \PP\left(\left|T_p^{q,t}-T_p^{q,s}\right|
       >C\left(\sqrt{Lu|t-s|}+u\right)\right)
 \le2e^{-u}.
\end{equation}
At the level of second moments, the sharper estimate
\begin{equation}\label{eq:mainstabilitysecond}
 \norm{T_p^{q,t}-T_p^{q,s}}_2
 \le\frac\pi2\sqrt{L(1-e^{-|t-s|})}
\end{equation}
holds. 
\end{proposition}

{The proof uses a Gaussian rotation argument. We refer the reader to \cite[Theorem 2.2]{Pis86} as the rotation step in our argument is along the same lines. Compared with the second-moment arguments in \cite{Cha14,GH24}, the additional step is to control higher moments of the derivative by estimating its random conditional variance under discrete resampling. To estimate this variance, we condition on the geodesic in the current environment and use its independence from the refresh indicators.}

In \eqref{eq:mainstabilitysecond} above, $1-e^{-|t-s|}$ is the probability that a given block
is refreshed between the two times. We retain the exponential in the sharper
second-moment bound, and the other statements use the simpler upper
bound $1-e^{-|t-s|}\le |t-s|$. We shall first prove a version of Proposition \ref{prop:generalstability} for maxima over finitely many staircases, and a compact approximation will then give the unrestricted and
exit-constrained passage values.

\subsection{Rotating the Brownian increment processes}

For deterministic endpoints $p=(x_-,j_-)\le q=(x_+,j_+)$, we first
express a staircase's weight as a sum of its contributions from
individual unit blocks.

Every staircase from $p$ to $q$ is contained in the rectangle
\[
 \mathcal R(p,q)=[x_-,x_+]\times[j_-,j_+].
\]
The finite set of unit blocks that can contribute to its weight is
\begin{equation}\label{eq:endpointblocks}
 \cB(p,q)=\{(i,j)\in\ZZ^2:j_-\le j\le j_+,\quad
                         [i,i+1]\cap[x_-,x_+]\ne\emptyset\}.
\end{equation}
For $b=(i,j)\in\cB(p,q)$ and $\xi\in\mathfrak S(p,q)$, define
\begin{equation}\label{eq:blockcoordinates}
 \begin{split}
 a_b(\xi)&=\min\{1,\max\{0,\xi(j-1)-i\}\},\\
 d_b(\xi)&=\min\{1,\max\{0,\xi(j)-i\}\},\qquad
 \ell_b(\xi)=d_b(\xi)-a_b(\xi).
 \end{split}
\end{equation}
When $a_b(\xi)<d_b(\xi)$, the part of $\xi$ in block $b$ is the
horizontal segment from $(i+a_b(\xi),j)$ to $(i+d_b(\xi),j)$.
When they coincide, the block contributes zero weight. We refer the reader to Figure~\ref{fig:blockcoordinates} which shows these coordinates within a single block. In particular,
\begin{equation}\label{eq:coeff}
 0\le\ell_b(\xi)\le1,\qquad
 \sum_{b\in\cB(p,q)}\ell_b(\xi)=x_+-x_-.
\end{equation}
For each $b\in\cB(p,q)$, let $X_b\colon[0,1]\to\RR$ be a continuous
function with $X_b(0)=0$. In the BLPP environment, $X_b(t)$ is the
Brownian increment $W_j(i+t)-W_j(i)$ when $b=(i,j)$. For any such
collection $X=(X_b)_{b\in\cB(p,q)}$, define
\begin{equation}\label{eq:blockweight}
 \wgt^X(\xi)=\sum_{b\in\cB(p,q)}
       \bigl(X_b(d_b(\xi))-X_b(a_b(\xi))\bigr).
\end{equation}
For the Brownian increments $X_{i,j}(u)=W_j(i+u)-W_j(i)$, summing
the block contributions along each row gives
\[
 \wgt^X(\xi)
 =\sum_{j=j_-}^{j_+}
    \bigl(W_j(\xi(j))-W_j(\xi(j-1))\bigr)
 =\wgt(\xi).
\]
{Thus,} \eqref{eq:blockweight} agrees with the staircase weight from
Section~\ref{sec:dynamics}. This notation will be useful to denote the weight of a staircase after some of the blocks have been replaced.

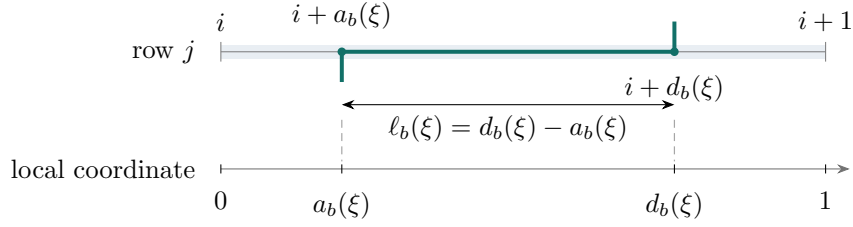
\begin{figure}[tbp]
\centering
\begin{tikzpicture}[x=8cm,y=1cm]
  \fill[clusterblue!10] (0,-.09) rectangle (1,.09);
  \draw[black!50] (0,0)--(1,0);
  \draw[black!50] (0,-.13)--(0,.13);
  \draw[black!50] (1,-.13)--(1,.13);
  \draw[refineteal,line width=1.5pt] (.2,-.4)--(.2,0)--(.75,0)--(.75,.4);
  \fill[refineteal] (.2,0) circle (1.6pt);
  \fill[refineteal] (.75,0) circle (1.6pt);
  \node[above=5pt] at (.2,0) {$i+a_b(\xi)$};
  \node[below=5pt] at (.75,0) {$i+d_b(\xi)$};
  \node[above=4pt] at (0,0) {$i$};
  \node[above=4pt] at (1,0) {$i+1$};
  \node[anchor=east] at (-.025,0) {row $j$};
  \draw[<->] (.2,-.7)--node[below] {$\ell_b(\xi)=d_b(\xi)-a_b(\xi)$}(.75,-.7);
  \draw[->,black!55] (0,-1.55)--(1.04,-1.55);
  \foreach \x/\lab in {0/0,.2/{a_b(\xi)},.75/{d_b(\xi)},1/1}{
    \draw (\x,-1.61)--(\x,-1.49);
    \node[below=3pt] at (\x,-1.61) {$\lab$};
  }
  \node[anchor=east] at (-.025,-1.55) {local coordinate};
  \draw[guide] (.2,-.9)--(.2,-1.45);
  \draw[guide] (.75,-.9)--(.75,-1.45);
\end{tikzpicture}
\caption{The portion of a staircase in the unit block $b=(i,j)$.
The coordinates $a_b(\xi),d_b(\xi)\in[0,1]$ are measured from the
left endpoint $i$ of the block. Its weight contribution is
$X_b(d_b(\xi))-X_b(a_b(\xi))$. If the staircase enters row $j$
to the left of the block, then $a_b=0$; if it exits to the right,
then $d_b=1$. A block traversed completely has $\ell_b=1$.
The drawing shows the case where both endpoints lie inside the block.}
\label{fig:blockcoordinates}
\end{figure}

Fix a deterministic staircase $\xi$ from $p$ to $q$. In block $b$,
it uses a horizontal length $\ell_b(\xi)$, and these lengths sum to
$L=x_+-x_-$. Let $S_b$ indicate whether block $b$ is refreshed.
If the indicators are independent Bernoulli variables with parameter
$\rho$, then
\[
 \EE\left[\sum_{b\in\cB(p,q)}S_b\ell_b(\xi)\right]=\rho L.
\]
The sum is over all rows traversed by the staircase. It measures only
the horizontal portions lying in refreshed blocks; see
Figure~\ref{fig:refreshedlength}.

For an optimized passage value, the maximizing staircase is random
and changes when the environment changes. The useful feature of the
interpolation below is that, at each fixed angle, its Brownian
environment has the same distribution for every choice of the
indicators $(S_b)$. Consequently, after revealing that environment
and its selected maximizing staircase, the indicators are still
independent Bernoulli variables. This will allow the same computation
of the mean horizontal length in refreshed blocks, together with an
exponential-moment bound, for the selected staircase.

\begin{figure}[tbp]
\centering
\begin{tikzpicture}[x=1.7cm,y=1.1cm]
  \foreach \j in {0,1,2}{
    \draw[guide] (-.05,\j)--(5.15,\j);
    \node[anchor=east] at (-.1,\j) {$\j$};
    \foreach \i in {0,1,2,3,4}{
      \draw[black!25] (\i,\j-.11) rectangle (\i+1,\j+.11);
    }
  }
  \foreach \i/\j in {0/0,2/1,3/1,4/2}{
    \fill[witnessgold!25] (\i,\j-.11) rectangle (\i+1,\j+.11);
  }
  \draw[refineteal,line width=1.3pt]
    (.3,0)--(1.2,0)--(1.2,1)--(3.6,1)--(3.6,2)--(4.8,2);
  \draw[witnessgold,line width=2.4pt] (.3,0)--(1,0);
  \draw[witnessgold,line width=2.4pt] (2,1)--(3.6,1);
  \draw[witnessgold,line width=2.4pt] (4,2)--(4.8,2);
  \fill (.3,0) circle (1.5pt) node[below=4pt] {$p$};
  \fill (4.8,2) circle (1.5pt) node[above=4pt] {$q$};
  \node[anchor=west,refineteal] at (1.28,.5) {$\xi_\theta$};
  \draw[<->] (2,.62)--node[below] {$\ell_{(2,1)}=1$}(3,.62);
  \draw[<->] (.3,-.75)--node[below] {$L=x_+-x_-$}(4.8,-.75);
  \node[anchor=west,font=\footnotesize] at (.3,2.75)
    {Shaded blocks have $S_b=1$; gold portions contribute to $V_\theta$.};
\end{tikzpicture}
\caption{The horizontal length in refreshed blocks. A block contributes
only the length of the portion of the selected staircase inside it,
which can be smaller than one. {Thus,} $V_\theta=\sum_bS_b\ell_b(\xi_\theta)$,
whereas $\sum_b\ell_b(\xi_\theta)=L$. Conditional on the current
environment and $(S_b)$, the derivative $D_\theta$ has variance
$V_\theta$. The configuration shown is schematic.}
\label{fig:refreshedlength}
\end{figure}
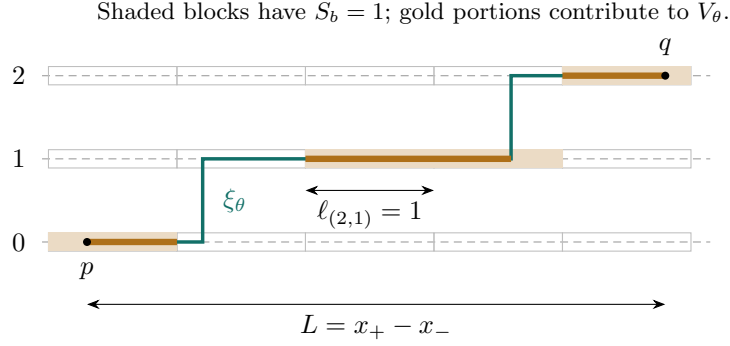

We will use the following elementary bound when converting an
exponential tail into moment estimates.

\begin{lemma}[An elementary gamma-function bound]\label{lem:gammamoment}
Let $\Gamma(z)=\int_0^\infty x^{z-1}e^{-x}\,dx$, $z>0$, be the
Euler gamma function. For every $v\ge1$,
\[
 \Gamma(v+1)^{1/v}\le4v.
\]
\end{lemma}

\begin{proof}
The function $x\mapsto x^v e^{-x/2}$ attains its maximum at $x=2v$.
Consequently,
\[
 \Gamma(v+1)
 =\int_0^\infty (x^v e^{-x/2})e^{-x/2}\,dx
 \le (2v/e)^v\int_0^\infty e^{-x/2}\,dx
 =2(2v/e)^v.
\]
Taking $v$th roots and using $2^{1/v}\le2$ proves the bound.
\end{proof}

For the application to passage times over $[0,h]$, keep in mind
$L\asymp n$ and $\rho=1-e^{-h}\sim h$,
with {$h=n^{-1/3}\ell^{-4}$}. At a shorter
geometric scale $N$, {the same argument has $L\asymp N$, while the time interval $[0,h]$ remains unchanged}. For a single increment of duration $s\le h$, the
refresh probability is instead $1-e^{-s}\le h$.
The following lemma carries out the Gaussian rotation argument
for a finite collection of staircases. Its constants are independent
of the number of staircases, {allowing us to obtain the corresponding
estimates for unrestricted and restricted passage times by compact approximation}.

\begin{samepage}
\begin{lemma}[A Gaussian rotation estimate]\label{lem:gaussian}
Fix $p=(x_-,j_-)\le q=(x_+,j_+)\in\ZZ_{\RR}$, put $L=x_+-x_-$,
and let $\cA\subseteq\mathfrak S(p,q)$ be a finite, nonempty and
deterministic set of staircases. Write $\cB=\cB(p,q)$. Let
$(X_b)_{b\in\cB}$ and $(\widetilde X_b)_{b\in\cB}$ be independent
families of {i.i.d.\ standard Brownian motions on $[0,1]$}.
Independently of these families, let $(S_b)_{b\in\cB}$ be {i.i.d.}
Bernoulli variables with parameter $\rho\in[0,1]$. Define
\begin{equation}\label{eq:replacementcoupling}
 X'_b=(1-S_b)X_b+S_b\widetilde X_b,\qquad
 F=\max_{\xi\in\cA}\wgt^X(\xi),\qquad
 F'=\max_{\xi\in\cA}\wgt^{X'}(\xi).
\end{equation}
There is an absolute constant $C$, independent in particular of
$|\cA|$, such that, for every $k\ge2$,
\begin{equation}\label{eq:lp}
 \norm{F'-F}_k\le C(\sqrt{k\rho L}+k).
\end{equation}
Moreover, $\norm{F'-F}_2\le(\pi/2)\sqrt{\rho L}$, and, for $u\ge2$,
\begin{equation}\label{eq:tail}
 \PP\bigl(|F'-F|>C(\sqrt{\rho Lu}+u)\bigr)\le2e^{-u}.
\end{equation}
\end{lemma}
\end{samepage}

\begin{proof}
Fix a deterministic enumeration
\[
 \cA=\{\xi_1,\ldots,\xi_M\}.
\]
We interpolate between the original and refreshed environments, using
this enumeration to select a maximizing staircase whenever there is
a tie. We then estimate the derivative of the optimized weight.

\paragraph{\textbf{An interpolation with the same marginal environment.}}
For $0\le\theta\le\pi/2$ and $b\in\cB$, set
\begin{equation}\label{eq:rotation}
 X_b(\theta)=
 \begin{cases}
 X_b,&S_b=0,\\
 \cos\theta\,X_b+\sin\theta\,\widetilde X_b,&S_b=1,
 \end{cases}
 \qquad F_\theta=\max_{\xi\in\cA}\wgt^{X(\theta)}(\xi).
\end{equation}
At angle zero this is the original environment $X$, and at angle
$\pi/2$, it is the refreshed environment $X'$. Thus, in this coupling,
\[
 F_0=F,\qquad F_{\pi/2}=F'.
\]
The interpolation gives a continuous parameter with respect to which
we can differentiate, even though the original update dynamics has jumps.

For each fixed $\theta$, conditional on any realization of
$(S_b)_{b\in\cB}$, the processes $(X_b(\theta))_{b\in\cB}$ are
independent standard Brownian motions: an unrefreshed block uses
$X_b$, while a refreshed block uses a linear combination with
$\cos^2\theta+\sin^2\theta=1$. Since this conditional law does not
depend on the indicators,
\begin{equation}\label{eq:independence}
 (X_b(\theta))_{b\in\cB}\text{ is independent of }
 (S_b)_{b\in\cB}\qquad\text{for each fixed }\theta.
\end{equation}
Using the enumeration fixed at the start, define
\begin{equation}\label{eq:selectedmaximiser}
 i_\theta=\min\{i:\wgt^{X(\theta)}(\xi_i)=F_\theta\},
 \qquad \xi_\theta=\xi_{i_\theta}.
\end{equation}
The selected maximizer $\xi_\theta$ is determined by $X(\theta)$
alone. {Hence,} it is independent of $(S_b)$ at each fixed angle $\theta$ as well.

\paragraph{\textbf{Differentiation and the moment to be estimated.}}
For $b\in\cB$, define
\[
 Y_b(\theta)=-\sin\theta\,X_b+\cos\theta\,\widetilde X_b.
\]
Recall from \eqref{eq:blockcoordinates} that $a_b(\xi),d_b(\xi)$
are the coordinates, measured from the left edge of the unit block,
of the portion used by $\xi$, and that
$\ell_b(\xi)=d_b(\xi)-a_b(\xi)$. For each fixed $\xi_i$,
$g_i(\theta)=\wgt^{X(\theta)}(\xi_i)$ is continuously differentiable:
it is a finite sum of fixed Brownian increments multiplied by
$1$, $\cos\theta$, or $\sin\theta$. Its derivative is
\[
 g_i'(\theta)=\sum_{b\in\cB}S_b
 \bigl(Y_b(\theta)(d_b(\xi_i))-Y_b(\theta)(a_b(\xi_i))\bigr).
\]
Almost surely, $\max_i\sup_\theta|g_i'(\theta)|<\infty$.
{Consequently,} $F_\theta=\max_i g_i(\theta)$ is Lipschitz and
absolutely continuous, and is differentiable at almost every angle.

At an interior angle where $F$ is differentiable, every index $i$
with $g_i(\theta)=F_\theta$ satisfies $F_\theta'=g_i'(\theta)$.
Indeed, $\beta\mapsto F_\beta-g_i(\beta)$ has a minimum of zero
at $\theta$, so its derivative there vanishes. Define at every angle
\begin{equation}\label{eq:rotationderivative}
 D_\theta=\sum_{b\in\cB}S_b
 \bigl(Y_b(\theta)(d_b(\xi_\theta))
             -Y_b(\theta)(a_b(\xi_\theta))\bigr).
\end{equation}
Thus, almost surely, $D_\theta=F_\theta'$ for Lebesgue-almost every
$\theta$. Note that this conclusion does not require uniqueness
of the maximizing staircase simultaneously for every value of
$\theta$. The almost-sure absolute continuity of
$\theta\mapsto F_\theta$ yields
\begin{equation}\label{eq:rotationintegral}
 F_{\pi/2}-F_0=\int_0^{\pi/2}D_\theta\,d\theta
 \qquad\text{almost surely}.
\end{equation}
Therefore, by Minkowski's integral inequality, we have
\begin{equation}\label{eq:rotationminkowski}
 \norm{F'-F}_k\le\int_0^{\pi/2}\norm{D_\theta}_k\,d\theta.
\end{equation}
It remains to bound $\norm{D_\theta}_k$ uniformly over deterministic
$\theta\in[0,\pi/2]$.

\paragraph{\textbf{The conditional Gaussian law.}}
Let
\[
 \mathscr X_\theta=\sigma(X_b(\theta):b\in\cB),\qquad
 \mathscr S=\sigma(S_b:b\in\cB).
\]
Conditional on $\mathscr S$, consider a block with $S_b=1$.
The pair of processes $(X_b(\theta),Y_b(\theta))$ is jointly Gaussian,
each component is standard Brownian motion, and for $u,v\in[0,1]$,
\[
 \Cov\bigl(Y_b(\theta)(u),X_b(\theta)(v)\mid\mathscr S\bigr)
 =(-\sin\theta\cos\theta+\cos\theta\sin\theta)\min(u,v)=0.
\]
The two processes on this block are therefore independent. Moreover,
the pairs $(X_b,\widetilde X_b)$ used on distinct blocks are
independent. It follows that all the $Y_b(\theta)$ with $S_b=1$
are independent of each other and of the entire collection
$(X_b(\theta))_{b\in\cB}$, conditional on $\mathscr S$.

Conditioning further on $\mathscr X_\theta$ fixes $\xi_\theta$ and
the endpoints $a_b(\xi_\theta),d_b(\xi_\theta)$. Each summand in
$D_\theta$ with $S_b=1$ is now a centered Gaussian variable with
variance $\ell_b(\xi_\theta)$, and these summands are conditionally
independent. With
$V_\theta=\sum_{b\in\cB}S_b\ell_b(\xi_\theta)$, we have
\begin{equation}\label{eq:variance}
 D_\theta\mid\sigma(\mathscr X_\theta\cup\mathscr S)
 \sim\cN(0,V_\theta).
\end{equation}
The variance is the total horizontal length of $\xi_\theta$ in
refreshed blocks. If $Z$ is a standard normal variable, taking
conditional moments and then expectations gives
\begin{equation}\label{eq:gaussianmixturemoment}
 \norm{D_\theta}_k
 =\norm Z_k\norm{V_\theta}_{k/2}^{1/2}
 \le C\sqrt{k}\,\norm{V_\theta}_{k/2}^{1/2}.
\end{equation}

\paragraph{\textbf{Moments of the horizontal length in refreshed blocks.}}
We now condition only on $\mathscr X_\theta$, leaving the refresh
indicators $(S_b)_{b\in\cB}$ unrevealed. This reveals the
selected maximizing staircase $\xi_\theta$, so the numbers
$w_b=\ell_b(\xi_\theta)$ are fixed. By \eqref{eq:coeff},
\[
 0\le w_b\le1,\qquad \sum_{b\in\cB}w_b=L.
\]
The sum is over all blocks, including those that are not refreshed.
Crucially, \eqref{eq:independence} says that after this conditioning
the variables $(S_b)_{b\in\cB}$ are still independent Bernoulli
variables of parameter $\rho$. {Thus,}
\begin{align*}
 \EE[e^{V_\theta}\mid\mathscr X_\theta]
 &=\EE\left[\exp\left(\sum_{b\in\cB}S_bw_b\right)
                      \biggm|\mathscr X_\theta\right]\\
 &=\prod_{b\in\cB}\EE[e^{S_bw_b}\mid\mathscr X_\theta]
   =\prod_{b\in\cB}(1-\rho+\rho e^{w_b})\\
 &\le\exp\left(\rho\sum_{b\in\cB}(e^{w_b}-1)\right)
 \le\exp((e-1)\rho L).
\end{align*}
For each factor, $S_b=0$ with probability $1-\rho$ and $S_b=1$
with probability $\rho$, which gives $1-\rho+\rho e^{w_b}$.
The first inequality uses $1+x\le e^x$, and the last uses
$e^w-1\le(e-1)w$ for $0\le w\le1$, together with
$\sum_{b\in\cB}w_b=L$.

Applying conditional Markov's inequality to $e^{V_\theta}$ gives,
for every $u\ge0$,
\begin{equation}\label{eq:refreshedtail}
 \begin{aligned}
 \PP(V_\theta>(e-1)\rho L+u\mid\mathscr X_\theta)
 &\le e^{-(e-1)\rho L-u}\EE[e^{V_\theta}\mid\mathscr X_\theta]\\
 &\le e^{-u}.
 \end{aligned}
\end{equation}
Put $W_\theta=\max\{V_\theta-(e-1)\rho L,0\}$.
Then $V_\theta\le(e-1)\rho L+W_\theta$, and for every $v\ge1$,
integration of \eqref{eq:refreshedtail} yields
\[
 \EE[W_\theta^v\mid\mathscr X_\theta]
 =v\int_0^\infty u^{v-1}
               \PP(W_\theta>u\mid\mathscr X_\theta)\,du
 \le v\int_0^\infty u^{v-1}e^{-u}\,du=\Gamma(v+1).
\]
Here $\Gamma$ denotes the usual Euler gamma function. The elementary
bound $\Gamma(v+1)^{1/v}\le Cv$ is proved in
Lemma~\ref{lem:gammamoment}. Taking expectations and using the
triangle inequality in $L^v$, we obtain\footnote{Here $\rho L$ is
the mean horizontal length in refreshed blocks. The additive $v$
accounts in high moments for rare configurations in which a much
larger portion of the selected maximizing staircase lies in refreshed
blocks, using the exponential tail in
\eqref{eq:refreshedtail}. In the corresponding rotation argument for the OU dynamics, each driving Brownian motion would be rotated with an independent copy through the same angle. Conditional on the current environment, the derivative would be Gaussian with variance exactly $L$, and the additional randomness arising from the choice of refreshed blocks in the discrete dynamics would therefore be absent.}
\begin{equation}\label{eq:refreshedmoments}
 \norm{V_\theta}_v\le(e-1)\rho L+Cv\le C(\rho L+v).
\end{equation}
Substituting $v=k/2$ in \eqref{eq:gaussianmixturemoment} gives
\begin{equation}\label{eq:derivativemoments}
 \norm{D_\theta}_k
 \le C\sqrt{k}\sqrt{\rho L+k}
 \le C(\sqrt{k\rho L}+k).
\end{equation}
This bound is uniform in $\theta$. {Hence,} \eqref{eq:rotationminkowski}
gives, after absorbing the factor $\pi/2$ into $C$,
\[
 \norm{F'-F}_k\le C(\sqrt{k\rho L}+k),
\]
which is \eqref{eq:lp}.

For the sharper second-moment bound, recall from
\eqref{eq:independence} that the refresh indicators $(S_b)_{b\in\cB}$
are independent of the current environment $\mathscr X_\theta$.
Since $\xi_\theta$ is determined by that environment, conditioning
on $\mathscr X_\theta$ fixes its block lengths and leaves
$\EE[S_b\mid\mathscr X_\theta]=\rho$. {Thus,} \eqref{eq:coeff} and
the conditional Gaussian law \eqref{eq:variance} give
\[
 \EE[V_\theta\mid\mathscr X_\theta]
   =\rho\sum_{b\in\cB}\ell_b(\xi_\theta)=\rho L,\qquad
 \norm{D_\theta}_2^2=\EE V_\theta=\rho L.
\]
Applying \eqref{eq:rotationminkowski} with $k=2$ proves
$\norm{F'-F}_2\le(\pi/2)\sqrt{\rho L}$.
{Finally, denote the constant $C$ in \eqref{eq:lp} by $C_1$.} For $u\ge2$,
Markov's inequality with moment order $k=u$ gives
\[
 \PP\bigl(|F'-F|>eC_1(\sqrt{u\rho L}+u)\bigr)
 \le\frac{\EE|F'-F|^u}
              {[eC_1(\sqrt{u\rho L}+u)]^u}
 \le e^{-u}.
\]
Increasing the absolute constant gives \eqref{eq:tail}.
\end{proof}

\subsection{From finitely many staircases to passage times}

An unrestricted BLPP passage time is a maximum over the infinite collection
$\mathfrak S(p,q)$, whereas Lemma~\ref{lem:gaussian} concerns a finite
collection. We pass between them by approximating a compact class in
the exit-coordinate metric of Section~\ref{sec:pathclasses} by
increasing finite subsets. Continuity of the weight makes their
maxima converge to the desired passage value. This also allows us to treat restricted passage times.

\begin{samepage}
\begin{corollary}[Constrained BLPP passage increments]\label{cor:increments}
Let $p=(x_-,j_-)\le q=(x_+,j_+)\in\ZZ_{\RR}$ be deterministic and set
$L=x_+-x_-$. Let
$\cA\subseteq\mathfrak S(p,q)$ be nonempty, compact and deterministic,
and put $F^t=\max_{\xi\in\cA}\wgt^t(\xi)$. For $k\ge2$ and $s\ge0$,
\[
 \norm{F^s-F^0}_k\le C(\sqrt{kLs}+k).
\]
For $u\ge2$,
\[
 \PP\bigl(|F^s-F^0|>C(\sqrt{Lsu}+u)\bigr)\le2e^{-u}.
\]
Also $\norm{F^s-F^0}_2\le(\pi/2)\sqrt{L(1-e^{-s})}$.
These estimates apply to the unrestricted value $F^t=T_p^{q,t}$
and to the restricted value
\[
 F_I^t=\max_{x\in I}Z_p^{q,\bullet,t}(x,m),
\]
for deterministic $j_-\le m<j_+$ and nonempty closed
$I\subseteq[x_-,x_+]$.
\end{corollary}
\end{samepage}

\begin{proof}
Fix $s\ge0$. For $b\in\cB(p,q)$, let $S_b$ be the indicator that
its clock rings at least once in $(0,s]$. These variables are
independent Bernoulli variables with parameter $\rho=1-e^{-s}$.
They are independent of the initial Brownian samples. Conditional on
all the clocks, the last sample used at time $s$ on a block with
$S_b=1$ is a fresh Brownian motion independent of the initial samples
and of the samples used in the other blocks. {Thus,}
\begin{equation}\label{eq:twotimecoupling}
 \bigl((X_b^0)_{b\in\cB(p,q)},(X_b^s)_{b\in\cB(p,q)}\bigr)
 \stackrel{\mathrm d}{=}(X,X'),
\end{equation}
where $X,X'$ are exactly the collections in
\eqref{eq:replacementcoupling} with parameter $\rho=1-e^{-s}$.

Choose a countable dense subset $\{\zeta_1,\zeta_2,\ldots\}$ of
$\cA$ in the Euclidean metric on the exit coordinates of
Section~\ref{sec:pathclasses}, and set
$\cA_j=\{\zeta_1,\ldots,\zeta_j\}$. For $t\in\{0,s\}$ define
\[
 F_j^t=\max_{\xi\in\cA_j}\wgt^t(\xi),\qquad
 F^t=\max_{\xi\in\cA}\wgt^t(\xi).
\]
Since $\cA_j$ increases to a dense subset and the weight is
continuous on $\cA$, we have $F_j^t\to F^t$ almost surely for
$t=0,s$. By \eqref{eq:twotimecoupling}, Lemma~\ref{lem:gaussian}
applies to every $\cA_j$ with $\rho=1-e^{-s}$ and the same
absolute constant $C$. Fatou's lemma gives
\[
 \EE|F^s-F^0|^k
 \le\liminf_{j\to\infty}\EE|F_j^s-F_j^0|^k,
 \qquad k\ge2.
\]
{Thus,} both the general moment bound and the sharper second-moment
bound pass to the compact class. Using $1-e^{-s}\le s$ gives the
stated moment estimates. The tail bound follows by Markov's
inequality at moment order $k=u$.
\end{proof}

\par\begingroup
\par\begingroup
\begin{proof}[Proof sketch for Theorem~\ref{thm:oustability}]
First take a finite class $\cA$ of staircases from $\0$ to $\bn$,
and write $F^r=\max_{\xi\in\cA}\wgt^r(\xi)$. Let
$B,\widetilde B$ be independent Brownian environments and set
\[
 B_\alpha=\cos\alpha\,B+\sin\alpha\,\widetilde B,
 \qquad
 Y_\alpha=-\sin\alpha\,B+\cos\alpha\,\widetilde B,
 \qquad 0\le\alpha\le\vartheta=\arccos(e^{-t}).
\]
The environments $(B_0,B_\vartheta)$ have the two-time OU law,
and, at each deterministic angle, $Y_\alpha$ is a Brownian
environment independent of $B_\alpha$. If $f(\alpha)$ is the
optimized value over $\cA$ in $B_\alpha$, the finite-maximum
derivative argument used in Lemma~\ref{lem:gaussian} gives,
almost surely for almost every $\alpha$, that $f'(\alpha)$ is
the weight collected from $Y_\alpha$ along a selected maximiser
for $B_\alpha$. Conditional on $B_\alpha$, this weight is centered
Gaussian with variance exactly $n$, the common horizontal length
of these staircases. Thus $\|f'(\alpha)\|_k\le C\sqrt{kn}$, and
Minkowski's inequality gives
\[
 \|F^t-F^0\|_k
 \le\int_0^\vartheta\|f'(\alpha)\|_k\,d\alpha
 \le C\vartheta\sqrt{kn}
 \le C\sqrt{kn(1-e^{-t})}.
\]
The bound is independent of the finite class. Approximating the
compact space of all staircases from $\0$ to $\bn$ by increasing
finite subsets, as in the proof of Corollary~\ref{cor:increments},
passes it to the unrestricted passage times. Finally,
$1-e^{-t}\le t$ gives the second inequality in
\eqref{eq:oumainmoment}, and Markov's inequality at order $k=u$,
after increasing $C$, gives \eqref{eq:oumaintail}.
\end{proof}
\par\endgroup
\par\endgroup

\begin{proof}[Proof of Proposition~\ref{prop:generalstability} and Theorem~\ref{thm:stability}]
Take $\cA=\mathfrak S(p,q)$ in the approximation argument in the proof
of Corollary~\ref{cor:increments}.
It gives \eqref{eq:lp} and its second-moment version with
$\rho=1-e^{-|t-s|}$ and horizontal length exactly $L$. Stationarity
of the dynamics identifies the increment between $s$ and $t$ with
that between $0$ and $|t-s|$, up to sign. Using
$1-e^{-|t-s|}\le |t-s|$ proves
\eqref{eq:mainstability} and \eqref{eq:mainstabilitysecond}; the same
Markov bound proves \eqref{eq:generalstabilitytail}. If $L=0$, all
staircase weights are zero and the assertions are immediate.
Finally, taking $p=\0$, $q=\bn$ and $s=0$ in
\eqref{eq:generalstabilitytail} gives \eqref{eq:mainstabilitytail}
for every $t\ge0$, proving Theorem~\ref{thm:stability}.
\end{proof}

\subsection{Uniform control over a time interval}

As explained in the outline, if the geodesic $\Gamma_p^{q,t}$
exits row $m$ through an interval $I$, then $F_I^t=T_p^{q,t}$. To compare with the
time-zero environment, we must control the changes of both $T_p^{q,t}$
and $F_I^t$. Moreover, the visit itself can occur at any time in $[0,h]$, and we therefore need bounds on quantuties of the type $\sup_{0\le t\le h}|F^t-F^0|$ for both restricted and unrestricted passage times.
In the following lemma, we shall use an index set $\mathcal L_n$ to record different path
classes and endpoint pairs. Its polynomial size permits a union
bound over the list.

\begingroup
\begin{lemma}[Uniform stability for an indexed family]\label{lem:timeuniform}
Fix $B,{C_0}>0$. For each $n$, let $\mathcal L_n$ be a deterministic
index set with $|\mathcal L_n|\le n^B$. For $i\in\mathcal L_n$,
let $p_i=(x_{i,-},j_{i,-})\le q_i=(x_{i,+},j_{i,+})$ be
deterministic, put $L_i=x_{i,+}-x_{i,-}$, and let
$\cA_i\subseteq\mathfrak S(p_i,q_i)$ be nonempty, compact and
deterministic. Suppose all their containing rectangles lie in a
deterministic rectangle $\mathcal R_n$ meeting at most ${C_0}n^2$
update blocks. For $\ell\ge\log n$ and $0<h\le1$, put
\begin{equation}\label{eq:AN}
 A_N=\ell\sqrt{Nh}+\ell^2\quad(N\ge0),\qquad
 F_i^t=\max_{\xi\in\cA_i}\wgt^t(\xi).
\end{equation}
There are $C_1,C,c>0$ and $n_0<\infty$, depending only on
$B,{C_0}$, such that, for every $n\ge n_0$,
\begin{equation}\label{eq:uniftime}
 \PP\left(\exists i\in\mathcal L_n:
       \sup_{0\le t\le h}|F_i^t-F_i^0|>C_1 A_{L_i}\right)
 \le Ce^{-c\ell^2}.
\end{equation}
The constants are uniform in $\ell,h$ and the listed classes.
The same bound holds conditionally on a random indexed list,
independent of the dynamics, whenever it satisfies these size and
rectangle hypotheses.
\end{lemma}

{Note that, as the dynamics proceeds over $[0,h]$, only finitely many configurations occur in the containing rectangle.}
A sufficiently fine deterministic time mesh sees every configuration
with high probability, because no mesh interval contains two updates.
We can therefore apply the fixed-time estimate at deterministic
mesh points, without conditioning the environment on update events.

\begin{proof}
Set $\Delta=e^{-\ell^2}$, $K_n=\lceil h/\Delta\rceil$ and
$t_j=jh/K_n$ for $0\le j\le K_n$. Then
$t_j-t_{j-1}\le\Delta$ and $K_n+1\le e^{\ell^2}+2$.
Let $\cB_n$ be the update blocks meeting $\mathcal R_n$, and
write $K_{\cB_n}(I)$ for the total number of their rings in $I$.
Define
\begin{align*}
 \cE_n^{\mathrm{mesh}}
 &=\bigcap_{i\in\mathcal L_n}\bigcap_{j=0}^{K_n}
       \{|F_i^{t_j}-F_i^0|\le C_1A_{L_i}\},\\
 \cE_n^{\mathrm{one}}
 &=\bigcap_{j=1}^{K_n}
       \{K_{\cB_n}((t_{j-1},t_j])\le1\}.
\end{align*}
For $L_i>0$, Corollary~\ref{cor:increments}, with deviation
parameter $2\ell^2$ and elapsed time $t_j\le h$, gives a
failure bound $2e^{-2\ell^2}$ at each index and mesh point.
Its threshold is at most $C_1A_{L_i}$ for a fixed sufficiently
large $C_1$. When $L_i=0$ the passage value is identically zero.
Consequently,
\[
 \PP((\cE_n^{\mathrm{mesh}})^c)
 \le 2n^B(e^{\ell^2}+2)e^{-2\ell^2}.
\]
{For every deterministic interval $I$, $K_{\cB_n}(I)$ is Poisson with mean $\Lambda_n|I|$, where the total clock rate is $\Lambda_n=|\cB_n|\le {C_0}n^2$.}
The Poisson bound $\PP(\mathrm{Poi}(v)\ge2)\le v^2/2$ gives
\[
 \PP((\cE_n^{\mathrm{one}})^c)
 \le\frac{\Lambda_n^2}{2}\sum_{j=1}^{K_n}(t_j-t_{j-1})^2
 \le Cn^4h e^{-\ell^2}.
\]
Their sum is at most $Ce^{-c\ell^2}$ uniformly for
$\ell\ge\log n$ and large $n$.

Almost surely there is no ring at any deterministic $t_j$.
On $\cE_n^{\mathrm{one}}$, every configuration during
$[t_{j-1},t_j]$ is its configuration at one of the two endpoints.
Thus $\cE_n^{\mathrm{mesh}}\cap\cE_n^{\mathrm{one}}$ implies
the simultaneous stability bound throughout $[0,h]$.
Finally, conditioning on an independent indexed list leaves the
law of the dynamical environment unchanged. Applying the same
argument to each realization satisfying the hypotheses proves
the conditional assertion.
\end{proof}
\par\endgroup

\section{Covering dynamic exits by static near maximisers}\label{sec:crossings}

Fix an endpoint pair $u,v$ and a row $m$. If the geodesic exits
through a unit interval $I$ at some time $t\le h$, then the
restricted value $F_I^t$ equals $T_u^{v,t}$. The stability estimates
from the preceding section will imply that $I$ contains a point
whose time-zero routed weight is close to $T_u^{v,0}$. We will cover
all such points by a small number of intervals, using
the static near-maximiser estimate in Lemma~\ref{lem:packing}.
This gives the local {exit cover} that will later be applied at every scale.

{{The local scale will be denoted by $N\le n$.}
We will use an allowance $\ell\ge\log n$, work at scales
$N\ge\ell^Q$ for a fixed $Q>1080$, and take
$0<h\le n^{-1/3}\ell^{-4}$.}
Fix $0<a_0<b_0<\infty$, $\beta\in(0,1/2)$ and a compact
interval $K\subset(0,\infty)$. The deterministic endpoints
$u=(x_-,j_-)\le v=(x_+,j_+)$ and integer row $m$ satisfy the
following scale conditions:
\begin{equation}\label{eq:localranges}
 \begin{gathered}
 a_0N\le D:=j_+-j_-\le b_0N,\qquad
 \lambda:=\frac{x_+-x_-}{D}\in K,\\
 \beta D\le m-j_-\le(1-\beta)D-1.
 \end{gathered}
\end{equation}
The first condition says that the vertical distance between the
endpoints is comparable to the scale $N$. The slope condition makes
their horizontal distance comparable to $N$ as well, with constants
uniform over $\lambda\in K$; see Figure~\ref{fig:localgeometry}
for a depiction of these conditions. The last condition keeps the crossing
away from both endpoints: the lower portion, from $u$ to row $m$,
has vertical length $m-j_-\ge\beta D$, and the upper portion, from
row $m+1$ to $v$, has vertical length $j_+-m-1\ge\beta D$.
{The subtraction of one is a minor correction coming from splitting at the exit between}
rows $m$ and $m+1$. {Thus,} both portions have length comparable to
$N$.
All constants in this section may depend on $a_0,b_0,\beta,K$.

Since the endpoints and row are fixed in this section, we abbreviate
the \emph{exit-constrained} routed profile by
\begin{equation}\label{eq:routed}
 Z^s(x):=Z_u^{v,\bullet,s}(x,m)
       =T_u^{(x,m),s}+T_{(x,m+1)}^{v,s}.
\end{equation}
{Thus,} $Z^s$ is only shorthand for the profile defined earlier, with
its endpoint and row arguments suppressed.
The two terms use disjoint sets of Brownian rows. Moreover,
\(T_u^{v,s}=\max_xZ^s(x)\), and the geodesic exit at row \(m\) is a
maximiser.

\begin{figure}[tbp]
\centering
\begin{tikzpicture}[x=1.3cm,y=1cm,font=\small]
  \foreach \y/\lab in {0/{j_-},2/{m},2.55/{m+1},4/{j_+}}{
    \draw[guide] (-.15,\y)--(5.65,\y);
    \node[anchor=east] at (-.25,\y) {$\lab$};
  }
  \draw[black!55,dashed] (0,0)--(5,4);
  \node[anchor=west,fill=white,inner sep=2pt] at (3.75,3.05)
    {slope $\lambda$};
  \draw[refineteal,line width=2pt] (.9,2)--(4.1,2);
  \draw[refineteal] (.9,1.91)--(.9,2.09);
  \draw[refineteal] (4.1,1.91)--(4.1,2.09);
  \node[refineteal,above=3pt] at (1.3,2) {$W$};
  \fill[refineteal] (2.5,2) circle (2pt);
  \node[below=5pt,fill=white,inner sep=1pt] at (2.5,2) {$(\chi,m)$};
  \fill (0,0) circle (2pt) node[below right=3pt] {$u=(x_-,j_-)$};
  \fill (5,4) circle (2pt) node[above=3pt] {$v=(x_+,j_+)$};
  \draw[<->] (-1.15,0)--node[left] {$D$}(-1.15,4);
  \draw[<->] (5.9,0)--node[right] {$m-j_-$}(5.9,2);
  \draw[<->] (5.9,2.55)--node[right] {$j_+-m-1$}(5.9,4);
  \draw[<->] (0,-.85)--node[below] {$x_+-x_-=\lambda D$}(5,-.85);
\end{tikzpicture}
\caption{The geometry of the scale conditions in \eqref{eq:localranges}. Both vertical
lengths on the right are at least $\beta D$, and $D$ is comparable to $N$.
The straight line from $u$ to $v$ meets row $m$ at
$\chi=x_-+\lambda(m-j_-)$. The window $W$ in
Lemma~\ref{lem:packing} is centered at this deterministic location.
The gap between rows $m$ and $m+1$ is exaggerated for visibility;
it accounts for the one-row correction in the upper length.}
\label{fig:localgeometry}
\end{figure}
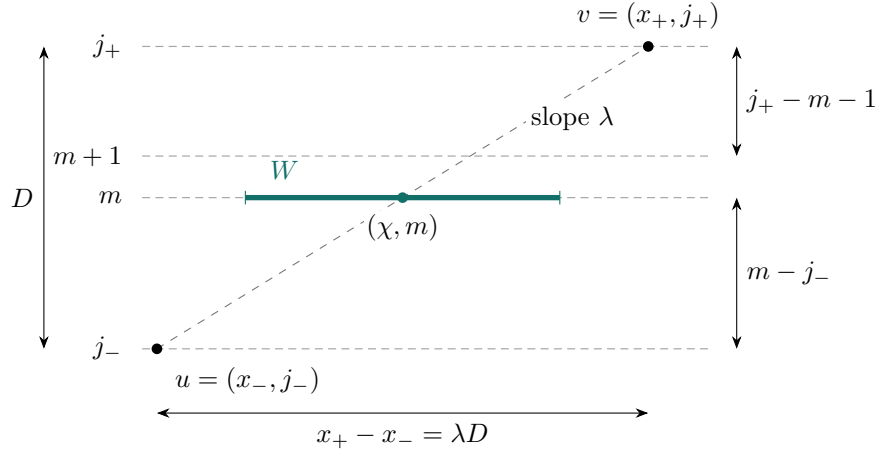

\subsection{An estimate on the number of near-maximisers}

Recall that $\nearmax^\alpha(f;J)$ counts points in $J$ whose values
are within $\alpha$ of $\max_J f$ and whose pairwise distances are
at least $\alpha^2$; see \eqref{eq:NMdefinition}. We estimate this
count directly on the window needed for the {exit cover}, using
Brownian comparison for the routed profile and the corresponding
Brownian near-maximiser bound. The same argument is used in
\cite[Lemma 80]{B25}; only bulk rows are needed here, as opposed to the setting therein.

\begingroup
\begin{lemma}[Near-maximisers on the crossing window]\label{lem:packing}
Fix $C_{\mathrm{win}}>0$ and $Q>1080$.
{Let $\ell_0$ be sufficiently large, depending only on these
constants and the fixed geometric ranges.} For $\ell\ge\ell_0$
and an integer $N\ge\ell^Q$, let $u,v,m$ satisfy \eqref{eq:localranges}, and set
\begin{equation}\label{eq:packingwindow}
 \chi=x_-+\lambda(m-j_-),\qquad
 W=[\chi-C_{\mathrm{win}}\ell^2N^{2/3},\chi+C_{\mathrm{win}}\ell^2N^{2/3}].
\end{equation}
{There are $C,c>0$ such that, uniformly over the}
stated endpoints, scales and $0<\alpha\le N^{1/3}$,
\begin{equation}\label{eq:localpacking}
 \PP\bigl(\nearmax^\alpha(Z^0;W)>\ell^{90}\bigr)
 \le Ce^{-c\ell^{90}}.
\end{equation}
The constants depend only on $Q,C_{\mathrm{win}}$ and the fixed geometric ranges.
\end{lemma}

The window $W$ is wider than the usual transverse scale by the
factor $\ell^2$, so that it can contain all the dynamic exits on
the regularity event. The proof first counts separated near-maximisers
for Brownian motion and then transfers this bound to the routed
profile. The quantitative comparison has a cost depending on the
window width; the count $\ell^{90}$ is chosen so that Brownian
rarity dominates this cost. The cutoff $N\ge\ell^Q$ ensures that
{the comparison is applicable with uniform constants at every scale used later.}

\begin{proof}
{The bulk condition in \eqref{eq:localranges} keeps $\chi$ at
distance at least $cN$ from both ends of $[x_-,x_+]$.
Since $\ell^2N^{-1/3}\to0$ uniformly for $N\ge\ell^Q$,
the window $W$ lies in this interval for large enough $\ell_0$.
Put $a=(m-j_-)/D$, the fraction of the total vertical separation
lying below row $m$. Translation and Brownian scaling give the
normalized routed profile}
\[
 \widehat Z(y)=
 \frac{\lambda^{-1/2}Z^0(\chi+2\lambda D^{2/3}y)-2D}
      {\sqrt2 D^{1/3}}.
\]
{The bulk condition gives $a\ge\beta$ and}
$1-a-D^{-1}\ge\beta$. The window and tolerance become
\[
 [-d,d],\qquad
 d=\frac{C_{\mathrm{win}}\ell^2N^{2/3}}{2\lambda D^{2/3}}\asymp\ell^2,
 \qquad \widetilde\alpha=\frac{\alpha}{\sqrt{2\lambda}D^{1/3}}\le C.
\]
The horizontal scaling factor is the square of the weight scaling
factor, so it preserves the definition of the near-maximiser count:
\[
 \nearmax^\alpha(Z^0;W)
 =\nearmax^{\widetilde\alpha}(\widehat Z;[-d,d]).
\]

Let $B$ be a rate-two Brownian motion on $[-d,d]$, anchored
at its left endpoint, and put $b_\ell=\ell^{90}$.
The Brownian near-maximiser estimate
\cite[Proposition 2.5]{CHH23}, with Brownian scaling, gives
\[
 \PP(\nearmax^{\widetilde\alpha}(B;[-d,d])>b_\ell)
 \le C_0e^{-c_0b_\ell}.
\]
Choose $0<c_1<c_0$ and enlarge this event, under the atomless
Brownian law, to probability exactly $q=e^{-c_1b_\ell}$.
Subtracting the left-endpoint value does not affect the count.
The comparison in \cite[Theorem 1.2]{GH23}\footnote{\label{fn:strongercomparison}
Dauvergne~\cite{Dau24} proves much sharper Brownianity
estimates for the Airy line ensemble, and
\cite[Appendix 2]{B25} provides an adaptation of some of the results therein to BLPP. The comparison in \cite[Theorem 1.2]{GH23} already
suffices for the near-maximiser bound here, so we use it for
simplicity.} applies to this
enlarged event for $\widehat Z-\widehat Z(-d)$.
Indeed, its window is centered at zero, so the linear correction
vanishes, and its probability cutoffs hold uniformly since
\[
 d^{12}=O(\ell^{24})=o(b_\ell),\qquad
 b_\ell=o(D^{1/12}),
\]
where the second assertion uses $Q>12\cdot90$.
The resulting bound is
\begin{align*}
 \PP(\nearmax^\alpha(Z^0;W)>b_\ell)
 &\le C d^6
       \exp\{-c_1b_\ell+C d^7b_\ell^{5/6}\}\\
 &\le C\ell^{12}\exp\{-c_1\ell^{90}+C\ell^{89}\}
 \le Ce^{-c\ell^{90}}.
\end{align*}
All comparison constants are uniform over the fixed bulk and slope
ranges. The strict inequality $90>42\cdot2$ makes the negative
term dominate the comparison loss.
\end{proof}
\par\endgroup

\subsection{Concentration turns dynamic crossings into static witnesses}

A dynamic exit through a unit interval forces the restricted passage
value to equal the unrestricted one at that time. Applying stability
to these two values separately produces a static near-maximum witness
inside the unit interval, and the proposition below records the resulting
cover.
Figure~\ref{fig:crossingwitness} distinguishes the dynamic exit from
the static witness in its unit interval, while
Figure~\ref{fig:localcover} shows the resulting cover of the entire
dynamic exit set.

\begin{figure}[tbp]
\centering
\begin{tikzpicture}[x=1cm,y=1cm,font=\footnotesize]
  \node[anchor=west,font=\small] at (0,4.5)
    {\textbf{(a)} Two optimisations with an exit in $I$};
  \draw[guide] (0,1.65)--(5.25,1.65);
  \draw[guide] (0,2.25)--(5.25,2.25);
  \node[anchor=east] at (0,1.65) {$m$};
  \node[anchor=east] at (0,2.25) {$m+1$};
  \fill[clusterblue!12] (2.18,1.49) rectangle (3.0,1.81);
  \draw[clusterblue,line width=1pt] (2.18,1.65)--(3.0,1.65);
  \node[below=7pt] at (2.59,1.49) {$I$};
  \draw[refineteal,line width=1.3pt]
    (.3,.15)--(.75,.15)--(.75,.8)--(1.55,.8)
    --(1.55,1.65)--(2.38,1.65)--(2.38,2.25)
    --(3.35,2.25)--(3.35,3.0)--(4.3,3.0)--(4.3,3.65)--(4.9,3.65);
  \draw[witnessgold,densely dashed,line width=1.2pt]
    (.3,.15)--(1.05,.15)--(1.05,.8)--(1.88,.8)
    --(1.88,1.65)--(2.82,1.65)--(2.82,2.25)
    --(3.7,2.25)--(3.7,3.0)--(4.6,3.0)--(4.6,3.65)--(4.9,3.65);
  \fill (.3,.15) circle (1.5pt) node[left] {$u$};
  \fill (4.9,3.65) circle (1.5pt) node[above] {$v$};
  \fill[refineteal] (2.38,1.65) circle (2pt);
  \fill[witnessgold] (2.82,1.65) circle (2pt);
  \node[above left,refineteal] at (2.38,1.65) {$x$};
  \node[above right,witnessgold] at (2.82,1.65) {$y$};
  \node[refineteal,anchor=west] at (.15,-.45) {$\Gamma_u^{v,t}$: weight $T_u^{v,t}=F_I^t$};
  \node[witnessgold,anchor=west] at (.15,-.95) {$\xi_I^0$: weight $F_I^0$ in the static environment};
  \node[anchor=west,font=\small] at (6.2,4.5)
    {\textbf{(b)} The static near-maximum witness};
  \draw[->,black!60] (6.4,.25)--(13.1,.25) node[right] {$z$};
  \fill[clusterblue!7] (6.6,2.65) rectangle (12.75,3.25);
  \draw[guide] (6.55,3.25)--(12.8,3.25);
  \draw[guide] (6.55,2.65)--(12.8,2.65);
  \draw[clusterblue,line width=1pt] plot coordinates {
    (6.6,.85)(7.0,1.3)(7.3,1.1)(7.6,1.8)(7.9,1.55)
    (8.3,2.2)(8.55,2.35)(8.8,2.85)(9.0,2.6)(9.2,2.35)
    (9.45,1.75)(9.8,2.1)(10.1,1.75)(10.45,2.45)
    (10.75,2.2)(11.1,2.85)(11.4,2.6)(11.7,3.25)
    (11.95,2.95)(12.2,2.55)(12.5,2.7)(12.75,2.1)};
  \draw[cluster,line width=2pt] (8.25,.25)--(9.08,.25);
  \node[below] at (8.665,.25) {$I$};
  \draw[witnessgold,densely dashed] (8.8,.25)--(8.8,2.85);
  \fill[witnessgold] (8.8,2.85) circle (2pt);
  \node[above left,witnessgold] at (8.8,2.85) {$Z^0(y)=F_I^0$};
  \node[above left] at (12.75,3.25) {$T_u^{v,0}=\max Z^0$};
  \draw[<->] (13.05,2.65)--node[right] {$2A$}(13.05,3.25);
  \node[anchor=west] at (6.45,-.65)
    {$T_u^{v,0}-F_I^0\le |T_u^{v,t}-T_u^{v,0}|+|F_I^t-F_I^0|\le2A$};
\end{tikzpicture}
\caption{A dynamic exit produces a static witness in the same unit
interval. The green geodesic exits at $x\in I$ at time $t$. The gold
staircase maximises the time-zero weight subject to exiting in $I$;
its exit $y$ may differ from $x$. Stability of the unrestricted and
restricted passage values puts $Z^0(y)$ within $2A$ of the static
maximum. The gold staircase need not be the unrestricted static
geodesic. The two panels are schematic.}
\label{fig:crossingwitness}
\end{figure}
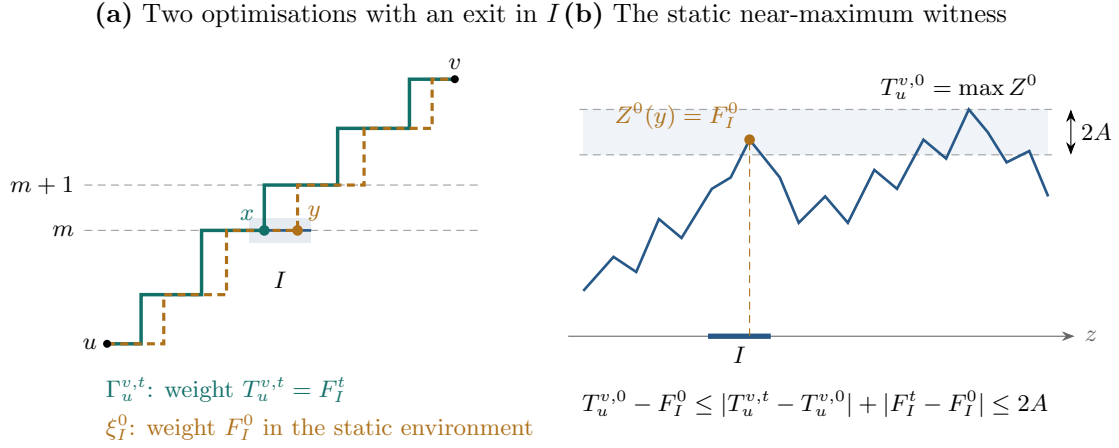

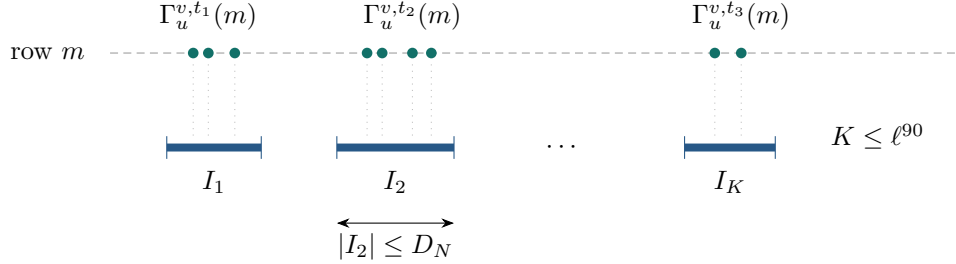
\begin{figure}[tbp]
\centering
\begin{tikzpicture}[x=1cm,y=1cm,font=\small]
  \draw[guide] (.6,1.8)--(11.8,1.8);
  \node[anchor=east] at (.4,1.8) {row $m$};
  \foreach \x in {1.7,1.9,2.25,4.0,4.2,4.6,4.85,8.6,8.95}{
    \fill[refineteal] (\x,1.8) circle (2pt);
    \draw[black!25,dotted] (\x,1.65)--(\x,.65);
  }
  \node[above=5pt] at (1.9,1.8) {$\Gamma_u^{v,t_1}(m)$};
  \node[above=5pt] at (4.6,1.8) {$\Gamma_u^{v,t_2}(m)$};
  \node[above=5pt] at (8.95,1.8) {$\Gamma_u^{v,t_3}(m)$};
  \foreach \a/\b/\lab in {1.35/2.6/{I_1},3.6/5.15/{I_2},8.2/9.4/{I_K}}{
    \draw[clusterblue,line width=3pt] (\a,.55)--(\b,.55);
    \draw[clusterblue] (\a,.4)--(\a,.7);
    \draw[clusterblue] (\b,.4)--(\b,.7);
    \node[below=5pt] at ({(\a+\b)/2},.55) {$\lab$};
  }
  \node at (6.6,.55) {$\cdots$};
  \draw[<->] (3.6,-.45)--node[below] {$|I_2|\le D_N$}(5.15,-.45);
  \node[anchor=west,align=left] at (10.0,.7) {{$K\le\ell^{90}$}};
\end{tikzpicture}
\caption{The conclusion of Proposition~\ref{prop:local}, for endpoints
$u,v$ whose horizontal and vertical separations are of order $N$.
Here $N$ is the local scale, while $n$ is the original scale that
determines {$h=n^{-1/3}\ell^{-4}$}. The dots
represent exits on the same row at different times in $[0,h]$.
{One random family of at most $\ell^{90}$ intervals,
each of length at most $D_N=C_{\mathrm{width}}(Nh\ell^2+\ell^4)$, covers all these
exits simultaneously.}
Different times can use different intervals. The intervals may
contain additional points that the geodesic never visits.}
\label{fig:localcover}
\end{figure}

\begingroup
\begin{proposition}[{Local exit cover}]\label{prop:local}
Fix $Q>1080$ and the geometric ranges in \eqref{eq:localranges}.
Let $\ell\ge\log n$, $\ell^Q\le N\le n$, and
$0<h\le n^{-1/3}\ell^{-4}$. For deterministic $u,v,m$ in those
ranges, define
\[
 E(u,v,m;h)=\{\Gamma_u^{v,t}(m):0\le t\le h\}.
\]
There are $C_{\mathrm{width}},C,c>0$ such that, uniformly in these parameters
and for all sufficiently large $n$, this exit set has a cover
by at most $\ell^{90}$ closed intervals of length at most
\begin{equation}\label{eq:localwidth}
 D_N=C_{\mathrm{width}}(Nh\ell^2+\ell^4)
\end{equation}
except with probability at most $Ce^{-c\ell^2}$.
\end{proposition}

The allowance determines both the number and width of the covering
intervals. {To apply the static near-maximiser estimate, the passage-time
tolerance must be smaller than the local fluctuation scale $N^{1/3}$.
The restrictions $h\le n^{-1/3}\ell^{-4}$ and $N\ge\ell^Q$
ensure this throughout the range of scales: they control,
respectively, the terms $\ell\sqrt{Nh}$ and $\ell^2$ in $A_N$,
as verified in \eqref{eq:alphacheck}.}

\begin{proof}
\par\begingroup
Let $\chi=x_-+\lambda(m-j_-)$, and let $C_{\mathrm{reg}}$
be the constant in Lemma~\ref{lem:regularity}, with $a=a_0$,
$b=b_0$ and slope interval $K$. Put
\[
 \begin{aligned}
 W_0&=[\chi-C_{\mathrm{reg}}\ell^2N^{2/3},
          \chi+C_{\mathrm{reg}}\ell^2N^{2/3}],\\
 W&=[\chi-2C_{\mathrm{reg}}\ell^2N^{2/3},
       \chi+2C_{\mathrm{reg}}\ell^2N^{2/3}].
 \end{aligned}
\]
\par\endgroup
These windows lie in $[x_-,x_+]$ for large $n$, by the bulk
condition and $N\ge\ell^Q$. Let
\[
 \cI(u,v)=\{[i,i+1]\cap[x_-,x_+]:i\in\ZZ,
                          \ [i,i+1]\cap[x_-,x_+]\ne\emptyset\},
 \qquad F_I^t=\max_{x\in I}Z^t(x).
\]
This is a deterministic list with at most $CN+3$ members.
\par\begingroup
Every corresponding path class has horizontal separation
$L=\lambda D\le b_0(\max K)N$, so \eqref{eq:AN} gives
\[
 A_L\le\max\{1,\sqrt{b_0\max K}\}\,A_N.
\]
Apply Lemma~\ref{lem:timeuniform} to the unrestricted class and
these restricted classes. There are at most $CN+4\le n^2$
classes for large $n$, and their common containing rectangle
meets at most $Cn^2$ update blocks. Let $C_1$ be the constant
in \eqref{eq:uniftime} for this application, and define
\[
 {C_1'}=C_1\max\{1,\sqrt{b_0\max K}\},\qquad
 \alpha_N=2{C_1'}A_N.
\]
\par\endgroup
Define
\begin{align*}
 \cE^{\mathrm{stab}}_{u,v,m}
 &=\left\{\sup_{t\le h}|T_u^{v,t}-T_u^{v,0}|\le {C_1'}A_N\right\}
   \cap\bigcap_{I\in\cI(u,v)}
      \left\{\sup_{t\le h}|F_I^t-F_I^0|\le {C_1'}A_N\right\},\\
 \cE^{\mathrm{loc}}_{u,v,m}
 &=\{E(u,v,m;h)\subseteq W_0\},\\
 \cE^{\mathrm{pack}}_{u,v,m}
 &=\{\nearmax^{\alpha_N}(Z^0;W)\le\ell^{90}\}.
\end{align*}
The suprema here and below are over $0\le t\le h$.
{The preceding application of Lemma~\ref{lem:timeuniform}
and the definition of ${C_1'}$ bound the first failure probability
by $Ce^{-c\ell^2}$.}
Lemma~\ref{lem:regularity}, at scale $N$ with allowance $\ell$,
bounds the second by $Ce^{-c\ell^3}$: its restrictions hold
since $\log N\le\log n\le\ell$ and $N\ge\ell^Q$.

For the packing event, the tolerance satisfies
\begin{equation}\label{eq:alphacheck}
 \frac{A_N}{N^{1/3}}
 =\ell N^{1/6}h^{1/2}+\ell^2N^{-1/3}
 \le\ell^{-1}+\ell^{2-Q/3}\longrightarrow0
\end{equation}
uniformly in the stated ranges. The first term is most restrictive
at $N=n$, and the second at the cutoff $N=\ell^Q$.
Thus $\alpha_N\le N^{1/3}$ for large $n$, and
{Lemma~\ref{lem:packing}, with $C_{\mathrm{win}}=2C_{\mathrm{reg}}$,
bounds the third failure probability by $Ce^{-c\ell^{90}}$.} Consequently the event
\[
 \cE^{\mathrm{cross}}_{u,v,m}
 =\cE^{\mathrm{stab}}_{u,v,m}\cap\cE^{\mathrm{loc}}_{u,v,m}
                            \cap\cE^{\mathrm{pack}}_{u,v,m}
\]
satisfies
\begin{equation}\label{eq:crossingtail}
 \PP((\cE^{\mathrm{cross}}_{u,v,m})^c)\le Ce^{-c\ell^2}.
\end{equation}

We construct the cover on this event. For an exit
$x=\Gamma_u^{v,t}(m)$ choose $I\in\cI(u,v)$ containing $x$.
Then $F_I^t=T_u^{v,t}$, so
\begin{equation}\label{eq:witness}
 0\le T_u^{v,0}-F_I^0
   =(T_u^{v,0}-T_u^{v,t})+(F_I^t-F_I^0)\le\alpha_N.
\end{equation}
Choose a maximiser $y\in I$ of $Z^0$ on $I$. Since
$x\in W_0$ and $|x-y|\le1$, we have $y\in W$ for large $n$,
and \eqref{eq:witness} implies
$Z^0(y)\ge T_u^{v,0}-\alpha_N\ge\max_WZ^0-\alpha_N$.
Let $S$ be a maximal $\alpha_N^2$-separated subset of the
compact set $\{y\in W:Z^0(y)\ge\max_WZ^0-\alpha_N\}$.
The packing event gives $|S|\le\ell^{90}$, and maximality gives
a point $z\in S$ within distance $\alpha_N^2$ of each such $y$.
It follows that
\[
 \{[z-\alpha_N^2-1,z+\alpha_N^2+1]:z\in S\}
\]
covers $E(u,v,m;h)$. Its interval lengths are at most
$2\alpha_N^2+2\le C_{\mathrm{width}}(Nh\ell^2+\ell^4)$, after fixing
$C_{\mathrm{width}}$ sufficiently large. This proves the proposition.
\end{proof}

\begin{corollary}[{Exit covers for an independent endpoint list}]
\label{cor:independentcovers}
Fix $n,\ell,Q,h$ and common geometric ranges as in
Proposition~\ref{prop:local}. Let $\mathscr G$ be a {$\sigma$-algebra}
independent of the dynamics and let
\[
 \mathscr L=((N_i,u_i,v_i,m_i))_{i=1}^{M}
\]
be a finite $\mathscr G$-measurable list. Suppose every entry
satisfies $\ell^Q\le N_i\le n$ and {the geometric conditions
\eqref{eq:localranges} with $(u,v,m,N)=(u_i,v_i,m_i,N_i)$.}
Let $\cE_i^{\mathrm{cov}}$ be the event that $E(u_i,v_i,m_i;h)$
has a cover by at most $\ell^{90}$ intervals, each of length
at most $D_{N_i}$. Then, almost surely,
\begin{equation}\label{eq:independentcovers}
 \PP\left(\bigcup_{i=1}^{M}(\cE_i^{\mathrm{cov}})^c
                  \,\middle|\,\mathscr G\right)
 \le CM e^{-c\ell^2}.
\end{equation}
{In particular,} the unconditional failure is at most
$C(\EE M)e^{-c\ell^2}$ whenever $\EE M<\infty$.
\end{corollary}

\begin{proof}
Conditional on $\mathscr G$, the pairs, scales and rows are
deterministic, while the environment retains its original law.
Apply Proposition~\ref{prop:local} to each entry and sum the
failure probabilities. Taking expectations gives the last assertion.
No independence between entries is needed.
\end{proof}
\par\endgroup

\section{Poisson capture of local subpaths}\label{sec:capture}

To refine the cover of Proposition~\ref{prop:local} inside a parent
interval, we cut a shorter segment around the row of interest.
{Its endpoints depend on the environment, so the estimate in
Proposition~\ref{prop:local} for covering exits from a fixed row
does not apply directly: that proposition assumes deterministic endpoints.}

Our aim in this section is to preserve the central part of each such
segment while replacing its endpoints by a pair from a small list
sampled independently of the environment. Corollary~\ref{cor:independentcovers}
then provides {exit covers} for the entire list before we select a
representative for any particular segment. This is the local step that
will be iterated in Section~\ref{sec:multiscale}.

The construction is a local version of the \textit{Poisson capture}
argument of \cite[Section 5]{B25}. Its geometric foundation is the
one-sided volume-accumulation estimate in
\cite[Section 5]{BB23}, developed for exponential LPP and adapted
to BLPP in \cite[Section 11]{B25}. This estimate provides a lower
bound on the volume of endpoints whose geodesics join a given path
sufficiently quickly, and this estimate is crucial for the
\textit{Poisson capture} to work.
We first explain the geometry of cutting, then state the resulting
basin-volume estimate from \cite[Proposition 28]{B25}. {The final lemma combines \textit{Poisson capture} with the covers
of row exits supplied by Corollary~\ref{cor:independentcovers},
on one event that can be used at every stage of the refinement.}

\subsection{Shorter portions of a fixed-endpoint geodesic}
\label{sec:cutting}

We begin with the geodesic $\Gamma_p^{q,t}$ between the original
deterministic endpoints $p,q$, at a fixed dynamical time $t$.
Suppose that $m$ is a bulk row and that the rows
$m-N$ and $m+N$ lie strictly between its endpoint levels. The two exits
\[
 a_t=\bigl(\Gamma_p^{q,t}(m-N),m-N\bigr),\qquad
 b_t=\bigl(\Gamma_p^{q,t}(m+N),m+N\bigr)
\]
delimit a shorter portion of this path. By optimality, that portion
is a geodesic between $a_t$ and $b_t$: a replacement of larger weight
would also improve the original path from $p$ to $q$.

The reason for making this cut is that the shorter geodesic has exactly
the same exit from row $m$ as the original one, while its longitudinal
scale is only $N$. {A cover of this shorter segment's exits
from row $m$, as in Proposition~\ref{prop:local}, would therefore
give a finer cover of the original exits.}
The dynamical horizon remains
{$h=n^{-1/3}\ell^{-4}$}, whereas the critical
time scale for a segment of length $N<n$ is larger:
$N^{-1/3}>n^{-1/3}$. {Thus,} the ratio $h/N^{-1/3}=hN^{1/3}$ becomes
smaller as $N$ decreases. Correspondingly, the stability tolerance
$\sqrt{Nh}$ is smaller relative to the fluctuation scale $N^{1/3}$,
so the same exit must pass a more stringent near-maximum test at the
shorter scale.
We refer the reader to Figure~\ref{fig:cutting} for a depiction of the setting.

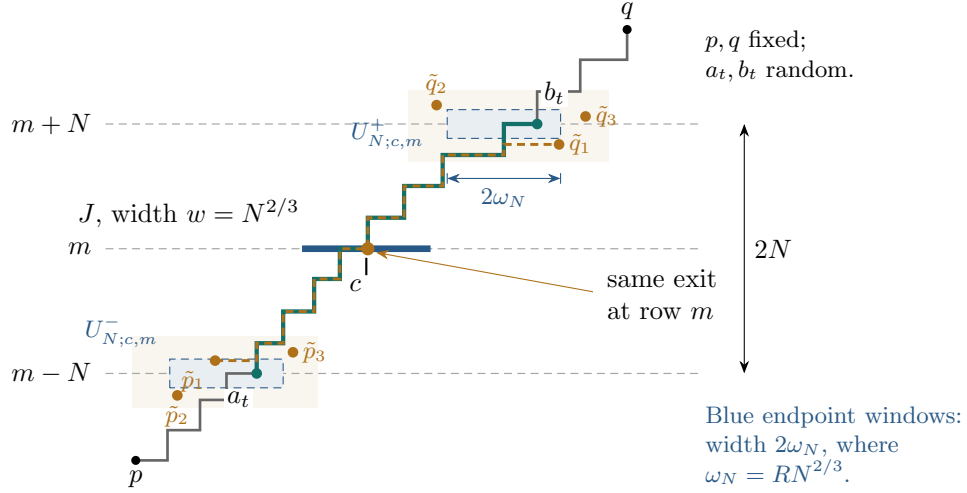
\begin{figure}[tbp]
\centering
\begin{tikzpicture}[x=1cm,y=1cm]
  \fill[witnessgold!7] (.55,.7) rectangle (3.0,1.65);
  \fill[witnessgold!7] (4.2,3.95) rectangle (6.85,4.9);
  \draw[guide] (.2,1.15)--(8.05,1.15);
  \draw[guide] (.2,2.8)--(8.05,2.8);
  \draw[guide] (.2,4.45)--(8.05,4.45);
  \node[anchor=east] at (.15,1.15) {$m-N$};
  \node[anchor=east] at (.15,2.8) {$m$};
  \node[anchor=east] at (.15,4.45) {$m+N$};
  \fill[clusterblue!10] (1.05,.96) rectangle (2.55,1.34);
  \draw[clusterblue,densely dashed] (1.05,.96) rectangle (2.55,1.34);
  \fill[clusterblue!10] (4.72,4.26) rectangle (6.22,4.64);
  \draw[clusterblue,densely dashed] (4.72,4.26) rectangle (6.22,4.64);
  \node[clusterblue,anchor=south east,font=\footnotesize]
    at (1.03,1.37) {$U^-_{N;c,m}$};
  \node[clusterblue,anchor=east,font=\footnotesize]
    at (4.55,4.32) {$U^+_{N;c,m}$};
  \draw[clusterblue,<->] (4.72,3.73)--node[below,font=\footnotesize]
    {$2\omega_N$}(6.22,3.73);
  \draw[clusterblue] (4.72,3.66)--(4.72,3.8);
  \draw[clusterblue] (6.22,3.66)--(6.22,3.8);
  \draw[cluster,line width=2.5pt] (2.8,2.8)--(4.5,2.8);
  \draw[line width=.8pt] (3.65,2.68)--(3.65,2.45);
  \node[anchor=north east,fill=white,inner sep=1pt]
    at (3.62,2.43) {$c$};
  \node[anchor=south east] at (2.9,2.95) {$J$, width $w=N^{2/3}$};
  \draw[black!60,line width=1pt] (.6,0)--(1.02,0)--(1.02,.4)--(1.45,.4)
    --(1.45,.8)--(1.8,.8)--(1.8,1.15)--(2.2,1.15);
  \draw[refineteal,line width=1.6pt] (2.2,1.15)--(2.2,1.55)
    --(2.55,1.55)--(2.55,1.97)--(2.96,1.97)--(2.96,2.4)
    --(3.3,2.4)--(3.3,2.8)--(3.67,2.8)--(3.67,3.21)
    --(4.15,3.21)--(4.15,3.63)--(4.66,3.63)--(4.66,4.04)
    --(5.47,4.04)--(5.47,4.45)--(5.91,4.45);
  \draw[black!60,line width=1pt] (5.91,4.45)--(5.91,4.88)
    --(6.48,4.88)--(6.48,5.3)--(7.1,5.3)--(7.1,5.7);
  \fill (.6,0) circle (1.7pt) node[below] {$p$};
  \fill (7.1,5.7) circle (1.7pt) node[above] {$q$};
  \fill[refineteal] (2.2,1.15) circle (2pt);
  \node[anchor=north east,fill=white,inner sep=1pt] at (2.16,.95) {$a_t$};
  \fill[refineteal] (5.91,4.45) circle (2pt);
  \node[anchor=south west,fill=white,inner sep=1pt] at (5.96,4.67) {$b_t$};
  \fill[witnessgold] (3.67,2.8) circle (2.5pt);
  \draw[->,witnessgold] (6.65,2.23)--(3.78,2.75);
  \node[anchor=west,align=left] at (6.7,2.2) {same exit\\at row $m$};
  \draw[<->] (8.65,1.15)--node[right] {$2N$}(8.65,4.45);
  \node[anchor=west,align=left,font=\footnotesize] at (8.0,5.35)
    {$p,q$ fixed;\\$a_t,b_t$ random.};
  \node[clusterblue,anchor=west,align=left,font=\footnotesize] at (8.0,.2)
    {Blue endpoint windows:\\width $2\omega_N$, where\\{$\omega_N=RN^{2/3}$}.};
  \foreach \x/\y in {1.15/.86,2.68/1.43,1.65/1.32,4.58/4.7,6.55/4.55,6.2/4.18}{
    \fill[witnessgold] (\x,\y) circle (2pt);
  }
  \node[below left,witnessgold,font=\footnotesize] at (1.65,1.32) {$\tilde p_1$};
  \node[right,witnessgold,font=\footnotesize] at (6.2,4.18) {$\tilde q_1$};
  \node[below,witnessgold,font=\footnotesize] at (1.15,.86) {$\tilde p_2$};
  \node[above,witnessgold,font=\footnotesize] at (4.58,4.7) {$\tilde q_2$};
  \node[right,witnessgold,font=\footnotesize] at (2.68,1.43) {$\tilde p_3$};
  \node[right,witnessgold,font=\footnotesize] at (6.55,4.55) {$\tilde q_3$};
  \draw[witnessgold,densely dashed,line width=1.1pt]
    (1.65,1.32)--(2.2,1.32)--(2.2,1.55)
    --(2.55,1.55)--(2.55,1.97)--(2.96,1.97)--(2.96,2.4)
    --(3.3,2.4)--(3.3,2.8)--(3.67,2.8)--(3.67,3.21)
    --(4.15,3.21)--(4.15,3.63)--(4.66,3.63)--(4.66,4.04)
    --(5.47,4.04)--(5.47,4.18)--(6.2,4.18);
\end{tikzpicture}
\caption{Cutting and \textit{Poisson capture}, schematically, for a parent
interval $J$ of width $w=N^{2/3}$. The tick labeled $c$ marks
the fixed midpoint of $J$; the gold dot is the geodesic's exit,
which need not equal $c$. The green subpath from $a_t$ to
$b_t$ has the same exit at row $m$ as the original
geodesic $\Gamma_p^{q,t}$. The blue windows $U^-_{N;c,m}$ and $U^+_{N;c,m}$
contain $a_t$ and $b_t$, respectively. Each has horizontal width
{$2\omega_N=2RN^{2/3}$}. For a fixed parent $J$, these windows remain
fixed as the cut endpoints vary. Gold dots are projections of sampled pairs
$(\tilde p_i,\tilde q_i)$ onto the lower and upper sampling regions;
matching subscripts identify a pair. The dashed gold representative
for pair $1$ agrees with $\Gamma_p^{q,t}$ around row $m$, although its endpoints
differ from $a_t,b_t$. The sampling regions are depicted more precisely
in Figure~\ref{fig:windowsbasins}.}
\label{fig:cutting}
\end{figure}

As time varies, the original endpoints $p,q$ remain fixed but the cut points
$a_t,b_t$ can move. Even at one time they are selected by the
optimising path, and conditioning on them can change the law of the
Brownian environment. We therefore cannot
apply the deterministic-endpoint concentration estimate to them by
conditioning. Moreover, freezing the cut points from time zero does not
force $\Gamma_p^{q,t}$ at later times to pass through them.

Since the cut points are random, we discretize their possible
locations by deterministic windows. For each possible interval $J$
on row $m$, we choose lower and upper windows that contain the cut
points whenever $\Gamma_p^{q,t}(m)\in J$ and the bounds of
Lemma~\ref{lem:regularity} hold. We then sample a Poisson collection of endpoint pairs in
slightly larger regions around those two windows.

The \textit{Poisson capture} statement will show that geodesics between the sampled
pairs collectively contain every exit of $\Gamma_p^{q,t}$ associated with
these windows, at every time in $[0,h]$. The cloud is independent of
the environment, so Corollary~\ref{cor:independentcovers} gives
{exit covers} for all its pairs simultaneously. Taking the union
of those covers supplies a finer cover of the actual exits in $J$.
The representative pair may depend on the time and on the original geodesic $\Gamma_p^{q,t}$; the independently sampled list is fixed throughout.
The next two subsections give the windows and the \textit{Poisson capture} statement.

\subsection{{Local windows and sampling regions}}\label{sec:windowsbasins}

Fix a compact slope interval $K\subset(0,\infty)$, a slope
$\lambda\in K$, an integer scale $N\ge1$, and a center
$(c,m)\in\ZZ_{\RR}$.
{Let $R\ge1$ be a transverse allowance.
Lemma~\ref{lem:capture} specifies its permitted range.}
Positive decay constants are denoted by $c_\star$
to distinguish them from the horizontal center $c$. Write
\[
 \omega_N={RN^{2/3}},\qquad
 z_{c,m}^{\lambda}(x,j)=x-c-\lambda(j-m).
\]
Figure~\ref{fig:windowsbasins} illustrates the endpoint windows and
the larger regions in which we sample the Poisson pairs.
The endpoint windows, written in the original coordinates $(x,j)$, are
\begin{align}
 U^-_{N;c,m}
 &=\{(x,j)\in\ZZ_{\RR}:|j-(m-N)|\le N/32,\
              |x-c-\lambda(j-m)|\le \omega_N\},\nonumber\\
 U^+_{N;c,m}
 &=\{(x,j)\in\ZZ_{\RR}:|j-(m+N)|\le N/32,\
              |x-c-\lambda(j-m)|\le \omega_N\}.
 \label{eq:localwindows}
\end{align}
{Thus,} their centers are $(c-\lambda N,m-N)$ and
$(c+\lambda N,m+N)$. The corresponding sampling regions are
\begingroup
\begin{align}
 D^-_{N;c,m}
 &=\{(x,j)\in\ZZ_{\RR}:|j-(m-N)|\le N/8,\
                     |x-c-\lambda(j-m)|\le3\omega_N\},\nonumber\\
 D^+_{N;c,m}
 &=\{(x,j)\in\ZZ_{\RR}:|j-(m+N)|\le N/8,\
                     |x-c-\lambda(j-m)|\le3\omega_N\}.
 \label{eq:outerslabs}
\end{align}
\par\endgroup
The dependence of these sets on the fixed parameters {$\lambda,R$}
is suppressed; their dependence on the scale and center is displayed.

The sets $U^\pm_{N;c,m}$ contain the random cut endpoints. The larger
sets $D^\pm_{N;c,m}$ are where the endpoints of the independent
representatives are sampled. In the transverse coordinate
$z=z_{c,m}^{\lambda}(x,j)$, all four sets are rectangles, centered on
$z=0$; Figure~\ref{fig:windowsbasins} displays their different widths
and level ranges. {Each atom of the Poisson process is an
ordered pair $(\tilde p,\tilde q)$, with $\tilde p\in D^-_{N;c,m}$
and $\tilde q\in D^+_{N;c,m}$. It specifies both endpoints of one
representative geodesic $\Gamma_{\tilde p}^{\tilde q,t}$; the two
endpoints may each lie on any integer row allowed by their respective
sampling regions.}

\begin{figure}[tbp]
\centering
\begin{tikzpicture}[x=1.2cm,y=2.5cm,font=\footnotesize]
  \draw[->] (-3.6,0)--(3.65,0) node[right] {$z/\omega_N$};
  \draw[->] (0,-1.32)--(0,1.35) node[above] {$(j-m)/N$};
  \foreach \sgn in {-1,1}{
    \fill[witnessgold!12] (-3,{\sgn-.125}) rectangle (3,{\sgn+.125});
    \draw[witnessgold] (-3,{\sgn-.125}) rectangle (3,{\sgn+.125});
    \fill[clusterblue!25] (-1,{\sgn-.03125}) rectangle (1,{\sgn+.03125});
    \draw[clusterblue,line width=1pt] (-1,{\sgn-.03125}) rectangle (1,{\sgn+.03125});
  }
  \node[clusterblue,below=3pt] at (.7,-1.04) {$U^-_{N;c,m}$};
  \node[clusterblue,above=3pt] at (.7,1.04) {$U^+_{N;c,m}$};
  \node[witnessgold,below=3pt] at (2.5,-1.125) {$D^-_{N;c,m}$};
  \node[witnessgold,above=3pt] at (2.5,1.125) {$D^+_{N;c,m}$};
  \foreach \x in {-3,-1,1,3}{
    \draw (\x,-.025)--(\x,.025);
    \node[below=2pt] at (\x,0) {$\x$};
  }
  \node[left] at (-3.05,.875) {$7/8$};
  \node[left] at (-3.05,1.125) {$9/8$};
  \node[left] at (-3.05,-.875) {$-7/8$};
  \node[left] at (-3.05,-1.125) {$-9/8$};
  \fill[witnessgold] (-2.1,-.95) circle (2pt) node[left] {$\tilde p_1$};
  \fill[witnessgold] (1.9,1.03) circle (2pt) node[right] {$\tilde q_1$};
  \fill[witnessgold] (2.3,-1.05) circle (2pt) node[right] {$\tilde p_2$};
  \fill[witnessgold] (-1.8,.93) circle (2pt) node[left] {$\tilde q_2$};
\end{tikzpicture}
\caption{The local windows in coordinates transverse to the line of
slope $\lambda$ through $(c,m)$. The blue windows have horizontal
half-width $\omega_N$ and vertical half-height $N/32$, centered at levels
$m-N$ and $m+N$. The gold sampling regions have horizontal half-width
$3\omega_N$ and vertical half-height $N/8$. Matching subscripts on the gold
points indicate sampled endpoint pairs. The continuous rectangles
depict the defining inequalities; their BLPP points lie on integer
rows.}
\label{fig:windowsbasins}
\end{figure}
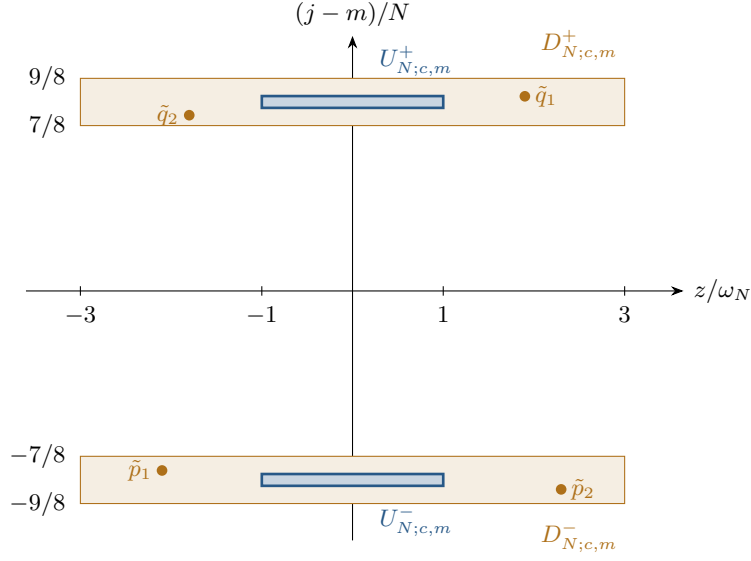

\begingroup
For a Borel set $E\subseteq\ZZ_{\RR}$, define its horizontal
volume by
\[
 |E|_{\hor}=\sum_{j\in\ZZ}\int_{\RR}\ind\{(x,j)\in E\}\,dx.
\]
For a Borel set $H\subseteq(\ZZ_{\RR})^2$ of endpoint pairs, define
\[
 |H|_{\hor}
 =\sum_{j,k\in\ZZ}\int_{\RR^2}
       \ind\{((x,j),(y,k))\in H\}\,dx\,dy.
\]
Thus, $|H|_{\hor}$ is the two-dimensional Lebesgue measure of the
horizontal endpoint coordinates, summed over all pairs of rows.
\par\endgroup

\subsection{A simultaneous \textit{Poisson capture} lemma}

For each fixed box $(N,c,m)$, we use one independent cloud of endpoint
pairs. We first prove that this cloud captures every geodesic between
that box's endpoint windows, throughout the time interval. A later
lemma will make these local statements simultaneous over the different
boxes used in the refinement.

\begingroup
\begin{lemma}[Local Poisson capture with an adjustable window]
\label{lem:capture}
There exist $\bar\eta\in(0,1/10)$ and $R_0,N_1,C,c_\star>0$,
depending only on the fixed compact slope interval $K$, such that
the following holds for every $N\ge N_1$, $R_0\le R\le N^{\bar\eta}$,
$\lambda\in K$ and center $(c,m)$. Use the windows and sampling
regions in \eqref{eq:localwindows}--\eqref{eq:outerslabs}, with
$\omega_N=RN^{2/3}$. Independently of the entire dynamical environment,
sample a Poisson cloud $\cQ_{N;c,m}$ on
$D^-_{N;c,m}\times D^+_{N;c,m}$ with intensity
\[
 \iota_N=N^{-10/3}R^4
\]
relative to $|\cdot|_{\hor}$.
Let $\cE^{\mathrm{cap}}_{N;c,m}(R)$ be the event that for every
$t\in[0,1]$, every $a\in U^-_{N;c,m}$, $b\in U^+_{N;c,m}$,
and every geodesic $\Gamma:a\to b$ at time $t$, there is
$(\tilde p,\tilde q)\in\cQ_{N;c,m}$ such that
\[
 \Gamma_{\tilde p}^{\tilde q,t}\cap[m-N/2,m+N/2]_{\RR}
 =\Gamma\cap[m-N/2,m+N/2]_{\RR}.
\]
Then
\begin{equation}\label{eq:capturefail}
 \PP((\cE^{\mathrm{cap}}_{N;c,m}(R))^c)
 \le CN^{11}e^{-c_\star R^{3/11}}
\end{equation}
and
\begin{equation}\label{eq:cloudmean}
 \EE|\cQ_{N;c,m}|\le CR^6.
\end{equation}
The sampled geodesics are almost surely unique at every time in
$[0,1]$. All constants are uniform in $N,R,c,m,\lambda$ in the
stated ranges.
\end{lemma}

We call $\cE^{\mathrm{cap}}_{N;c,m}(R)$ the capture event
for the box $(N,c,m)$. Note that we deliberately keep the polynomial union factor in \eqref{eq:capturefail}
visible. {The estimate is uniform in $R_0\le R\le N^{\bar\eta}$,
so it applies, for example, to $R=(\log N)^a$ for any fixed $a>0$
and to $R=N^\eta$ for any fixed $0<\eta\le\bar\eta$, once $N$
is large enough. The constants do not depend on the choice of $R$.}

\begin{proof}
We follow the proof of \cite[Proposition 30]{B25}, using its
basin-volume input, \cite[Proposition 28]{B25}. We give the
parameter substitutions and the counts needed to retain the dependence
on $R$. The four steps below are the endpoint-mesh reduction, the
basin-volume estimate, the passage to all dynamical configurations,
and Poisson sampling. The factor $N^{11}$ in
\eqref{eq:capturefail} comes from three counts: at most $CN^3$
auxiliary mesh points in one basin argument, $CN^6$ endpoint pairs,
and $CN^2$ configurations in expectation. For one basin, both the
volume failure exponent and the conditional Poisson failure exponent
depend only on $R$.

Throughout the proof the box $(N,c,m)$ and the slope $\lambda$ are
fixed. In the notation of \cite[Section 5]{B25}, we take the
longitudinal scale to be $N$ and $\gamma=1/4$, and replace the
reference line of slope $1$ by $x=c+\lambda(j-m)$. The row
half-widths $N/32$, $N/16$ and $N/8$ below are precisely
$\gamma N/8$, $\gamma N/4$ and $\gamma N/2$ in that proof.
{Translation and Brownian scaling give static estimates that
are uniform over the fixed compact slope interval $K$.}

\par\smallskip\noindent\textbf{Step 1. Confinement and a finite endpoint mesh.}\quad
This is the reduction of \cite[Lemma 31]{B25}. Put
\[
 \begin{aligned}
 \mathcal T_N(a)&=\{(x,j)\in\RR^2:
                 |x-c-\lambda(j-m)|\le aRN^{2/3}\},\\
 \mathcal G_N^\pm&=
 \{(x,j)\in N^{-1}\ZZ\times\ZZ:
       |j-(m\pm N)|\le N/16\}\cap\mathcal T_N(2).
 \end{aligned}
\]
\begingroup
Dropping only the horizontal lattice constraint $x\in N^{-1}\ZZ$
from $\mathcal G_N^\pm$, while retaining integer rows, gives a
window containing $U^\pm_{N;c,m}$ and contained in $D^\pm_{N;c,m}$.
Indeed, its row half-width is $N/16$ and its horizontal half-width
is $2RN^{2/3}$, between the respective widths of those two windows.

Let $\mathcal F_N^t$ be the event that every geodesic from
$U^-_{N;c,m}$ to $U^+_{N;c,m}$ at time $t$ lies in
$\mathcal T_N(2)$. For each pair of endpoint rows, planarity
bounds all these geodesics between the geodesic joining the two
leftmost endpoints and the one joining the two rightmost endpoints.
The straight lines joining these extreme endpoints lie at transverse
coordinates $-RN^{2/3}$ and $RN^{2/3}$, respectively.
Thus a bounding geodesic must deviate from its own straight line
by at least $RN^{2/3}$ to leave $\mathcal T_N(2)$.
A transversal fluctuation estimate
(e.g.\ \cite[Corollary 1.5]{GH23}) bounds each such probability
by $Ce^{-c_\star R^3}$. Each endpoint window has $O(N)$ rows,
and there are only two bounding geodesics for each row pair.
The union bound therefore gives
\[
 \PP((\mathcal F_N^t)^c)\le CN^2e^{-c_\star R^3}.
\]
To check the range of the cited estimate \cite[Corollary 1.5]{GH23}, let $L$ be the row
separation of a bounding pair. Here $L\asymp N$, so a deviation
of $RN^{2/3}$ corresponds, after Brownian scaling, to a normalized
deviation parameter $r\asymp R$. The corollary applies for
$r_0\le r\le L^{1/10}$. Taking $R_0$ large enough and
$\bar\eta<1/10$ ensures these inequalities uniformly for
$R_0\le R\le N^{\bar\eta}$ and sufficiently large $N$.

On $\mathcal F_N^t$, every such geodesic $\Gamma:a\to b$ meets
both $\mathcal G_N^-$ and $\mathcal G_N^+$, and this can be seen by the mesh
argument from \cite[Lemma 31]{B25}, which we now discuss. Write $a=(x_a,j_a)$ and
take $j_*=\lfloor m-N+N/16\rfloor$. If $\Gamma$ did not meet
$\mathcal G_N^-$, each horizontal portion between $j_a$ and
$j_*$ would have length at most $N^{-1}$, and hence
$\Gamma(j_*)-x_a\le C$.
On the other hand, $a\in U^-_{N;c,m}$ gives the first inequality
below, and $\Gamma\subseteq\mathcal T_N(2)$ gives the second:
\[
 \begin{aligned}
 x_a&\le c+\lambda(j_a-m)+RN^{2/3},\\
 \Gamma(j_*)&\ge c+\lambda(j_*-m)-2RN^{2/3}.
 \end{aligned}
\]
Subtracting these bounds and using
$j_*-j_a\ge N/32-1$, with $\lambda_-:=\min K>0$, yields
\[
 \Gamma(j_*)-x_a
 \ge\lambda(j_*-j_a)-3RN^{2/3}
 \ge\lambda_-(N/32-1)-3RN^{2/3}\longrightarrow\infty.
\]
This contradicts the bound $\Gamma(j_*)-x_a\le C$, uniformly
in the stated parameter range since $RN^{2/3}=o(N)$.
The same argument read backwards from $b\in U^+_{N;c,m}$ gives
$\Gamma\cap\mathcal G_N^+\ne\varnothing$. {We have therefore
proved that, on $\mathcal F_N^t$, for every $a\in U^-_{N;c,m}$,
$b\in U^+_{N;c,m}$ and every geodesic $\Gamma:a\to b$ at time $t$,}
\[
 \Gamma\cap\mathcal G_N^-\ne\varnothing,
 \qquad \Gamma\cap\mathcal G_N^+\ne\varnothing.
\]
Choose $p'\in\Gamma\cap\mathcal G_N^-$ and
$q'\in\Gamma\cap\mathcal G_N^+$. Almost surely, geodesics
are unique simultaneously for all pairs in this deterministic
mesh and all $t\in[0,1]$, by the observation after
Lemma~\ref{lem:clocks}. The subpath between $p'$ and $q'$ is
therefore $\Gamma_{p'}^{q',t}$.
Finally, each mesh point lies in a common rectangle with side lengths $O_K(N)$.
It has $O(N)$ rows and $O(N^2)$ horizontal mesh points per row,
so there are at most $CN^3$ points in each mesh and at most
$CN^6$ mesh point pairs $(p',q')$.
\par\endgroup

\begingroup
\par\smallskip\noindent\textbf{Step 2. A static basin estimate.}\quad
We first state the analogue of \cite[Proposition 28]{B25}
needed here. Put
\[
 S_0=[m-N/2,m+N/2]_{\RR},\qquad
 S_1=(m-3N/4,m+3N/4)_{\RR}.
\]
Following the definition at the start of \cite[Section 5.1]{B25},
let $\mathcal B^t(\Gamma)$ consist of pairs in
$D^-_{N;c,m}\times D^+_{N;c,m}$ for which some geodesic
$\widetilde\Gamma$ at time $t$ satisfies
$\widetilde\Gamma\cap S_1=\Gamma\cap S_1$.
Since $S_0\subset S_1$, membership in this basin gives the
agreement required in the statement, including the boundary rows
of $S_0$.

\emph{Static basin estimate.} We use the following form of
\cite[Proposition 28, (91)]{B25}, allowing the transverse
allowance $R$ to vary in the stated range.  After choosing $\bar\eta>0$
sufficiently small and $R_0,N_1$ sufficiently large, there are
constants $C,c_\star>0$ such that for every
$N\ge N_1$, $R_0\le R\le N^{\bar\eta}$ and deterministic
$p'\in\mathcal G_N^-$, $q'\in\mathcal G_N^+$,
\begingroup
\begin{equation}\label{eq:staticbasin}
 \PP\!\left(
  |\mathcal B^0(\Gamma_{p'}^{q',0})|_{\hor}
       <c_\star R^{-2}N^{10/3}\right)
 \le CN^3e^{-c_\star R^{3/11}}.
\end{equation}
\par\endgroup
The constants are uniform in the center $(c,m)$ and
$\lambda\in K$. {By stationarity,} the same estimate holds at
every fixed time $t$.

To obtain this estimate, follow the proof of
\cite[Proposition 28]{B25} with $n$ replaced by $N$ and
$n^{-\delta}$ replaced by $R^{-1}$. Its one-sided input,
\cite[Proposition 29]{B25}, is the BLPP adaptation of the
volume estimate in \cite[Section 5.6]{BB23} discussed in
\cite[Appendix 5]{B25}. The two applications of that input used to obtain
\cite[(94) and (97)]{B25}, respectively,
use $\varepsilon=R^{-1}$ and $2R^{-1}$; each has failure
bound $Ce^{-c\varepsilon^{-3/11}}$ and is valid when
$N\varepsilon^{1/\delta_0}\ge K_0$. Thus both applications
are uniform in our range if $\bar\eta<\delta_0$.

\begingroup
The factor $N^3$ in \eqref{eq:staticbasin} comes from the union
defining the event $\disvol$ in \cite[(96)]{B25}. There are
$O(N)$ possible rows. On each row, the auxiliary points lie
in a horizontal interval of length $O(R^{2/11}N^{2/3})$ and
have spacing $N^{-1}$. Hence their number is at most
\[
 CN\bigl(R^{2/11}N^{2/3}\times N+1\bigr)
 \le CN^{8/3}R^{2/11}\le CN^3,
\]
using $R\le N^{\bar\eta}$ and $\bar\eta<1/10$.
In \cite{B25}, where $R=N^\delta$ for fixed $\delta>0$,
this polynomial factor is absorbed into the stretched
exponential. We retain it because $R$ may grow more slowly here.
In the notation of that proof, the two one-sided volumes are
$|\underline V_N^*(p',q')|_{\hor}$ and
$|\overline V_N^*((b_j,j),q')|_{\hor}$. Each is bounded below
by a constant multiple of $R^{-1}N^{5/3}$; their product
gives the basin-volume threshold. Rounding to the $N^{-1}$
mesh loses only $O(1)$ in the first one-sided volume.
\par\endgroup

\begingroup
\begingroup
In \cite[Proposition 30]{B25}, the Poisson cloud is sampled
on $(\ZZ_{\RR})^2$. To justify restricting it here to
$D^-_{N;c,m}\times D^+_{N;c,m}$, we check that the basin constructed
in \cite[Proposition 28]{B25} lies in this product, and hence
provides the volume lower bound in \eqref{eq:staticbasin}. The basin in
\cite[Section 5.1]{B25} is already restricted to a horizontal
neighbourhood of its reference geodesic. In our parameters, that
neighbourhood has radius $CR^{-4/11}N^{2/3}$, while the reference
geodesic is within $CR^{2/11}N^{2/3}$ of the chord joining $p',q'$,
except with probability $Ce^{-cR^{6/11}}$. This failure is
absorbed in \eqref{eq:staticbasin}. Since $p',q'$ lie in
$\mathcal T_N(2)$, their chord has transverse coordinate of
absolute value at most $2RN^{2/3}$. Thus each basin endpoint
$(x,j)$ satisfies
\[
 |z_{c,m}^{\lambda}(x,j)|
 \le 2RN^{2/3}+C(R^{2/11}+R^{-4/11})N^{2/3}
 \le 3RN^{2/3}
\]
for sufficiently large $R_0$. The lower and upper basin endpoints
have rows in $[m-9N/8,m-7N/8]$ and
$[m+7N/8,m+9N/8]$, respectively, by the same construction.
These are precisely the bounds defining $D^-_{N;c,m}$ and
$D^+_{N;c,m}$ in \eqref{eq:outerslabs}. No additional union
bound is needed for this restriction.
\par\endgroup
\par\endgroup

\par\begingroup
For a fixed time $t$, define the event
\[
 \mathcal V_N^t=
 \bigcap_{\substack{p'\in\mathcal G_N^-\\q'\in\mathcal G_N^+}}
 \left\{|\mathcal B^t(\Gamma_{p'}^{q',t})|_{\hor}
                  \ge c_\star R^{-2}N^{10/3}\right\}.
\]
\par\endgroup
A union over the at most $CN^6$ pairs counted in Step 1 gives
\[
 \PP((\mathcal V_N^t)^c)\le CN^9e^{-c_\star R^{3/11}}.
\]
On $\mathcal F_N^t$, the pair selected in Step 1 has its
subpath inside $\Gamma$, so
\[
 \mathcal B^t(\Gamma_{p'}^{q',t})\subseteq\mathcal B^t(\Gamma).
\]
{This is the first basin inclusion in equation (106) of the statement of
\cite[Lemma 31]{B25}.}
\par\endgroup

\par\smallskip\noindent\textbf{Step 3. All dynamical configurations.}\quad
 Following \cite[Lemmas 32--34]{B25}, we now make the
basin-volume lower bound in \eqref{eq:staticbasin} hold
simultaneously for all mesh pairs and all dynamical times in
$[0,1]$. Fix a deterministic
rectangle containing all the endpoint regions and auxiliary
mesh points above. By monotonicity it also contains all paths
determining the events and basins. Its side lengths are $O_K(N)$,
so it meets at most $CN^2$ unit update blocks. Let
$0=\tau_0<\tau_1<\cdots<\tau_{K_N}\le1$ be their ring times,
with time zero adjoined, and let $\mathscr T_N$ be their clock
{$\sigma$-algebra}. Then $\EE(K_N+1)\le CN^2$.
Define the environmental event
\[
 \mathcal A_N=\bigcap_{i=0}^{K_N}
           (\mathcal F_N^{\tau_i}\cap\mathcal V_N^{\tau_i}).
\]
Conditional on $\mathscr T_N$, each configuration has the static
law by Lemma~\ref{lem:clocks}. Thus
\[
 \PP(\mathcal A_N^c\mid\mathscr T_N)
 \le C(K_N+1)N^9e^{-c_\star R^{3/11}},\qquad
 \PP(\mathcal A_N^c)\le CN^{11}e^{-c_\star R^{3/11}}.
\]
Every configuration during $[0,1]$ is represented in this list.
Using the expected number of rings suffices here; no additional
tail event for $K_N$ is needed.

\par\smallskip\noindent\textbf{Step 4. The cloud hits the mesh-pair basins.}\quad
 The independent Poisson cloud in the lemma has intensity
$\iota_N=N^{-10/3}R^4$, chosen so that, with high probability,
at least one sampled pair belongs to each mesh-pair basin in
every dynamical configuration. We verify this using the
conditional Poisson calculation in \cite[Proposition 30]{B25}. Let $\mathscr E_N$ denote the
{$\sigma$-algebra} of the entire dynamical environment, including its
clocks. On $\mathcal A_N$, for each mesh pair and each
$i\in\{0,\ldots,K_N\}$,
\[
 \begin{aligned}
 &\PP\!\left(\cQ_{N;c,m}\cap
    \mathcal B^{\tau_i}(\Gamma_{p'}^{q',\tau_i})=\varnothing
       \mid\mathscr E_N\right)\\
 &\qquad=\exp\!\left\{-\iota_N
       |\mathcal B^{\tau_i}(\Gamma_{p'}^{q',\tau_i})|_{\hor}\right\}
 \le e^{-c_\star R^2}.
 \end{aligned}
\]
Here $N^{-10/3}R^4\times c_\star R^{-2}N^{10/3}
=c_\star R^2$, so the powers of $N$ cancel in the exponent.
Let $\mathcal H_N$ be the event that the cloud hits every basin
in this finite list. A union over at most $CN^6(K_N+1)$ basins
and then expectation give
\[
 \PP(\mathcal A_N\cap\mathcal H_N^c)
 \le CN^6\EE(K_N+1)e^{-c_\star R^2}
 \le CN^8e^{-c_\star R^2}.
\]
On $\mathcal A_N\cap\mathcal H_N$, every original geodesic
has a mesh subpath from Step 1, and its basin contains the basin
of that subpath by Step 2. The sampled representative therefore
agrees with the original geodesic on $S_1$, and hence on $S_0$.
This proves
\[
 \PP((\cE^{\mathrm{cap}}_{N;c,m}(R))^c)
 \le CN^{11}e^{-c_\star R^{3/11}}+CN^8e^{-c_\star R^2},
\]
Since $R\ge1$ and $3/11<2$, the second term is
absorbed into the first after increasing $C$, proving
\eqref{eq:capturefail}. Conditioning on the
independent finite cloud and applying the all-time uniqueness
observation after Lemma~\ref{lem:clocks} gives uniqueness for
every sampled pair. Finally, each $D$ region has $O(N)$ rows
of horizontal length $6RN^{2/3}$, and therefore
\[
 \EE|\cQ_{N;c,m}|
 =N^{-10/3}R^4|D^-_{N;c,m}|_{\hor}|D^+_{N;c,m}|_{\hor}
 \le N^{-10/3}R^4(CRN^{5/3})^2\le CR^6.
\]
This proves \eqref{eq:cloudmean}.
\end{proof}

\begin{remark}
Taking $R=N^\eta$ for a fixed sufficiently small $\eta>0$
recovers the power-law capture estimate used in
\cite[Proposition 30]{B25}, with $\delta=\eta$ and $\nu=2\eta$
in that paper. Keeping $R$ explicit also permits logarithmic
allowances. This is the same basin and Poisson argument with its
parameter dependence retained.
\end{remark}
\par\endgroup

\begingroup
\subsection{A finer exit cover inside every possible parent interval}
\label{sec:refinement}

Fix deterministic endpoints $p=(x_0,j_0)\le q=(x_1,j_1)$ with
$an\le D:=j_1-j_0\le bn$, slope $\lambda\in K$, and bulk fraction
$\beta\in(0,1/2)$. Define
\begin{equation}\label{eq:initialscale}
 N_0=\left\lfloor\tfrac12\min\{n,\beta D\}\right\rfloor.
\end{equation}
Then $N_0\asymp n$. For every bulk row $m$ and
$1\le N\le N_0$, the cut rows satisfy
$j_0<m-N<m+N<j_1$.

\begingroup
Write $K=[\lambda_-,\lambda_+]$ and set
$K'=[\lambda_-/2,2\lambda_+]$. We use this fixed larger
interval because sampled endpoint pairs can have slopes slightly
different from $\lambda$, including when $\lambda$ is an endpoint
of $K$. Let $\bar\eta$ be supplied by Lemma~\ref{lem:capture}
for $K'$. Fix
\par\endgroup
\begin{equation}\label{eq:allowanceparameters}
 Q>\max\{1080,12/\bar\eta\},\qquad
 \log n\le\ell\le n^{1/(2Q)},\qquad
 h=n^{-1/3}\ell^{-4},\qquad R=\ell^{12}.
\end{equation}
The exponent $Q$ is fixed; the allowance $\ell$ may vary with $n$.
At every scale $N\ge\ell^Q$, we have $R\le N^{\bar\eta}$.
\begingroup
The same $\ell$, $h$ and $R$ are used throughout the refinement;
only the local scale $N$ and the positions of the windows change.

Our goal is to prepare capture clouds and {exit covers} at every
possible parent interval and every scale in a deterministic list,
before selecting intervals from the exits actually visited by the
geodesic. We will obtain one event on which all these estimates,
together with the geodesic regularity estimate, hold simultaneously.
On that event, the multiscale argument can follow the visited
intervals and count their descendants using covers already
available there. This preparation is by union bounds; it requires
no independence between successive generations.
For each integer scale $N$, define the deterministic mesh
\par\endgroup
\begin{equation}\label{eq:deterministicmesh}
 w_N=N^{2/3},\qquad J_{N,k}=[kw_N,(k+1)w_N],\qquad
 c_{N,k}=(k+\tfrac12)w_N\quad(k\in\ZZ).
\end{equation}
Let $\mathcal W_n=[x_0,x_1]$, and put
\begin{equation}\label{eq:boxindices}
 \cI_n(N)=\{k:J_{N,k}\cap\mathcal W_n\ne\emptyset\},\qquad
 \cM_n=\ZZ\cap[j_0+\beta D,j_1-\beta D].
\end{equation}
For a deterministic scale list $\mathcal S_n$, set
\[
 \mathfrak B_n(\mathcal S_n)
 =\{(N,c_{N,k},m):N\in\mathcal S_n,\ k\in\cI_n(N),\ m\in\cM_n\}.
\]
For each such box sample a cloud as in Lemma~\ref{lem:capture},
with allowance $R=\ell^{12}$. The clouds are mutually independent
and jointly independent of the dynamics. Each cloud serves the
capture statement at its own box. Define the actual exit sets
\begin{equation}\label{eq:rootexitset}
 E_m=\{\Gamma_p^{q,t}(m):0\le t\le h\},\qquad E_m(J)=E_m\cap J.
\end{equation}

\begin{lemma}[Finer exit covers, simultaneous over parent intervals]
\label{lem:refinement}
Use the preceding geometry and parameters. Let
$\mathcal S_n\subseteq\ZZ\cap[\ell^Q,N_0]$ be deterministic
with $|\mathcal S_n|\le\log n$. There are $C,c>0$ and an event
$\cE_n^{\mathrm{ref}}$ such that
\begin{equation}\label{eq:refinementprob}
 \PP((\cE_n^{\mathrm{ref}})^c)\le Ce^{-c\ell^2}.
\end{equation}
{On this event, the geodesics $\Gamma_p^{q,t}$ satisfy
\eqref{eq:rootTF}--\eqref{eq:rowlength} with $N=n$ and
$P=C_{\mathrm{reg}}\ell^2$, simultaneously for every $t\in[0,1]$
and all row indices and increments specified in
Lemma~\ref{lem:regularity}.}
Moreover, for every $(N,c_{N,k},m)\in\mathfrak B_n(\mathcal S_n)$,
there is a finite family $\mathcal C(N,c_{N,k},m)$ of closed
intervals, each of length at most $D_N=C_{\mathrm{width}}(Nh\ell^2+\ell^4)$,
such that
\begin{equation}\label{eq:refinementcover}
 E_m(J_{N,k})\subseteq\bigcup_{I\in\mathcal C(N,c_{N,k},m)}I,
 \qquad |\mathcal C(N,c_{N,k},m)|\le\ell^{163}.
\end{equation}
The constants and the lower threshold on $n$ are uniform over
$\ell$, the deterministic scale lists and endpoint pairs in the
stated ranges.
\end{lemma}

There are two stages in the proof. We first obtain capture and
{exit covers} for every prepared box. We then use geodesic
regularity to place the cut endpoints of each actual path in the
appropriate box, where those covers are already available. Thus
the estimates apply even though the box visited by the path was
not known when the clouds were sampled.

\begin{proof}
Write $\mathfrak B_n=\mathfrak B_n(\mathcal S_n)$ and
$\cR_n=\cE^{\mathrm{reg}}_{p,q,n}(\ell)$, {the regularity event
defined in Lemma~\ref{lem:regularity}.} Since
\[
 |\cI_n(N)|\le\frac{x_1-x_0}{N^{2/3}}+3\le Cn,
 \qquad |\cM_n|\le bn+1,
\]
there are at most $Cn^2\log n$ boxes in $\mathfrak B_n$. All sampling regions lie
in a single rectangle with side lengths $O(n)$. Define
\[
 \cA_n=\bigcap_{(N,c,m)\in\mathfrak B_n}\cE^{\mathrm{cap}}_{N;c,m}(R),
 \qquad
 \cD_n=\bigcap_{(N,c,m)\in\mathfrak B_n}
                       \{|\cQ_{N;c,m}|\le\ell^{73}\}.
\]
{Lemma~\ref{lem:capture}, applied with $R=\ell^{12}$,
and the Poisson exponential moment give}
\begin{align}
 \PP(\cA_n^c)
 &\le Cn^{13}\log n
       e^{-c\ell^{36/11}},
       \label{eq:allboxcapture}\\
 \PP(\cD_n^c)&\le Cn^2\log n\,e^{-\ell^{73}/2}.
       \label{eq:allcloudcount}
\end{align}
{For the second bound, each cloud cardinality $|\cQ_{N;c,m}|$
is Poisson with mean at most $C\ell^{72}$,} and
$\PP(\mathrm{Poi}(\mu)>a)\le\exp\{(e-1)\mu-a\}$ with
$a=\ell^{73}$ applies uniformly as $\ell\ge\log n\to\infty$.

Let $\mathscr Q_n$ be the {$\sigma$-algebra} of all clouds, and let
$M_n$ be the number of entries in their combined endpoint list,
retaining each entry's scale and row. Then
\begin{equation}\label{eq:totalcloudmean}
 \EE M_n\le Cn^2(\log n)\ell^{72}.
\end{equation}
\begingroup
The factor $n^2\log n$ bounds the number of prepared boxes:
there are at most $Cn$ mesh intervals at each scale, $Cn$ bulk
rows, and $\log n$ scales. Each box contributes a cloud of
mean size at most $C\ell^{72}$, giving
\eqref{eq:totalcloudmean} by summing these means.
\par\endgroup

Every sampled pair at scale $N$ has row separation in
$[7N/4,9N/4]$ and its crossing row has bulk fraction at least
$1/4$. Its slope differs from $\lambda$ by at most
$CRN^{-1/3}\le C\ell^{12-Q/3}=o(1)$, uniformly in the
scale list. Their slopes therefore belong to $K'$ for all
sufficiently large $n$, so Corollary~\ref{cor:independentcovers}
applies with the fixed slope interval $K'$.
\begingroup
{Let $\cX_n$ be the event that, for every $(N,c,m)\in\mathfrak B_n$
and every sampled pair $(\tilde p,\tilde q)\in\cQ_{N;c,m}$,
the exit set $E(\tilde p,\tilde q,m;h)$ has a cover by at most
$\ell^{90}$ intervals of length at most $D_N$, as in
Corollary~\ref{cor:independentcovers}.} Conditioning on
$\mathscr Q_n$ and then averaging gives
\begin{equation}\label{eq:allcrossingcovers}
 \begin{aligned}
 \PP(\cX_n^c\mid\mathscr Q_n)&\le CM_ne^{-c\ell^2},\\
 \PP(\cX_n^c)&\le Cn^2(\log n)\ell^{72}e^{-c\ell^2}.
 \end{aligned}
\end{equation}
\par\endgroup

Set $\cE_n^{\mathrm{ref}}=\cR_n\cap\cA_n\cap\cD_n\cap\cX_n$.
Lemma~\ref{lem:regularity} gives $\PP(\cR_n^c)\le Ce^{-c\ell^3}$.
\begingroup
Combining this regularity bound with \eqref{eq:allboxcapture},
\eqref{eq:allcloudcount} and \eqref{eq:allcrossingcovers}, we obtain
\[
 \PP((\cE_n^{\mathrm{ref}})^c)
 \le\PP(\cR_n^c)+\PP(\cA_n^c)+\PP(\cD_n^c)+\PP(\cX_n^c)
 \le Ce^{-c'\ell^2}.
\]
Here $R=\ell^{12}$ makes the basin exponent
$R^{3/11}=\ell^{36/11}$. Since $36/11>2$ and
$\ell\ge\log n$, this decay absorbs the polynomial factors
in the union over boxes while preserving the displayed bound.
The other three failure bounds are also at most
$Ce^{-c'\ell^2}$, uniformly in the allowed range of $\ell$.
\par\endgroup

Fix an indexed box and an actual exit
$x=\Gamma_p^{q,t}(m)\in J_{N,k}$. On $\cR_n$, its two cut
endpoints on rows $m-N,m+N$ lie in the windows $U^-_{N;c_{N,k},m}$,
$U^+_{N;c_{N,k},m}$, because
\[
 (C_{\mathrm{reg}}\ell^2+1/2)N^{2/3}
 \le\ell^{12}N^{2/3}.
\]
{On $\cA_n$, there is a pair
$(\tilde p,\tilde q)\in\cQ_{N;c_{N,k},m}$ such that
$\Gamma_{\tilde p}^{\tilde q,t}$ agrees with this subpath between
rows $m-N/2$ and $m+N/2$, and hence also exits row $m$ at $x$.
On $\cX_n$, choose one of the supplied exit covers for each pair
in this cloud and let}
$\mathcal C(N,c_{N,k},m)$ be their union. On $\cD_n$ it contains
at most $\ell^{73}\ell^{90}=\ell^{163}$ intervals.
Every actual exit in the parent is covered, which proves
\eqref{eq:refinementcover} for all boxes on the same event.
\end{proof}
\par\endgroup

\begingroup
\section{The accelerated multiscale hitset estimate}\label{sec:multiscale}

We now iterate the common refinement event. The probabilistic work
has been isolated in Lemma~\ref{lem:refinement}; the remaining
argument is a deterministic recursion of interval covers, followed
by an expectation bound. Keeping the allowance $\ell$ adjustable
will give both regimes of Theorem~\ref{thm:regional}.
{For convenience,} write
\begin{equation}\label{eq:allowancecost}
 \Phi_n(\ell)=(\log\ell)\log\!\left(\frac{\log n}{\log\ell}\right).
\end{equation}
The second factor bounds the number of generations.
Each generation costs a fixed power of $\ell$, whose logarithm
accounts for the first factor $\log\ell$.

\begin{proposition}[A uniform estimate with an adjustable allowance]
\label{prop:generalpoint}
Fix $0<a<b<\infty$, $\beta\in(0,1/2)$ and a compact slope
interval $K\subset(0,\infty)$. Choose $Q$ as in
\eqref{eq:allowanceparameters}, using
$K'=[(\min K)/2,2\max K]$. There are $C,c>0$ and $n_0<\infty$ such that the following
holds uniformly for $n\ge n_0$, $\log n\le\ell\le n^{1/(2Q)}$
and deterministic $p=(x_0,j_0)\le q=(x_1,j_1)$ satisfying
\[
 an\le D:=j_1-j_0\le bn,\qquad
 \lambda:=\frac{x_1-x_0}{D}\in K.
\]
\begingroup
With $h=n^{-1/3}\ell^{-4}$,
\begin{equation}\label{eq:generalpointprob}
 \PP\left(\left|\hitset_p^{q,[0,h]}
       (\operatorname{Bulk}_\beta(p,q))\right|
          >n\exp\{C\Phi_n(\ell)\}\right)
 \le Ce^{-c\ell^2}.
\end{equation}
Moreover,
\begin{equation}\label{eq:trace}
 \EE\left|\hitset_p^{q,[0,h]}(\operatorname{Bulk}_\beta(p,q))\right|
 \le n\exp\{C\Phi_n(\ell)\}.
\end{equation}
\par\endgroup

The constants may depend on the fixed geometry and $Q$, but not
on $n,\ell,p,q$.
\end{proposition}

As we shall see in the proof, the scales are chosen before
the locations of any clusters are examined. At generation $j$ the parent mesh has width $N_j^{2/3}$.
The local cover has interval lengths at most
$D_{N_j}=C_{\mathrm{width}}(N_jh\ell^2+\ell^4)$; the next mesh is chosen to
accommodate these intervals. Its width becomes a successively
smaller fraction of the parent width. Figures~\ref{fig:coverhierarchy}
and~\ref{fig:scaleselection} show the interval recursion and its
geometric meaning.

\begin{figure}[tbp]
\centering
\begin{tikzpicture}[x=1cm,y=1cm]
  \node[anchor=west,text=black] at (.2,4.1) {\textbf{One row $m$, one fixed time interval $[0,h]$}};
  \draw[cluster] (.7,2.9) rectangle (11.8,3.75);
  \node[text=black] at (6.25,3.325) {parent $J$: width $N_j^{2/3}$};
  \draw[->,black!60] (6.25,2.9)--(6.25,1.65);
  \node[anchor=west,align=left,font=\footnotesize,text=black] at (6.5,2.3)
    {local \textit{Poisson capture}\\and the scale-$N_j$ exit cover};
  \foreach \a/\b in {1.1/2.0,5.1/6.0,9.8/10.7}{
    \draw[child] (\a,1.18) rectangle (\b,1.55);
  }
  \node[anchor=west,font=\footnotesize] at (.7,.75)
    {$I\in\mathcal C_j(m,J)$: at most {$\ell^{163}$} intervals, each of length at most $D_{N_j}$};
  \draw[->,black!60] (6.25,.52)--(6.25,.02);
  \node[anchor=west,font=\footnotesize] at (6.5,.27)
    {replace by next-mesh intervals};
  \foreach \a/\b in {.7/1.88,1.88/3.06,4.24/5.42,5.42/6.6,8.96/10.14,10.14/11.32}{
    \draw[cluster,fill=clusterblue!5] (\a,-.45) rectangle (\b,-.1);
  }
  \node[anchor=west,font=\footnotesize] at (.7,-.87)
    {{retain next-mesh intervals $J'$ with $J'\cap I\cap\mathcal W_n\ne\varnothing$}};
  \node[anchor=west,font=\footnotesize] at (.7,-1.3)
    {$D_{N_j}\le N_{j+1}^{2/3}$: each covering interval meets at most three next-mesh intervals};
\end{tikzpicture}
\caption{One generation of the recursion in the proof of
Proposition~\ref{prop:generalpoint}. The green intervals cover the
actual exits in $J$, and need not cover all of $J$. The second arrow
replaces this cover by intervals of the next deterministic mesh; it
is part of the same generation. The blue intervals all have width
$N_{j+1}^{2/3}\asymp D_{N_j}$, comparable to the upper bound on the
green-interval lengths. A blue interval can therefore be larger than
a particular green interval, while still being much smaller than
the original parent $J$. Each green interval meets at most three
blue intervals. The retained blue intervals become the parents for
the following generation. The numbers and placements are schematic;
no independence between generations is used.}
\label{fig:coverhierarchy}
\end{figure}
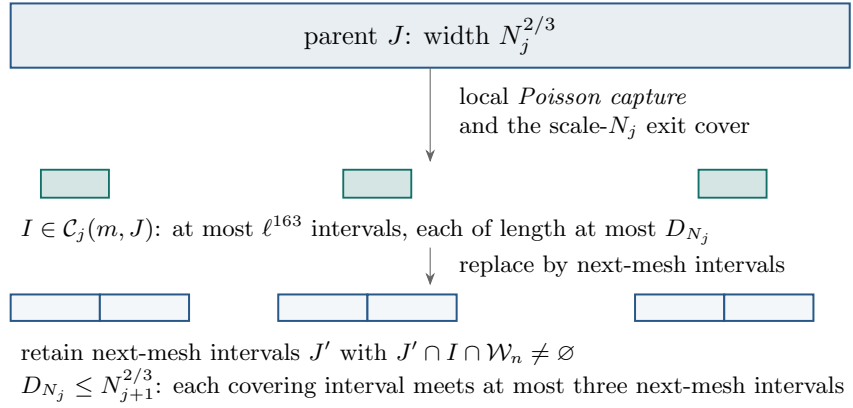

\begin{figure}[tbp]
\centering
\begin{tikzpicture}[x=.95cm,y=.85cm,font=\footnotesize]
  \node[anchor=west] at (0,7.65)
    {The same dynamical horizon $[0,h]$ at both scales; one fixed $t\in[0,h]$ is shown.};
  \node[anchor=west] at (0,7.05) {\textbf{(a) Nested cuts of the same staircase}};
  \node[anchor=west] at (8.15,7.05) {\textbf{(b) Choosing the next scale}};

  \fill[clusterblue!6] (1.6,1.2) rectangle (7.2,5.2);
  \fill[refineteal!12] (1.6,2.4) rectangle (7.2,4.0);
  \foreach \y/\lab in {1.2/{m-N_j},5.2/{m+N_j}}{
    \draw[clusterblue!65,densely dashed] (1.6,\y)--(7.2,\y);
    \node[anchor=east,text=clusterblue] at (1.45,\y) {$\lab$};
  }
  \foreach \y/\lab in {2.4/{m-N_{j+1}},4.0/{m+N_{j+1}}}{
    \draw[refineteal!75,densely dashed] (1.6,\y)--(7.2,\y);
    \node[anchor=east,text=refineteal] at (1.45,\y) {$\lab$};
  }
  \draw[black!45] (1.6,3.2)--(7.2,3.2);
  \node[anchor=east] at (1.45,3.2) {$m$};

  \draw[black!60,line width=.8pt]
    (2.0,0)--(2.4,0)--(2.4,.4)--(2.55,.4)--(2.55,.8)--(2.95,.8)
    --(2.95,1.2)--(3.15,1.2)--(3.15,1.6)--(3.55,1.6)
    --(3.55,2.0)--(3.8,2.0)--(3.8,2.4)--(4.0,2.4)
    --(4.0,2.8)--(4.25,2.8)--(4.25,3.2)--(4.75,3.2)
    --(4.75,3.6)--(5.0,3.6)--(5.0,4.0)--(5.15,4.0)
    --(5.15,4.4)--(5.6,4.4)--(5.6,4.8)--(5.7,4.8)
    --(5.7,5.2)--(6.2,5.2)--(6.2,5.6)--(6.45,5.6)
    --(6.45,6.0)--(6.7,6.0)--(6.7,6.4)--(7.0,6.4);
  \draw[clusterblue,line width=1.4pt]
    (3.15,1.2)--(3.15,1.6)--(3.55,1.6)--(3.55,2.0)--(3.8,2.0)
    --(3.8,2.4)--(4.0,2.4)--(4.0,2.8)--(4.25,2.8)
    --(4.25,3.2)--(4.75,3.2)--(4.75,3.6)--(5.0,3.6)
    --(5.0,4.0)--(5.15,4.0)--(5.15,4.4)--(5.6,4.4)
    --(5.6,4.8)--(5.7,4.8)--(5.7,5.2)--(6.2,5.2);
  \draw[refineteal,line width=1.8pt]
    (4.0,2.4)--(4.0,2.8)--(4.25,2.8)--(4.25,3.2)--(4.75,3.2)
    --(4.75,3.6)--(5.0,3.6)--(5.0,4.0)--(5.15,4.0);
  \foreach \x/\y in {3.15/1.2,6.2/5.2}
    \fill[clusterblue] (\x,\y) circle (2.2pt);
  \foreach \x/\y in {4.0/2.4,5.15/4.0}
    \fill[refineteal] (\x,\y) circle (2.2pt);
  \fill[witnessgold] (4.75,3.2) circle (2.6pt);
  \fill (2.0,0) circle (1.6pt) node[below left] {$p$};
  \fill (7.0,6.4) circle (1.6pt) node[above right] {$q$};
  \node[text=black!75] at (4.45,6.35) {$\Gamma_p^{q,t}$};
  \fill[witnessgold] (2.15,-.55) circle (2.3pt);
  \node[anchor=west] at (2.35,-.55) {the same exit on row $m$};
  \draw[black!15] (7.75,-.25)--(7.75,6.65);

  \node[text=clusterblue] at (11.15,6.55) {$J:\quad |J|=N_j^{2/3}$};
  \draw[cluster] (8.4,5.85) rectangle (13.9,6.15);
  \fill[witnessgold] (11.35,6.0) circle (2.4pt);
  \draw[->,black!60] (11.15,5.65)--(11.15,5.25);
  \node[align=center] at (11.15,4.4)
    {{exit cover} at scale $N_j$\\[2pt]
     interval lengths at most\\[2pt]
     {$D_{N_j}=C_{\mathrm{width}}(N_jh\ell^2+\ell^4)$}};
  \draw[->,black!60] (11.15,3.65)--(11.15,2.95);
  \node[anchor=west] at (11.4,3.3) {choose $N_{j+1}$};
  \draw[child] (10.05,2.35) rectangle (12.45,2.65);
  \fill[witnessgold] (11.35,2.5) circle (2.4pt);
  \node[text=refineteal] at (11.15,1.95) {$J':\quad |J'|=N_{j+1}^{2/3}$};
  \node at (11.15,1.3) {$D_{N_j}\le N_{j+1}^{2/3}\asymp D_{N_j}$};
  \node[align=center,text=refineteal] at (11.15,.35)
    {Cut at $m\pm N_{j+1}$\\and repeat around the same row $m$.};
\end{tikzpicture}
\caption{Two successive nonterminal scales in the
cut-and-refine construction.
\textbf{(a)} The blue and green bands mark the cuts at
$m\pm N_j$ and $m\pm N_{j+1}$ along the same geodesic
$\Gamma_p^{q,t}$. The colored endpoints are its exits from the
respective cut rows. Both portions retain the gold exit on row $m$.
\textbf{(b)} For the parent $J$, the local estimate gives covering
intervals of length at most $D_{N_j}$. We choose the next scale so
that $N_{j+1}^{2/3}\ge D_{N_j}$, and retain next-mesh intervals as in
Figure~\ref{fig:coverhierarchy}. One such interval $J'$ containing the
displayed exit is shown; the next step uses the shorter green cut
in (a). The scales are deterministic, while the cut endpoints and
retained intervals depend on the environment. The estimates use
the Poisson representatives from Figure~\ref{fig:cutting} and hold
simultaneously for all $t\in[0,h]$; only one time is drawn here.
The vertical coordinate denotes the BLPP row, not dynamical time.
Lengths are schematic. The recursion stops when {$N_{g_n}\le\ell^Q$}.}
\label{fig:scaleselection}
\end{figure}
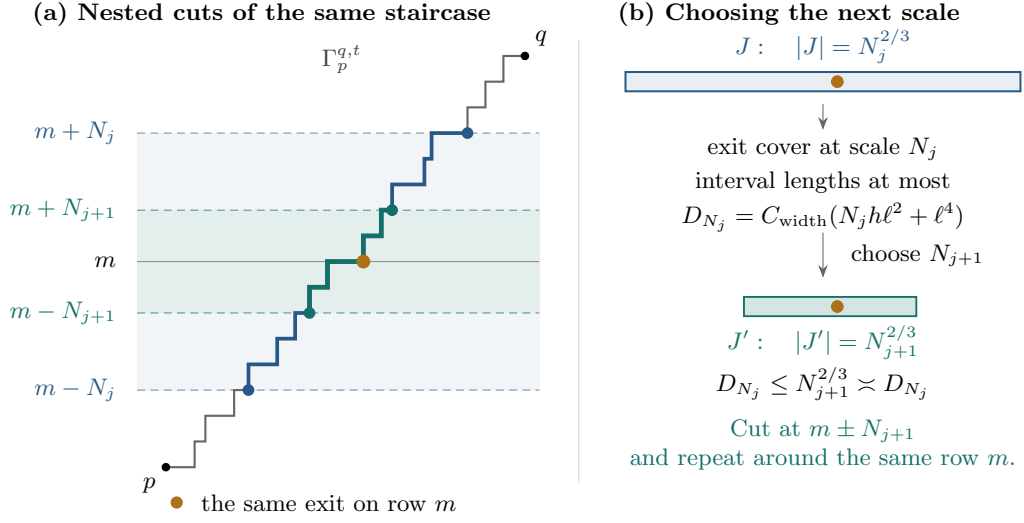

\begin{proof}
We divide the argument into four steps. Step 1 chooses the scales
and bounds the number of generations. Step 2 obtains one event on
which all the local covers hold. On this event, Step 3 constructs
the retained interval families in a fixed row and proves that they
continue to cover every actual exit. Step 4 counts these intervals,
converts them into a cover of the horizontal path portions, and
sums over rows. Figures~\ref{fig:coverhierarchy}
and~\ref{fig:scaleselection} illustrate the two constructions.

\begingroup
Fix $n$, an allowed value of $\ell$, and the endpoints $p,q$.
The quantities $h=n^{-1/3}\ell^{-4}$, $R=\ell^{12}$ and
$P=C_{\mathrm{reg}}\ell^2$ are then fixed for this proof.
Only the longitudinal scale and the retained interval families
change from one generation to the next. In particular, every
generation examines exits during the same time interval $[0,h]$.
\par\endgroup

\par\smallskip\noindent\textbf{Step 1. The scales and their acceleration.}\quad
The initial scale is $N_0$ from \eqref{eq:initialscale}.
At scale $N_j$, the local cover has interval lengths at most
$D_{N_j}=C_{\mathrm{width}}(N_jh\ell^2+\ell^4)$. The next mesh must have width
at least this bound. Since mesh width is the $2/3$ power of the
longitudinal scale, we raise the length bound to the power $3/2$
and define
\begin{equation}\label{eq:recursion}
 N_{j+1}=\left\lceil
       \bigl(C_*(N_jh\ell^2+\ell^4)\bigr)^{3/2}\right\rceil,
 \qquad g_n=\min\{j:N_j\le\ell^Q\},
\end{equation}
where $C_*\ge C_{\mathrm{width}}$ is fixed. Then
$N_{j+1}^{2/3}\ge D_{N_j}$. Also $\ell^Q\le n^{1/2}<N_0$
for all sufficiently large $n$, uniformly over the endpoint ranges.
For a fixed $C_1$, the sum and ceiling in \eqref{eq:recursion} give
\begin{equation}\label{eq:recursionbound}
 N_{j+1}\le C_1\max\{N_j^{3/2}n^{-1/2}\ell^{-3},\ell^6\}.
\end{equation}
For $\ell^Q<N_j\le n$, the right-hand side divided by $N_j$
is at most $C_1\max\{\ell^{-3},\ell^{6-Q}\}<1$.
Thus the scales decrease until the cutoff is reached.

\begingroup
There are two ways in which \eqref{eq:recursionbound} can control
the next scale. If its maximum is attained by $\ell^6$, then
$N_{j+1}\le C_1\ell^6<\ell^Q$, so this step reaches the cutoff.
Otherwise the first term bounds $N_{j+1}$ and gives the further
reduction that we now quantify. In particular, whenever
$j+1<g_n$, the maximum must be attained by that first term.

Write $b_j=\log(n/N_j)$. This measures how much the scale has
already decreased relative to $n$: smaller $N_j$ means larger
$b_j$, and the cutoff corresponds to
$b_j\ge\log(n/\ell^Q)$. For a step that remains above the cutoff,
taking logarithms of the first term in \eqref{eq:recursionbound}
gives
\par\endgroup
\begin{equation}\label{eq:acceleration}
 b_{j+1}\ge\tfrac32b_j+3\log\ell-\log C_1
           \ge\tfrac32b_j+2\log\ell.
\end{equation}
\begingroup
The factor $3/2$ amplifies the decrease already accumulated;
the gain is therefore faster than adding a fixed multiple of
$\log\ell$ at every generation. Iterating
\eqref{eq:acceleration} from $b_0\ge0$ gives, for $j<g_n$,
\[
 b_j\ge 2\log\ell\sum_{i=0}^{j-1}(3/2)^i
      =4\log\ell\bigl((3/2)^j-1\bigr).
\]
This lower bound grows geometrically in $j$, so only logarithmically
many generations are needed for it to exceed the stopping threshold.
More precisely,
\par\endgroup

\begin{equation}\label{eq:generations}
 g_n\le J_n(\ell):=
 \left\lceil\log_{3/2}\left(1+\frac{\log n}{4\log\ell}\right)\right\rceil
 \le C\log\!\left(\frac{\log n}{\log\ell}\right).
\end{equation}
\begingroup
To verify the first inequality, write $J=J_n(\ell)$ and suppose
that $g_n>J$. By the definition of the stopping generation,
$N_J>\ell^Q$, so the iteration of \eqref{eq:acceleration}
above is valid up to $j=J$. The definition of $J$ gives
$(3/2)^J\ge1+\log n/(4\log\ell)$, and therefore
\[
 b_J\ge4\log\ell\bigl((3/2)^J-1\bigr)\ge\log n.
\]
Since $b_J=\log(n/N_J)$, this implies $N_J\le1$,
contradicting $N_J>\ell^Q>1$. Thus the cutoff must have been
reached by generation $J$.
\par\endgroup
\begingroup
 The second inequality in \eqref{eq:generations}, which bounds
$J_n(\ell)$ by $C\log(\log n/\log\ell)$, uses
$\log n/\log\ell\ge2Q>2$ to absorb the additive constant
from the ceiling.
\par\endgroup
\begingroup
Uniformly over the allowed $\ell\ge\log n$,
\eqref{eq:generations} gives $g_n\le J_n(\ell)=O(\log\log n)$.
In particular, the scale list has at most $\log n$ elements for
all sufficiently large $n$, as required by
Lemma~\ref{lem:refinement} in the next step.
\par\endgroup

\par\smallskip\noindent\textbf{Step 2. One event for all possible parents.}\quad
The scale list $\mathcal S_n=\{N_j:0\le j<g_n\}$ is deterministic:
the recursion uses only $n,\ell$ and the original endpoint geometry.
Every scale in this list exceeds $\ell^Q$, and
\eqref{eq:generations} bounds its cardinality by $\log n$.
Lemma~\ref{lem:refinement} therefore supplies its common event
$\cE_n^{\mathrm{ref}}$, with failure at most $Ce^{-c\ell^2}$.
The terminal scale $N_{g_n}$ is excluded from $\mathcal S_n$
because we stop refining there; Step 4 will count cells in its
intervals directly.

For the remainder of the counting argument fix an outcome in
$\cE_n^{\mathrm{ref}}$. The local covers are available for every
box in $\mathfrak B_n(\mathcal S_n)$ before the retained families
are selected. We will check that each retained parent indexes one
of these boxes. No additional probability estimate is needed when
the random family of parents becomes known.

\par\smallskip\noindent\textbf{Step 3. The retained families and their cover invariant.}\quad
Fix $m\in\cM_n$. For $0\le j\le g_n$ set
$w_j=N_j^{2/3}$ and
$\mathcal G_j=\{[kw_j,(k+1)w_j]:k\in\ZZ\}$.
With $P=C_{\mathrm{reg}}\ell^2$, define
\begin{equation}\label{eq:initialfamily}
 \begin{gathered}
 R_m=[x_0+\lambda(m-j_0)-Pn^{2/3},
             x_0+\lambda(m-j_0)+Pn^{2/3}],\\
 \mathcal J_0(m)=\{J\in\mathcal G_0:
                          J\cap R_m\cap\mathcal W_n\ne\emptyset\}.
 \end{gathered}
\end{equation}
Every staircase from $p$ to $q$ stays in the horizontal range
$\mathcal W_n=[x_0,x_1]$. {Since we have fixed an outcome in
$\cE_n^{\mathrm{ref}}$, the regularity bound \eqref{eq:rootTF}
places the exits of $\Gamma_p^{q,t}$ in $R_m$ for every $t\in[0,h]$.}
Hence $E_m\subseteq R_m\cap\mathcal W_n$, which is
covered by $\mathcal J_0(m)$. Moreover,
\begin{equation}\label{eq:initialcount}
 |\mathcal J_0(m)|\le\frac{2Pn^{2/3}}{w_0}+3\le C\ell^2,
\end{equation}
since $N_0\asymp n$.

\begingroup
The families $\mathcal G_j$ contain all intervals of the mesh;
$\mathcal J_j(m)$ contains only those retained in the cover.
We construct the latter by induction, maintaining two properties:
every retained interval meets $\mathcal W_n$, and the family
covers $E_m$. Both properties hold for $\mathcal J_0(m)$ by
\eqref{eq:initialfamily} and the preceding localization argument.

Suppose that $j<g_n$ and $\mathcal J_j(m)$ has these properties.
For $J\in\mathcal J_j(m)$, let $c_J$ be its midpoint. Since
$J\in\mathcal G_j$ meets $\mathcal W_n$, its horizontal mesh
index belongs to $\cI_n(N_j)$. Together with $m\in\cM_n$ and
$N_j\in\mathcal S_n$, this places $(N_j,c_J,m)$ in
$\mathfrak B_n(\mathcal S_n)$. The common refinement event
therefore supplies the family
$\mathcal C_j(m,J)=\mathcal C(N_j,c_J,m)$, with
\[
 E_m\cap J\subseteq\bigcup_{I\in\mathcal C_j(m,J)}I,
 \qquad |\mathcal C_j(m,J)|\le\ell^{163},
 \qquad |I|\le D_{N_j}\le w_{j+1}
       \quad(I\in\mathcal C_j(m,J)).
\]
This family comes from the representative geodesics for the parent
$J$. To use the prepared estimates at the next scale, we replace
its covering intervals by intervals of the next deterministic mesh.
{We keep those that meet the cover within $\mathcal W_n$:}
\begin{equation}\label{eq:childfamily}
 \mathcal J_{j+1}(m)=
 \left\{J'\in\mathcal G_{j+1}:\begin{array}{l}
 {J'\cap I\cap\mathcal W_n\ne\emptyset}\\
 \text{for some }J\in\mathcal J_j(m)
                 \text{ and }I\in\mathcal C_j(m,J)
 \end{array}\right\}.
\end{equation}

First, every interval in $\mathcal J_{j+1}(m)$ meets
$\mathcal W_n$ by this definition. If the next generation is
nonterminal, the same index check therefore permits another
application of the refinement lemma to each of its parents.
Second, let $x\in E_m$. The induction hypothesis gives a
$J\in\mathcal J_j(m)$ containing $x$, and the local cover
gives an $I\in\mathcal C_j(m,J)$ containing $x$. Since
$x\in\mathcal W_n$, any interval $J'\in\mathcal G_{j+1}$
containing $x$ satisfies the intersection condition in
\eqref{eq:childfamily} and is retained. Thus every actual exit
is still covered in generation $j+1$. Therefore, by induction, we have
\begin{equation}\label{eq:coverinvariant}
 E_m\subseteq\bigcup_{J\in\mathcal J_j(m)}J,
       \qquad 0\le j\le g_n.
\end{equation}
The invariant guarantees coverage of the points in $E_m$,
namely the actual exits during $[0,h]$; it does not require
covering the other points of each parent interval. These two
inductive checks ensure both that each refinement is available
and that no exit is lost as the cover becomes finer.
\par\endgroup

\par\smallskip\noindent\textbf{Step 4. Counting intervals and cells.}\quad
Each covering interval has length at most
$D_{N_j}\le w_{j+1}$ and therefore meets at most three closed
intervals of $\mathcal G_{j+1}$, including endpoint contacts.
Summing over the retained parents and using
\eqref{eq:refinementcover} and \eqref{eq:initialcount}, we obtain
\begin{equation}\label{eq:branchcount}
 \begin{aligned}
 |\mathcal J_{j+1}(m)|
 &\le3\sum_{J\in\mathcal J_j(m)}|\mathcal C_j(m,J)|
 \le3\ell^{163}|\mathcal J_j(m)|,\\
 |\mathcal J_{g_n}(m)|&\le C\ell^2(3\ell^{163})^{g_n}.
 \end{aligned}
\end{equation}
These are cardinality inequalities for the fixed outcome in
$\cE_n^{\mathrm{ref}}$. Dependence between generations has no
effect on them: we add the number of children over the parents
actually retained. The polynomial number of prepared boxes entered
the probability estimate for the common event, not this count.

At the terminal generation we impose no further near-maximum test.
For each $J\in\mathcal J_{g_n}(m)$, put
\[
 \widehat J=[\inf J-P,\sup J],\qquad
 H_m=\hitset_p^{q,[0,h]}(\RR\times\{m\}).
\]
\begingroup
Suppose that at time $t$ the exit $x=\Gamma_p^{q,t}(m)$ lies
in $J$. The geodesic enters row $m$ at
$y=\Gamma_p^{q,t}(m-1)$ and traverses the horizontal interval
$[y,x]$ before leaving that row. By \eqref{eq:rowlength},
$0\le x-y\le P$, and therefore
\[
 [y,x]\subseteq[x-P,x]\subseteq[\inf J-P,\sup J]=\widehat J.
\]
The exit-cover invariant ensures that every exit belongs to a
retained terminal interval. Applying this inclusion at every
$t\in[0,h]$ gives
\par\endgroup

\[
 H_m\subseteq\bigcup_{J\in\mathcal J_{g_n}(m)}
                      \coarse(\widehat J\times\{m\}).
\]
\begingroup
Thus $\widehat J$ extends $J$ to the left to include every
possible entrance coordinate of a row portion whose exit lies
in $J$. Its length is $w_{g_n}+P$. An interval of length $L$
meets at most $L+2$ closed unit cells, including contacts at
its endpoints, so each $\widehat J$ meets at most
\par\endgroup

\[
 w_{g_n}+P+2\le\ell^{2Q/3}+C_{\mathrm{reg}}\ell^2+2
                  \le C\ell^{2Q/3}
\]
unit cells, using $N_{g_n}\le\ell^Q$ and $Q>3$. Combining this
with \eqref{eq:branchcount} proves
\begin{equation}\label{eq:rowcount}
 \left|\hitset_p^{q,[0,h]}(\RR\times\{m\})\right|
 \le C\ell^{2+2Q/3}(3\ell^{163})^{g_n}.
\end{equation}
There are at most $bn+1$ bulk rows. Summing and using
\eqref{eq:generations} yields
\begin{equation}\label{eq:cellcount}
 \left|\hitset_p^{q,[0,h]}(\operatorname{Bulk}_\beta(p,q))\right|
 \le Cn\ell^{2+2Q/3}(3\ell^{163})^{g_n}
 \le n\exp\{C'\Phi_n(\ell)\}.
\end{equation}
The factors in the first bound have distinct origins:
$\ell^2$ counts the initial intervals, $(3\ell^{163})^{g_n}$
counts their descendants, and $\ell^{2Q/3}$ counts cells in one
terminal interval. Their product has logarithm at most
$C_Qg_n\log\ell$.
Equation~\eqref{eq:generations} bounds
this by $C'\Phi_n(\ell)$, which gives the second inequality
in \eqref{eq:cellcount}.

This bound holds on $\cE_n^{\mathrm{ref}}$, proving the tail
estimate after adjusting $C$. Off this event the containing
rectangle gives the deterministic bound $Cn^2$. Hence the
expectation is at most
$n\exp\{C'\Phi_n(\ell)\}+Cn^2e^{-c\ell^2}$, which is bounded
by $n\exp\{C\Phi_n(\ell)\}$ uniformly for $\ell\ge\log n$.
\end{proof}

\begingroup
The role of acceleration is clearest for the logarithmic choice
$\ell=\log n$. Using only the additive consequence
$b_{j+1}\ge b_j+2\log\ell$ of \eqref{eq:acceleration}
would give a bound of
$O(\log n/\log\log n)$ generations. Multiplying a fixed power
of $\log n$ at each generation would then yield only a counting
bound of $n^{O(1)}$. By contrast, \eqref{eq:generations} gives
$O(\log\log n)$ generations, and hence the subpolynomial bound
$\exp\{O((\log\log n)^2)\}$. This is the improvement needed
for the finer hitset threshold and the Hausdorff-gauge conclusion.
For $\ell=n^a$ with fixed $a>0$, both arguments already give a
number of generations bounded independently of $n$: the additive
argument gives $O(1/a)$ many generations, while acceleration improves this to
$O(\log(1/a))$ many. Thus in the power-law case the gain from the accelaration is in the
dependence on $a$, whereas in the logarithmic case it makes the
accumulated counting bound subpolynomial in $n$.
\par\endgroup

\begingroup
In Proposition~\ref{prop:generalpoint}, the refinement is carried
out over the shorter horizon $h=n^{-1/3}\ell^{-4}$. This choice
keeps the passage-time tolerance small compared with the natural
weight fluctuation scale $N^{1/3}$ at every scale where {the exit-cover
estimate of Proposition~\ref{prop:local}} is applied; see \eqref{eq:alphacheck}. To
recover the estimate over the critical time interval
$[0,n^{-1/3}]$, we cover it by $\lceil\ell^4\rceil$ intervals
of length at most $h$ and add their hitsets. This costs another
fixed power of $\ell$, which is absorbed in
$\exp\{C\Phi_n(\ell)\}$. The following elementary lemma records
this transfer; it uses stationarity and requires no independence
between the time intervals.
\par\endgroup

\begin{lemma}[Time subdivision]\label{lem:timesubdivision}
Fix deterministic endpoint sets $A,B$ and a deterministic spatial
region $R$, and write $H(I)=\hitset_A^{B,I}(R)$ for compact
dynamical time intervals $I\subseteq[0,\infty)$. For $T,h>0$, put
$K=\lceil T/h\rceil$. Then, for every $u>0$,
\[
 \begin{aligned}
 \PP\bigl(|H([0,T])|>Ku\bigr)
   &\le K\,\PP\bigl(|H([0,h])|>u\bigr),\\
 \EE|H([0,T])|&\le K\,\EE|H([0,h])|.
 \end{aligned}
\]
\end{lemma}

\begin{proof}
Let $I_j=[jh,(j+1)h]$ for $0\le j<K$. Since $Kh\ge T$,
\[
 H([0,T])\subseteq\bigcup_{j=0}^{K-1}H(I_j),\qquad
 |H([0,T])|\le\sum_{j=0}^{K-1}|H(I_j)|.
\]
Stationarity gives $H(I_j)\stackrel{d}=H([0,h])$ for every $j$,
because $A,B,R$ are deterministic and unchanged by a time shift.
If the sum exceeds $Ku$, at least one term exceeds $u$; the union
bound proves the first inequality. Taking expectations proves the
second. No independence between the time intervals is used.
\end{proof}
\par\endgroup

\begingroup
\section{Regional hitsets and exceptional times}\label{sec:dimension}

{The preceding section controls one deterministic endpoint pair.
We now use Poisson capture once more to allow both endpoints to
vary. This was the original role of Poisson capture in
\cite[Section 5]{B25}. Its earlier use here, to represent shorter
geodesic segments with random cut endpoints, is new to this paper.
The resulting regional estimate gives Theorem~\ref{thm:regional}.
Spatial averaging then yields the bounds on the probability of
visiting a marked cell in Lemmas~\ref{lem:averaging}
and~\ref{lem:directionwindows}, which we use to study exceptional
times at which bigeodesics exist.}

\subsection{Endpoint regions and the critical time interval}

For $v=(\lambda n,n)$ and allowance $\ell$, define
\begin{equation}\label{eq:Kregion}
 \mathcal K_{n,\lambda}(\ell)
 =\{(x,j)\in\ZZ_{\RR}: |j|\le n/32,
                  \ |x-\lambda j|\le n^{2/3}\ell^{12}\}.
\end{equation}
Then $-v+\mathcal K_{n,\lambda}(\ell)$ and
$v+\mathcal K_{n,\lambda}(\ell)$ are exactly the windows
$U^-_{n;0,0},U^+_{n;0,0}$ with $R=\ell^{12}$ and slope $\lambda$ in
\eqref{eq:localwindows}; recall that this dependence on $R$ and $\lambda$ is suppressed in the
notation $U^\pm_{n;0,0}$.

\begin{proposition}[Regional hitsets with an adjustable allowance]
\label{prop:generalregional}
Fix a compact interval $K\subset(0,\infty)$. There is $Q_0<\infty$
such that, for each fixed $Q>Q_0$, there are $C,c>0$ and
$n_0<\infty$ with the following property. For $n\ge n_0$,
$\log n\le\ell\le n^{1/(2Q)}$ and $\lambda\in K$, let
\[
 X_{n,\lambda}(\ell)=
 \left|\hitset_{-v+\mathcal K_{n,\lambda}(\ell)}
                  ^{v+\mathcal K_{n,\lambda}(\ell),[0,n^{-1/3}]}
                           (\slab{-n/2}{n/2})\right|.
\]
\begingroup
Uniformly in these parameters,
\begin{equation}\label{eq:generalregionalprob}
 \PP\bigl(X_{n,\lambda}(\ell)>n\exp\{C\Phi_n(\ell)\}\bigr)
 \le Ce^{-c\ell^2}.
\end{equation}
Moreover,
\begin{equation}\label{eq:regionalmain}
 \EE X_{n,\lambda}(\ell)\le n\exp\{C\Phi_n(\ell)\}.
\end{equation}
\par\endgroup
\end{proposition}

\begin{proof}
Set $K_1=[(\min K)/2,2\max K]$ and
$K_2=[(\min K_1)/2,2\max K_1]$, so $K_2$ is the prescribed
enlargement of $K_1$ in Proposition~\ref{prop:generalpoint}.
Take $Q_0$ large enough
for \eqref{eq:allowanceparameters} using the capture exponent
for $K_2$. This permits Proposition~\ref{prop:generalpoint}
with input slopes in $K_1$, and all its local representatives
in $K_2$.

First work during $[0,h]$, where $h=n^{-1/3}\ell^{-4}$.
Apply Lemma~\ref{lem:capture} with $N=n$, $c=m=0$,
$R=\ell^{12}$ and slope $\lambda$. Write $\cQ_{n;0,0}$ for
its cloud, $M_n$ for its cardinality, and
$\cA_n^{\mathrm{out}}$ for its capture event. The mean satisfies
$\EE M_n\le C\ell^{72}$. Define
\[
 \cD_n^{\mathrm{out}}=\{M_n\le\ell^{73}\},\qquad
 H_n(a',b')=\hitset_{a'}^{b',[0,h]}(\slab{-n/2}{n/2}).
\]
The capture and Poisson bounds give
\[
 \PP((\cA_n^{\mathrm{out}})^c)
 \le Cn^{11}e^{-c\ell^{36/11}},\qquad
 \PP((\cD_n^{\mathrm{out}})^c)\le e^{-\ell^{73}/2}.
\]

 Each sampled pair has row separation in $[7n/4,9n/4]$, and its
slope differs from $\lambda$ by at most
$C\ell^{12}n^{-1/3}=o(1)$ uniformly in the allowed $\ell$.
For large $n$ its slope is in $K_1$. The slab
$\slab{-n/2}{n/2}$ is contained in
$\operatorname{Bulk}_{1/8}(a',b')$.
{Denote the threshold constant $C$ in \eqref{eq:generalpointprob}
by $C_1$, applying Proposition~\ref{prop:generalpoint} with
$a=7/4$, $b=9/4$, $\beta=1/8$ and slope interval $K_1$.}
Let $\cX_n^{\mathrm{out}}$ be the event that every sampled pair
satisfies $|H_n(a',b')|\le n\exp\{C_1\Phi_n(\ell)\}$.
Conditional only on $\mathscr Q=\sigma(\cQ_{n;0,0})$,
Proposition~\ref{prop:generalpoint} applies uniformly to this
deterministic list. It gives
\begingroup
\[
 \begin{aligned}
 \PP((\cX_n^{\mathrm{out}})^c\mid\mathscr Q)
 &\le CM_ne^{-c\ell^2},\\
 \PP((\cX_n^{\mathrm{out}})^c)
 &=\EE\!\left[\PP((\cX_n^{\mathrm{out}})^c\mid\mathscr Q)\right]
 \le C(\EE M_n)e^{-c\ell^2}
 \le C\ell^{72}e^{-c\ell^2}.
 \end{aligned}
\]
\par\endgroup
Let $H_n^{\mathrm{reg}}$ be the regional hitset in the proposition
with $[0,h]$ in place of $[0,n^{-1/3}]$. On
$\cA_n^{\mathrm{out}}$, capture gives the deterministic inclusion
\begin{equation}\label{eq:outerinclusion}
 H_n^{\mathrm{reg}}
 \subseteq\bigcup_{(a',b')\in\cQ_{n;0,0}}H_n(a',b').
\end{equation}
\begingroup
Set
$\cE_n^{\mathrm{out}}=\cA_n^{\mathrm{out}}\cap
\cD_n^{\mathrm{out}}\cap\cX_n^{\mathrm{out}}$.
On this event,
\[
 |H_n^{\mathrm{reg}}|
 \le\ell^{73}n\exp\{C_1\Phi_n(\ell)\}
 \le n\exp\{C_2\Phi_n(\ell)\},
\]
because $\Phi_n(\ell)\ge\log\ell$. Its complement is
\[
 (\cE_n^{\mathrm{out}})^c
 = (\cA_n^{\mathrm{out}})^c\cup(\cD_n^{\mathrm{out}})^c
      \cup(\cX_n^{\mathrm{out}})^c,
 \qquad
 \PP((\cE_n^{\mathrm{out}})^c)\le Ce^{-c\ell^2}.
\]
Every regional hitset has at most $Cn^2$ cells, since all its
paths lie in a rectangle of side lengths $O_K(n)$. Using this
deterministic bound on $(\cE_n^{\mathrm{out}})^c$ gives
\[
 \begin{aligned}
 \EE|H_n^{\mathrm{reg}}|
 &=\EE\!\left[|H_n^{\mathrm{reg}}|\ind_{\cE_n^{\mathrm{out}}}\right]
   +\EE\!\left[|H_n^{\mathrm{reg}}|\ind_{(\cE_n^{\mathrm{out}})^c}\right]\\
 &\le n\exp\{C_2\Phi_n(\ell)\}+Cn^2e^{-c\ell^2}
 \le n\exp\{C_3\Phi_n(\ell)\},
 \end{aligned}
\]
after increasing the constant, uniformly for $\ell\ge\log n$.
\par\endgroup

{Finally,} partition $[0,n^{-1/3}]$ into $\lceil\ell^4\rceil$
intervals of length at most $h$. Lemma~\ref{lem:timesubdivision}
multiplies the threshold, expectation and failure probability by
at most $\lceil\ell^4\rceil$. This factor is absorbed by
increasing $C_3$ in the threshold and decreasing the positive
constant in the failure exponent. These adjustments are uniform
over $\log n\le\ell\le n^{1/(2Q)}$.
\end{proof}

\subsection{The two regimes of the hitset theorem}

{We now deduce Theorem~\ref{thm:regional} by choosing the
allowance $\ell$ in terms of the desired size exponent $\varepsilon$.}

\begingroup
\begin{proof}[Proof of Theorem~\ref{thm:regional}]
Fix a compact slope interval containing $1$, and choose $Q$ as in
Proposition~\ref{prop:generalregional}. By \eqref{eq:introregions}
and \eqref{eq:Kregion},
$\mathscr R_n^\pm\subseteq\pm\bn+\mathcal K_{n,1}(\ell)$
whenever $\ell\ge\log n\ge1$. Thus the proposition gives fixed
constants $C_0,C,c>0$ such that, for all sufficiently large $n$
and every $\log n\le\ell\le n^{1/(2Q)}$,
\[
 \PP\bigl(X_n>n\exp\{C_0\Phi_n(\ell)\}\bigr)
       \le Ce^{-c\ell^2},
 \qquad
 \EE X_n\le n\exp\{C_0\Phi_n(\ell)\}.
\]
To obtain the threshold $n^{1+\varepsilon}$ in
\eqref{eq:segmenthitsetprob}, our goal is to choose $\ell$ so that
\[
 e^{C_0\Phi_n(\ell)}\le n^\varepsilon,
 \qquad\text{equivalently}\qquad
 C_0\Phi_n(\ell)\le\varepsilon\log n.
\]
Larger values of $\ell$ improve the failure probability. Writing
$\ell=n^a$, the bound on the hitset threshold becomes
$C_0a\log(1/a)\le\varepsilon$. Set
\begin{equation}\label{eq:allowancechoice}
 a=\frac{\varepsilon}{2C_0\log(1/\varepsilon)},\qquad \ell=n^a.
\end{equation}
If $\varepsilon_0$ is sufficiently small, then, uniformly for
$0<\varepsilon\le\varepsilon_0$,
\[
 \log(1/a)\le2\log(1/\varepsilon),
 \qquad C_0a\log(1/a)\le\varepsilon.
\]
Taking $\varepsilon_0$ small also ensures $a\le1/(2Q)$, while
taking $A$ sufficiently large ensures $\ell\ge\log n$
throughout the stated range of $\varepsilon$ for all sufficiently
large $n$. Thus the allowance is admissible, and the probability
bound above gives
\[
 \PP(X_n>n^{1+\varepsilon})
 \le C\exp\!\left\{-c n^{\varepsilon/(C_0\log(1/\varepsilon))}\right\}.
\]
Reducing $c$ if necessary gives \eqref{eq:segmenthitsetprob}.
The expectation estimate above also gives
$\EE X_n\le n^{1+\varepsilon}$ for the same choice of $\ell$.
All constants and the lower bound on $n$ are independent of
$\varepsilon$ in the stated range.

Finally, take $\varepsilon=D(\log\log n)^2/\log n$ with a
sufficiently large fixed $D\ge A$. This gives the smaller-threshold
tail \eqref{eq:quantitativehitsettail} from
\eqref{eq:segmenthitsetprob}, and the mean bound
\eqref{eq:segmenthitset} from the preceding expectation estimate.
\end{proof}
\par\endgroup

We will also need the logarithmically enlarged regions themselves.
Define
\begin{equation}\label{eq:logendpointregions}
 \mathcal K_{n,\lambda}^{\log}
 :=\mathcal K_{n,\lambda}(\log n)
 =\{(x,j)\in\ZZ_{\RR}:|j|\le n/32,
              \ |x-\lambda j|\le n^{2/3}(\log n)^{12}\}.
\end{equation}
The same choice in Proposition~\ref{prop:generalregional} gives,
uniformly for $\lambda$ in every fixed compact interval $K$,
\begin{equation}\label{eq:logregionalmean}
 \EE\left|\hitset_{-v+\mathcal K_{n,\lambda}^{\log}}
                  ^{v+\mathcal K_{n,\lambda}^{\log},[0,n^{-1/3}]}
                          (\slab{-n/2}{n/2})\right|
 \le n\exp\{C_K(\log\log n)^2\}.
\end{equation}

\subsection{A marked cell and one direction window}

For a small fixed $\delta>0$, let
\begin{equation}\label{eq:smallI}
 \mathscr I_n^\delta=[-n^{2/3+\delta}/2,n^{2/3+\delta}/2]\times\{0\}.
\end{equation}
For $v=(\lambda n,n)$ the two windows are
$-v+\mathscr I_n^\delta,v+\mathscr I_n^\delta$.
Integer translation invariance lets us average the marked-cell
probability over a central tube, while enlarging the endpoint sets.

\begin{lemma}[Spatial averaging]\label{lem:averaging}
Fix $e>0$ and a compact interval $K\subset(0,\infty)$.
For every sufficiently small fixed $\delta>0$, uniformly over
$\lambda\in K$ and all sufficiently large $n$,
\begin{equation}\label{eq:onewindow}
 \PP\left(\0\in
 \hitset_{-v+\mathscr I_n^\delta}^{v+\mathscr I_n^\delta,[0,n^{-1/3}]}
          \right)\le Cn^{-2/3+e}.
\end{equation}
\end{lemma}

\begin{proof}
{Choose $Q$ as in Proposition~\ref{prop:generalregional} on $K$,
and denote the constant $C$ in its expectation bound
\eqref{eq:regionalmain} by $C_0$.}
Choose a fixed $a\in(0,1/(2Q))$ so small that
$C_0a\log(1/a)<e/2$, and take $0<\delta<6a$.
Use the allowance $\ell=n^a$. Define the deterministic set of cells
\[
 S_{n,v}=\{(i,j)\in\ZZ^2: |j|\le n/128,
                              \ |i-\lambda j|\le n^{2/3}\}.
\]
It contains at least
$(2\lfloor n/128\rfloor+1)(2n^{2/3}-1)\ge cn^{5/3}$ cells.
For every $z\in S_{n,v}$,
$\mathscr I_n^\delta+z\subseteq\mathcal K_{n,\lambda}(\ell)$
for large $n$: the row is in $[-n/32,n/32]$, and its transverse
half-width is at most
$n^{2/3+\delta}/2+n^{2/3}\le n^{2/3+12a}$.
\begingroup
Let $p_n$ be the probability in \eqref{eq:onewindow}, and put
\[
 \mathcal H_{n,v}
 =\hitset_{-v+\mathcal K_{n,\lambda}(\ell)}
          ^{v+\mathcal K_{n,\lambda}(\ell),[0,n^{-1/3}]}
          (\slab{-n/2}{n/2}),
\]
so that $|\mathcal H_{n,v}|=X_{n,\lambda}(\ell)$.
Integer translation invariance gives, for each $z\in S_{n,v}$,
\[
 p_n=\PP\!\left(z\in
 \hitset_{-v+\mathscr I_n^\delta+z}
          ^{v+\mathscr I_n^\delta+z,[0,n^{-1/3}]}\right)
 \le\PP(z\in\mathcal H_{n,v}).
\]
The inequality uses the endpoint-set inclusion above and the fact
that $z$ lies in the central slab. Summing these inequalities
and then applying Proposition~\ref{prop:generalregional} yields
\[
 \begin{aligned}
 |S_{n,v}|p_n
 &\le\EE\sum_{z\in S_{n,v}}\ind\{z\in\mathcal H_{n,v}\}
 \le\EE X_{n,\lambda}(\ell)\\
 &\le n^{1+C_0a\log(1/a)}\le n^{1+e/2}.
 \end{aligned}
\]
\par\endgroup
Division by $|S_{n,v}|$ proves the stated bound, with constants
uniform in $\lambda\in K$.
\end{proof}
\par\endgroup

\subsection{Infinite geodesic localization and a fixed direction}
\label{sec:fixeddirection}

We first record the infinite-geodesic result needed to pass from
finite crossing events to bigeodesics. For $p,q\in\RR^2$, let
$\mathbb L_p^q$ denote the straight line segment joining them.
For $A\subseteq\RR^2$ and $R\ge0$, write
\[
 B_R(A)=\{(x+u,j):(x,j)\in A,\ |u|\le R\}
\]
for its horizontal $R$-neighbourhood. For the directedness and
localization assertions below, \cite[Appendix 1]{B25} adapts
the classical argument of Newman and Howard--Newman from static
first-passage percolation to dynamical BLPP, by making the
transversal fluctuation estimates uniform in dynamical time.

\begin{proposition}[{\cite[Proposition 13]{B25}}]
\label{prop:infinitegeometry}
For every fixed $\chi>0$, there is an event
$\Omega_\chi^{\mathrm{dir}}$ of probability one on which the following
holds simultaneously for all $t\in\RR$. Every non-trivial bigeodesic
$\Gamma$ in $T^t$ is $\theta$-directed for some $\theta\in(0,\infty)$,
and there is an integer $n_0=n_0(\Gamma,t,\chi)$ such that, for every
integer $n\ge n_0$,
\begin{equation}\label{eq:infiniteTF}
 \Gamma\cap[-n,n]_{\RR}
 \subseteq B_{n^{2/3+\chi}}
     \bigl(\mathbb L_{-(\theta n,n)}^{(\theta n,n)}\bigr).
\end{equation}
\end{proposition}

For a fixed direction, Lemma~\ref{lem:averaging} is strong enough
to exclude every time at which a bigeodesic in that direction visits
a marked cell. The reason is that a union over order {$n^{1/3}$}
time intervals still leaves a probability tending to zero. There
is no need to pay for a family of direction windows.

\begin{proof}[Proof of Theorem~\ref{thm:fixeddirection}]
\begingroup
Fix a deterministic $\theta\in(0,\infty)$. By integer spatial
and time translation invariance and a countable union over
$z\in\ZZ^2$ and $a\in\ZZ$, it suffices to rule out a
$\theta$-directed bigeodesic visiting the cell $\0$ at any time
in $[0,1]$. Choose a compact interval $K\subset(0,\infty)$
containing $\theta$, fix $0<e<1/3$, and then take $\delta>0$
sufficiently small for Lemma~\ref{lem:averaging}.
\par\endgroup

For each integer $n$, put $v_n=(\theta n,n)$,
{$h_n=n^{-1/3}$} and $q_n=\lceil h_n^{-1}\rceil$. Define
\[
 J_{n,i}=[ih_n,\min\{(i+1)h_n,1\}],\qquad 0\le i<q_n,
\]
and, for a time interval $J$, let
\[
 F_n^\theta(J)=
 \{\0\in\hitset_{-v_n+\mathscr I_n^\delta}^{v_n+\mathscr I_n^\delta,J}\},
 \qquad
 G_n^\theta=\bigcup_{i=0}^{q_n-1}F_n^\theta(J_{n,i}).
\]
By Lemma~\ref{lem:averaging}, time stationarity, and a union bound,
\begin{equation}\label{eq:fixedtimeunion}
 \PP(G_n^\theta)
 \le q_n Cn^{-2/3+e}
 \le{C n^{-1/3+e}}.
\end{equation}
{Since $e<1/3$}, we have $\PP(G_n^\theta)\to0$, and Fatou's lemma gives
\[
 \PP\bigl(\liminf_{n\to\infty}G_n^\theta\bigr)
 \le\liminf_{n\to\infty}\PP(G_n^\theta)=0.
\]

Work on the probability-one event $\Omega_{\delta/2}^{\mathrm{dir}}$
from Proposition~\ref{prop:infinitegeometry}. If a $\theta$-directed
bigeodesic $\Gamma$ visiting the cell $\0$ existed at some
$t\in[0,1]$, then \eqref{eq:infiniteTF} would put its intersections
with rows $-n,n$ within $n^{2/3+\delta/2}$ of $-\theta n,\theta n$,
respectively, for all sufficiently large $n$. Since
$n^{2/3+\delta/2}\le n^{2/3+\delta}/2$ for large $n$, those
intersections would lie in $-v_n+\mathscr I_n^\delta$ and
$v_n+\mathscr I_n^\delta$. The portion of $\Gamma$ between them
is a geodesic and contains its horizontal portion in the marked
cell. {Thus,} $F_n^\theta(\{t\})$, and hence $G_n^\theta$, would occur
for every sufficiently large $n$. Such a bigeodesic would therefore
force $\liminf_nG_n^\theta$, which has probability zero.

\end{proof}

\begingroup
\subsection{Logarithmic localization about the endpoint chord}
\label{sec:chordlocalization}

The quantitative hitset estimate allows logarithmic endpoint
windows. To use them for bigeodesics whose directions
range over a compact interval, we need a corresponding
logarithmic localization statement. The power allowance in
\eqref{eq:infiniteTF} does not provide this. Instead, we compare
each finite subpath with the line joining its own endpoints. Its
chord slope may vary with $n$; the later union over direction
windows will include all these slopes.

\begin{lemma}[Uniform localization about endpoint chords]
\label{lem:chordlocalization}
Fix $B>0$ and a compact interval $K\subset(0,\infty)$. There are
$C,c>0$ and $n_0<\infty$ such that the following holds for every
integer $n\ge n_0$. There is an event $\mathcal E_n^{\mathrm{ch}}(B,K)$
with
\begin{equation}\label{eq:chordfailure}
 \PP\bigl((\mathcal E_n^{\mathrm{ch}}(B,K))^c\bigr)
 \le Ce^{-c(\log n)^3}
\end{equation}
on which, simultaneously for every $t\in[0,1]$, every
$p=(x_-,-n)$, $q=(x_+,n)$ satisfying
\[
 |x_-|,|x_+|\le Bn,\qquad
 \frac{x_+-x_-}{2n}\in K,
\]
every geodesic $\Gamma:p\to q$ in $T^t$, and every $(x,j)\in\Gamma$
with integer $j$, we have
\begin{equation}\label{eq:chordlocalization}
 \left|x-\frac{n-j}{2n}x_--\frac{n+j}{2n}x_+\right|
 \le C n^{2/3}\log n.
\end{equation}
The constants are uniform over all the endpoint pairs in this statement.
\end{lemma}

\begin{proof}
\begingroup
Put $a_n=\lceil n^{2/3}\rceil$, $I_k=[ka_n,(k+1)a_n]$ for $k\in\ZZ$,
and $K_0=[\tfrac12\min K,2\max K]$. Define the interval-pair list
\[
 \mathfrak P_n=\left\{(I_k,I_l):
 I_k\cap[-Bn,Bn]\ne\varnothing,\quad
 I_l\cap[-Bn,Bn]\ne\varnothing,\quad
 \lambda_{k,l}:=\frac{(l-k)a_n}{2n}\in K_0\right\}.
\]
There are at most $C_Bn^{2/3}$ pairs. Every endpoint pair in the
statement belongs to one of them for large $n$, since rounding
changes its slope by at most $a_n/n=o(1)$.

Fix $(I_k,I_l)\in\mathfrak P_n$. After translation and horizontal
Brownian scaling, the reference endpoints $(ka_n,-n),(la_n,n)$
become $(0,0),(2n,2n)$. In the scaled coordinates of
\cite[Corollary 1.5]{GH23}, the two endpoint intervals have
bounded width, uniformly for $\lambda_{k,l}\in K_0$.
That corollary with $r=\log n\le(2n)^{1/10}$ therefore places
all geodesics between these intervals within horizontal distance
$C_Kn^{2/3}\log n$ of the reference chord, except with
probability $Ce^{-c(\log n)^3}$. The corollary holds simultaneously for all these geodesics,
including when a pair of endpoints admits more than one geodesic.
The chord of each such pair differs from the reference chord
by at most $a_n$, so increasing $C_K$ gives
\eqref{eq:chordlocalization} throughout this interval pair.

Let $\mathcal E_n^{\mathrm{stat}}$ be the static event that these
bounds hold for all pairs in $\mathfrak P_n$. The union bound gives
\[
 \PP((\mathcal E_n^{\mathrm{stat}})^c)
 \le Cn^{2/3}e^{-c(\log n)^3}.
\]
All the paths are determined by the environment in 
$[-(B+1)n,(B+1)n]\times[-n,n]$, which meets at most $C_Bn^2$
update blocks{, where $C_B$ depends on $B$.} Let $\mathcal E_n^{\mathrm{ch}}(B,K)$ be the event
that $\mathcal E_n^{\mathrm{stat}}$ holds in every configuration
of the dynamical environment during $[0,1]$. Lemma~\ref{lem:clocks} gives
\[
 \PP((\mathcal E_n^{\mathrm{ch}}(B,K))^c)
 \le Cn^{8/3}e^{-c(\log n)^3}+e^{-cn^2}
 \le C'e^{-c'(\log n)^3},
\]
as required.
\par\endgroup
\end{proof}

In particular, if such a geodesic visits the cell $\0$, there is
a point $(x,0)\in\Gamma$ with $x\in[0,1]$, and the lemma gives
\begin{equation}\label{eq:chordmidpoint}
 \left|\frac{x_-+x_+}{2}\right|
 \le 1+C n^{2/3}\log n.
\end{equation}
Thus its two endpoints are almost opposite about the marked
cell. This is the property required by the paired direction
windows; no estimate about the limiting direction on the
$n^{2/3}\log n$ scale is needed.

\subsection{Logarithmic windows for a compact range of directions}
\label{sec:logdirectionwindows}

Set
\begin{equation}\label{eq:logdirectionwidth}
 r_n=n^{-1/3},\qquad
 w_n=n^{2/3}(\log n)^{12},\qquad
 \mathscr I_n^{\log}=[-w_n/2,w_n/2]\times\{0\}.
\end{equation}
Fix $K=[\theta_1,\theta_2]\subset(0,\infty)$ and take a compact
interval $K'\subset(0,\infty)$ whose interior contains $K$. Define
\[
 \operatorname{Dir}_n(K)
 =\{j\in\ZZ:\theta_1n-w_n\le jw_n/2\le\theta_2n+w_n\},
 \qquad x_{n,j}=jw_n/2.
\]
For $j\in\operatorname{Dir}_n(K)$, put
$v_{n,j}=(x_{n,j},n)$ and
$\mathcal W_{n,j}=v_{n,j}+\mathscr I_n^{\log}$. Then
\[
 |\operatorname{Dir}_n(K)|
 \le 2(\theta_2-\theta_1)n/w_n+5\le C_Kn^{1/3},
\]
and, for all sufficiently large $n$, $x_{n,j}/n\in K'$ for every $j\in\operatorname{Dir}_n(K)$. For a time interval $J$, let
\begin{equation}\label{eq:En}
 E_n^K(J)=\bigcup_{j\in\operatorname{Dir}_n(K)}
       \{\0\in\hitset_{-\mathcal W_{n,j}}^{\mathcal W_{n,j},J}\}.
\end{equation}

\begin{lemma}[Quantitative marked-cell bound over directions]
\label{lem:directionwindows}
For every compact interval $K\subset(0,\infty)$, there are
$C_K,n_0>0$ such that, for every integer $n\ge n_0$ and every
deterministic time interval $J$ of length at most $r_n$,
\begin{equation}\label{eq:anywindow}
 \PP(E_n^K(J))
 \le n^{-1/3}\exp\{C_K(\log\log n)^2\}.
\end{equation}
\end{lemma}

\begin{proof}
We first bound the marked-cell probability for one paired
window $-v+\mathscr I_n^{\log},v+\mathscr I_n^{\log}$, where
$v=(\lambda n,n)$ and $\lambda\in K'$. Define
\begingroup
\[
 p_n(\lambda)
 :=\PP\!\left(\0\in
 \hitset_{-v+\mathscr I_n^{\log}}
          ^{v+\mathscr I_n^{\log},[0,r_n]}\right).
\]
We use the averaging set from Lemma~\ref{lem:averaging}:
\par\endgroup
\[
 S_{n,v}=\{(i,j)\in\ZZ^2: |j|\le n/128,
                             \ |i-\lambda j|\le n^{2/3}\}.
\]
Its size is at least $cn^{5/3}$. For every $z\in S_{n,v}$,
\[
 \mathscr I_n^{\log}+z\subseteq\mathcal K_{n,\lambda}^{\log},
\]
because its transverse half-width is at most
$w_n/2+n^{2/3}\le w_n$ and its row belongs to $[-n/32,n/32]$.
\begingroup
Put
\[
 \mathcal H_{n,v}^{\log}
 :=\hitset_{-v+\mathcal K_{n,\lambda}^{\log}}
          ^{v+\mathcal K_{n,\lambda}^{\log},[0,r_n]}
                         (\slab{-n/2}{n/2}).
\]
Integer translation invariance gives, for each $z\in S_{n,v}$,
\par\begingroup
\[
 p_n(\lambda)
 =\PP\!\left(z\in
 \hitset_{-v+\mathscr I_n^{\log}+z}
          ^{v+\mathscr I_n^{\log}+z,[0,r_n]}\right)
 \le\PP(z\in\mathcal H_{n,v}^{\log}).
\]
\par\endgroup
The inequality follows from the endpoint-set inclusion above
and the fact that $z$ lies in the central slab. Summing over
$z\in S_{n,v}$ and using \eqref{eq:logregionalmean} gives
\[
 \begin{aligned}
 |S_{n,v}|p_n(\lambda)
 &\le\EE\sum_{z\in S_{n,v}}\ind\{z\in\mathcal H_{n,v}^{\log}\}
 \le\EE|\mathcal H_{n,v}^{\log}|\\
 &\le n\exp\{C_{K'}(\log\log n)^2\}.
 \end{aligned}
\]
\par\endgroup
It follows that
$p_n(\lambda)\le Cn^{-2/3}\exp\{C_{K'}(\log\log n)^2\}$,
uniformly over these slopes. A union bound over
$\operatorname{Dir}_n(K)$ proves \eqref{eq:anywindow} for
$J=[0,r_n]$, after increasing $C_K$. Stationarity and monotonicity
of the hitset give the same bound for every deterministic interval
of length at most $r_n$.
\end{proof}

\subsection{Completion of the gauge-measure proof}
\label{sec:gaugeproof}

We now prove Theorem~\ref{thm:main}. The quantitative estimate
controls the number of selected time intervals. The remaining
steps verify that they cover the exceptional times and that the
sum of their gauge costs is finite over dyadic scales.

\begin{proof}[Proof of Theorem~\ref{thm:main}]
\begingroup
Fix a compact interval $K\subset(0,\infty)$.
Let $\scT_{\0}^K$ be the set of times in $[0,1]$ at which a
non-trivial bigeodesic with direction in $K$ visits the cell $\0$.
It suffices to prove $\mathcal H^H(\scT_{\0}^K)=0$ for each
such $K$, by integer spatial and time translation invariance
and the countable exhaustion of non-axial directions used below.
\par\endgroup
Choose a compact interval $K_1\subset(0,\infty)$ whose interior
contains $K$, and a constant $B>1+\max K_1$.

\paragraph{{\textbf{The selected time intervals and their expected cost.}}}
For every $n$, let $q_n=\lceil r_n^{-1}\rceil$ and put
\[
 I_{n,i}=[ir_n,\min\{(i+1)r_n,1\}],\quad 0\le i<q_n,
 \qquad
 \mathcal I_n^{K_1}
 =\{0\le i<q_n:E_n^{K_1}(I_{n,i})\text{ occurs}\}.
\]
Lemma~\ref{lem:directionwindows} implies, for some $C_1>0$ and
all sufficiently large $n$,
\begin{equation}\label{eq:counttime}
 \EE|\mathcal I_n^{K_1}|
 \le q_n n^{-1/3}\exp\{C_{K_1}(\log\log n)^2\}
 \le \exp\{C_1(\log\log n)^2\}.
\end{equation}
\begingroup
Since $H$ is nondecreasing and $|I_{n,i}|\le r_n$, for all
sufficiently large $n$,
\begin{equation}\label{eq:content}
 \EE\sum_{i\in\mathcal I_n^{K_1}}H(|I_{n,i}|)
 \le \exp\left\{C_1(\log\log n)^2
               -L(r_n)^2\log L(r_n)\right\}.
\end{equation}
{Recall that $L(r)=\log\log(1/r)$ for $0<r\le e^{-e}$.
Since $r_n=n^{-1/3}$, we have $L(r_n)=\log\log n-\log3$,
which tends to infinity.}
On taking $n=2^k$, the right-hand side of \eqref{eq:content}
is at most $k^{-2}$ for all sufficiently large $k$: the
negative term is asymptotic to $(\log k)^2\log\log k$,
whereas the positive term is $O((\log k)^2)$.
Thus the event
\[
 \Omega_{K_1}^{\mathrm{cost}}
 =\left\{\sum_{k\ge1}\sum_{i\in\mathcal I_{2^k}^{K_1}}
                        H(|I_{2^k,i}|)<\infty\right\}
\]
has probability one. The finitely many small scales contribute
only finitely many intervals and do not affect this assertion.
\par\endgroup

\paragraph{{\textbf{The cover contains every exceptional time.}}}
By \eqref{eq:chordfailure} and Borel--Cantelli, the event
\[
 \Omega_{B,K_1}^{\mathrm{ch}}
 =\{\mathcal E_{2^k}^{\mathrm{ch}}(B,K_1)
                              \text{ holds for all large }k\}
\]
has probability one. The event $\Omega_1^{\mathrm{dir}}$ from
Proposition~\ref{prop:infinitegeometry} also has probability one;
on it every non-trivial bigeodesic at every time has a direction.
Only this direction assertion is used here, not its power-scale
bound \eqref{eq:infiniteTF}. Work on
$\Omega_{B,K_1}^{\mathrm{ch}}\cap\Omega_1^{\mathrm{dir}}$.

Take $t\in\scT_{\0}^K$ and a bigeodesic $\Gamma$ witnessing it,
with direction $\theta\in K$. Define its endpoints on rows
$-n,n$ by
\[
 x_{n,-}=\Gamma(-n),\qquad x_{n,+}=\Gamma(n),\qquad
 \lambda_n=\frac{x_{n,+}-x_{n,-}}{2n},\qquad
 c_n=\frac{x_{n,+}+x_{n,-}}2.
\]
Directionality implies $x_{n,-}/n\to-\theta$,
$x_{n,+}/n\to\theta$ and $\lambda_n\to\theta$. Consequently,
for all sufficiently large $n$, both endpoint coordinates have
absolute value at most $Bn$ and $\lambda_n\in K_1$.
The portion between these rows is a geodesic and contains the
horizontal portion visiting the marked cell.

On $\Omega_{B,K_1}^{\mathrm{ch}}$, for all sufficiently large
dyadic $n$, equation~\eqref{eq:chordmidpoint} gives
$|c_n|\le1+Cn^{2/3}\log n\le w_n/4$.
Choose $j\in\operatorname{Dir}_n(K_1)$ such that
$|x_{n,j}-\lambda_n n|\le w_n/4$. This is possible because
the centers have spacing $w_n/2$ and $\lambda_n\in K_1$.
Since $x_{n,-}=c_n-\lambda_n n$ and
$x_{n,+}=c_n+\lambda_n n$, we obtain
\[
 |x_{n,-}+x_{n,j}|\le w_n/2,\qquad
 |x_{n,+}-x_{n,j}|\le w_n/2.
\]
Thus these endpoints lie in the paired windows
$-\mathcal W_{n,j},\mathcal W_{n,j}$ and their geodesic visits
the marked cell at time $t$. We have proved, simultaneously for
all these exceptional times,
\begin{equation}\label{eq:eventual}
 t\in\scT_{\0}^K
 \quad\Longrightarrow\quad
 E_{2^k}^{K_1}(\{t\})\text{ for all sufficiently large }k.
\end{equation}
The threshold may depend on the witnessing bigeodesic and time;
the events on which the implication holds do not.

\paragraph{{\textbf{Vanishing Hausdorff measure.}}}

\begingroup
On $\Omega_{K_1}^{\mathrm{cost}}\cap
\Omega_{B,K_1}^{\mathrm{ch}}\cap\Omega_1^{\mathrm{dir}}$,
for every $k_0$ we have the cover
\[
 \scT_{\0}^K\subseteq
 \bigcup_{k\ge k_0}\ \bigcup_{i\in\mathcal I_{2^k}^{K_1}}I_{2^k,i}.
\]
The interval lengths in this cover are at most $r_{2^{k_0}}$,
which tends to zero, and their total $H$-cost tends to
zero as $k_0\to\infty$, being a tail of the convergent series
defining $\Omega_{K_1}^{\mathrm{cost}}$.
By \eqref{eq:gaugemeasure},
$\mathcal H^H(\scT_{\0}^K)=0$ almost surely.

Apply this conclusion to $K=[1/b,b]$, $b\ge2$ an integer,
after every integer cell translation $z\in\ZZ^2$ and every
integer time translation $a\in\ZZ$. These translations preserve
the dynamics. By \cite[Proposition 49]{B25}, almost surely no
time admits a non-trivial axially directed semi-infinite geodesic.
Together with Proposition~\ref{prop:infinitegeometry}, this ensures
that every non-trivial bigeodesic has direction in one of these
compact intervals. It also visits some cell and occurs in some
$[a,a+1]$. Countable subadditivity of Hausdorff measure therefore
gives $\mathcal H^H(\scT)=0$ almost surely.
Since $r^s/H(r)\to0$ as $r\downarrow0$ for every $s>0$,
the dimension-zero assertion follows as well.
\par\endgroup
\end{proof}

\begingroup
For a fixed cell, bounded time interval and compact direction
range, \eqref{eq:counttime} also proves zero Hausdorff measure
for the gauge $\exp\{-A L(r)^2\}$ whenever $A>C_1$, by the
same summability calculation. The factor $\log L(r)$ in $H$
gives one gauge for every compact direction range without
requiring a common bound on the constants $C_1$.
More generally, for any fixed function $\phi$ as in {Footnote~\ref{fn:generalgauge}}
to Theorem~\ref{thm:main}, replace $\log L(r_n)$ in
\eqref{eq:content} by $\phi(L(r_n))$. Since
$\phi(L(r_n))\to\infty$, the expected costs over dyadic
scales are again summable for each compact direction range.
The same covering argument and countable union therefore prove
{the assertion in Footnote~\ref{fn:generalgauge}.}
\par\endgroup
\par\endgroup

\section{Further questions}

{Theorem~\ref{thm:regional}
bounds the expected hitset size by
$n\exp\{C(\log\log n)^2\}$ and gives a superpolynomial failure
bound at the same threshold. In view of this, we have the following natural question.}
\par\begingroup
\begin{question}[Linear dynamical hitsets]\label{qu:linearhitset}
Is the expected dynamical hitset on the critical time interval of
linear order? More precisely, is there an absolute constant $C$
such that, for every integer $n\ge1$,
\begin{equation}\label{eq:linearhitsetquestion}
 \EE\left|\hitset_{\mathscr R_n^-}^{\mathscr R_n^+,[0,n^{-1/3}]}
                     (\slab{-n/2}{n/2})\right|\le Cn?
\end{equation}
Further, do there exist constants $C,c>0$ and $\gamma\in(0,1)$,
independent of $n$ and $\alpha$, such that, for every integer
$n\ge1$ and every $\alpha\ge1$,
\begin{equation}\label{eq:uniformhitsetquestion}
 \PP\left(
 \left|\hitset_{\mathscr R_n^-}^{\mathscr R_n^+,[0,n^{-1/3}]}
                     (\slab{-n/2}{n/2})\right|>\alpha n
       \right)\le Ce^{-c\alpha^\gamma}?
\end{equation}
\end{question}
The tail estimate would imply the expectation bound by integration.
Both questions are already of interest for the single endpoint
pair $-\bn,\bn$.
\par\endgroup

Our proof first obtains simultaneous local estimates for
polynomially many possible boxes and representative endpoint pairs,
and then applies a deterministic recursion to the intervals actually
retained. A failure bound depending only on $\alpha$ would leave a
polynomial factor in $n$ after this union bound. {The accelerated recursion reduces
the accumulated branching cost to
$\exp\{C(\log\log n)^2\}$, while terminal intervals cost a fixed
power of $\log n$. These factors still depend on $n$, and
shrinking the cluster widths does not by itself control the amount
of new hitset contributed by successive generations.}

One possible approach to \eqref{eq:uniformhitsetquestion} would be to
explore the hierarchy of clusters from large scales to small scales,
retaining enough unrevealed randomness to control the smaller
excursions conditional on the information used to discover their
parent clusters. Such estimates could exploit the increasing
stability at smaller scales to control the total contribution of
the clusters encountered. This would require understanding the
dependence between scales, since a priori, discovering a cluster through an
optimizing geodesic already reveals information about its interior.
Exact independence need not be necessary, but suitable conditional
estimates are not supplied by the present proof.

Whether there actually exist exceptional times admitting a bigeodesic
with a random direction remains an open question. An affirmative
answer to the expectation bound in Question~\ref{qu:linearhitset}
should imply the a.s.\ non-existence of exceptional times. Indeed, after
fixing a unit cell and a compact slope interval $K\subset(0,\infty)$,
a first-moment
argument would show that the set of times $t\in[0,1]$ at which a
bigeodesic with direction in $K$ intersects that cell is a.s.\ finite.
Since the discrete BLPP dynamics form a well-behaved reversible
Markov process, general results ruling out isolated visits to closed
exceptional sets should then give non-existence. For a closely related
argument in dynamical percolation, see
\cite[Theorem~11.1 and its proof]{GS14}, and the original paper
\cite{HPS97}.

\begingroup
\appendix
\section{A guide to notation, dependencies and parameters}\label{app:guide}

The allowance $\ell$ is kept free until the end of the finite-hitset
proof. The choices $\ell=\log n$ and $\ell=n^a$ therefore use the
same objects and estimates. In the tables below, $j$ is a generation
index when attached to $N_j$ or $\mathcal J_j$, while $m$ denotes
a fixed spatial row during refinement.

\begingroup\small
\begin{longtable}{@{}p{0.25\textwidth}p{0.71\textwidth}@{}}
\hline
Notation & Meaning and defining reference\\[3pt]\hline
\endfirsthead
\hline
Notation & Meaning and defining reference (continued)\\[3pt]\hline
\endhead
{$\mathscr R_n^\pm$} & {The two KPZ-scale endpoint rectangles
in the regional theorem; \eqref{eq:introregions}.}\\[5pt]
$n$, $N$ & Original longitudinal scale and a shorter local scale.
Endpoint row and horizontal separations are comparable to the
stated scale, with slopes in a fixed compact interval.\\[5pt]
$p,q,D,\lambda$ & Original deterministic endpoints,
$D=j_1-j_0$, and slope $\lambda=(x_1-x_0)/D$.
Local endpoints $u,v$ obey the analogous convention in
\eqref{eq:localranges}.\\[5pt]
$\beta$, $\operatorname{Bulk}_\beta(p,q)$ & Fixed fraction excluded
near each endpoint and the resulting bulk slab; \eqref{eq:bulk}.\\[5pt]
$\ell$, $Q$, $h$ & The allowance $\ell$ controls the tolerances and failure
probabilities; it is chosen once for each $n$ and kept fixed at
every local scale. The fixed exponent $Q$ sets the cutoff
$\ell^Q$, below which we count cells directly. The refinement
is carried out over $[0,h]$, with $h=n^{-1/3}\ell^{-4}$;
see \eqref{eq:allowanceparameters}.\\[5pt]
$\Phi_n(\ell)$ & Accumulated logarithmic cost
$(\log\ell)\log(\log n/\log\ell)$; \eqref{eq:allowancecost}.
It bounds the logarithm of the hitset size divided by $n$.\\[5pt]
$P$ & Regularity margin $C_{\mathrm{reg}}\ell^2$ for global
transversal displacement, increments and horizontal row lengths;
\eqref{eq:regularitymargin}--\eqref{eq:rowlength}.\\[5pt]
$\mathcal L_n$, $L_i$, $F_i^t$ &  An indexed family of deterministic
path classes, their horizontal separations, and their optimized
passage values in Lemma~\ref{lem:timeuniform}.\\[5pt]
$A_N$, $D_N$ & Stability bound
$A_N=\ell\sqrt{Nh}+\ell^2$ and resulting cover width
$D_N=C_{\mathrm{width}}(Nh\ell^2+\ell^4)$;
\eqref{eq:AN}, Proposition~\ref{prop:local}.\\[5pt]
$Z^t$ & The exit-constrained routed profile
for the endpoints and crossing row under consideration;
\eqref{eq:routed}.\\[5pt]
$\nearmax^\alpha(f;W)$ & Largest number of $\alpha$-near-maximisers
in $W$ separated by at least $\alpha^2$; \eqref{eq:NMdefinition}.\\[5pt]
$R$, $\omega_N$ & Capture allowance and transverse window width
$\omega_N=RN^{2/3}$. The refinement uses $R=\ell^{12}$.\\[5pt]
$U^\pm_{N;c,m}$, $D^\pm_{N;c,m}$ & Endpoint windows and outer
sampling regions for scale $N$ and center $(c,m)$;
\eqref{eq:localwindows}--\eqref{eq:outerslabs}.
They also depend on the fixed slope and the allowance $R$.\\[5pt]
$\cQ_{N;c,m}$ & Independent Poisson cloud of endpoint pairs for
one box, with intensity $N^{-10/3}R^4$ and mean at most $CR^6$;
Lemma~\ref{lem:capture}.\\[5pt]
$\mathcal W_n$ & The interval $[x_0,x_1]$, containing every
horizontal coordinate of a staircase from $p$ to $q$.\\[5pt]
$J_{N,k}$, $c_{N,k}$ & Deterministic mesh interval of width
$N^{2/3}$ and its midpoint; \eqref{eq:deterministicmesh}.\\[5pt]
$\cI_n(N)$, $\cM_n$ & Mesh indices meeting $\mathcal W_n$ and
integer bulk rows; \eqref{eq:boxindices}. The row set is the same
at every scale.\\[5pt]
$\mathcal S_n$, $\mathfrak B_n(\mathcal S_n)$ & Deterministic list
of local scales and its boxes, all prepared before the actual
parents are selected; Section~\ref{sec:refinement}.\\[5pt]
$\mathscr Q_n$, $M_n$ & {$\sigma$-algebra} of all prepared clouds and
total number of sampled entries in the refinement proof;
\eqref{eq:totalcloudmean}. In Proposition~\ref{prop:generalregional}, $M_n$ instead
counts the pairs in $\cQ_{n;0,0}$, the cloud used to represent
geodesics between the original endpoint regions
$U^-_{n;0,0}$ and $U^+_{n;0,0}$.\\[5pt]
$E_m$, $E_m(J)$ & All exits of $\Gamma_p^{q,t}$ from row $m$
during $[0,h]$, and their restriction $E_m\cap J$;
\eqref{eq:rootexitset}.\\[5pt]
$\cR_n,\cA_n,\cD_n,\cX_n$ & Regularity, simultaneous capture,
cloud-count, and representative-cover events in the proof of
Lemma~\ref{lem:refinement}. Their intersection is
$\cE_n^{\mathrm{ref}}$.\\[5pt]
$\mathcal C(N,c,m)$ & A finer cover for one possible parent,
with at most $\ell^{163}$ intervals of length $D_N$;
\eqref{eq:refinementcover}.\\[5pt]
$N_0$, $N_j$, $g_n$ & Initial scale, recursively chosen scales,
and first generation below $\ell^Q$;
\eqref{eq:initialscale} and \eqref{eq:recursion}.\\[5pt]
$b_j$, $J_n(\ell)$ & Logarithmic scale $b_j=\log(n/N_j)$ and
deterministic upper bound on the number of generations;
\eqref{eq:acceleration}--\eqref{eq:generations}.\\[5pt]
$\mathcal G_j$, $R_m$, $\mathcal J_j(m)$ & Full mesh at generation
$j$, initial exit window, and retained interval family;
\eqref{eq:initialfamily}--\eqref{eq:childfamily}. Their union covers
$E_m$ by \eqref{eq:coverinvariant}.\\[5pt]
$c_J$, $\mathcal C_j(m,J)$ & Midpoint of a retained parent and
its finer cover $\mathcal C(N_j,c_J,m)$, used in
\eqref{eq:childfamily}.\\[5pt]
$\widehat J$, $H_m$ & A terminal interval enlarged to the left
by $P$, and the row hitset. Step 4 of the proof of
Proposition~\ref{prop:generalpoint} uses these to pass from exit
covers to cell covers; \eqref{eq:rowcount}.\\[5pt]
$\mathcal K_{n,\lambda}(\ell)$,
$\mathcal K_{n,\lambda}^{\log}$ & Enlarged endpoint region and
its specialization at $\ell=\log n$;
\eqref{eq:Kregion}, \eqref{eq:logendpointregions}.\\[5pt]
$\mathscr I_n^\delta$, $\mathscr I_n^{\log}$ & Horizontal windows
with power and logarithmic enlargement, respectively;
\eqref{eq:smallI}, \eqref{eq:logdirectionwidth}. The former are
used for a fixed direction, the latter for the Hausdorff gauge.\\[5pt]
$\mathcal W_{n,j}$ & Endpoint segment on row $n$, paired with
$-\mathcal W_{n,j}$ in the finite direction cover; \eqref{eq:En}.\\[5pt]
$\mathcal E_n^{\mathrm{ch}}(B,K)$ & Event of simultaneous
localization about finite endpoint chords;
Lemma~\ref{lem:chordlocalization}.\\[5pt]
$\lambda_n$, $c_n$ & Slope and midpoint of the chord joining
a bigeodesic's exits on rows $-n,n$; Section~\ref{sec:gaugeproof}.
These quantities may vary with $n$.\\[5pt]
$E_n^K(J)$ & Event that the marked cell is visited during $J$
by a geodesic in one of the paired direction windows;
\eqref{eq:En}. It is an event, whereas $E_m$ is an exit set.\\[5pt]
$I_{n,i}$, $\mathcal I_n^{K_1}$ & Deterministic time intervals of
length at most $n^{-1/3}$ and the random indices retained in
the final time cover; Section~\ref{sec:gaugeproof}.\\[5pt]
$L(r)$, $H$, $\mathcal H^H$ & Logarithmic scale,
gauge and associated Hausdorff measure; \eqref{eq:hausdorffgauge},
\eqref{eq:gaugemeasure}. The expected cover cost is bounded
in \eqref{eq:content}.\\[3pt]
\hline
\end{longtable}
\endgroup

\paragraph{\textbf{How the estimates fit together.}}
There are two uses of geodesic regularity. It locates the random
cut endpoints inside the capture windows, and it bounds the
horizontal row portions when exits are converted to visited cells.
Passage-time stability and the static near-maximiser count give
the {exit cover} for a deterministic endpoint pair in
Proposition~\ref{prop:local}. The independence of the clouds allows
this estimate to be applied to their sampled pairs; capture then
transfers the covers to the actual cut subpaths in
Lemma~\ref{lem:refinement}. Proposition~\ref{prop:generalpoint}
uses that lemma's common event and counts retained intervals
deterministically. Finally, Proposition~\ref{prop:generalregional} applies
Poisson capture to geodesics between the original endpoint
regions $U^-_{n;0,0}$ and $U^+_{n;0,0}$. Their portions in
the central slab are represented by the sampled endpoint pairs,
to which the point-to-point hitset estimate applies. Subdividing
the critical time interval into pieces of length at most $h$
then gives the regional bound over $[0,n^{-1/3}]$. Thus the local estimates are
reused at each scale, while the passage from one generation to
the next needs only the cover invariant \eqref{eq:coverinvariant}.

\paragraph{\textbf{Why these allowance powers suffice.}}
The numerical powers below are fixed throughout the argument.
Their purpose is to retain uniform estimates as $\ell$ varies,
including its smallest choice $\log n$.

\begin{center}\small
\begin{tabular}{p{0.29\textwidth}p{0.63\textwidth}}
\hline
Choice & Requirement it satisfies\\[3pt]\hline
$P=C\ell^2$ & Dominates the logarithmic modulus factors in
geodesic regularity; Lemma~\ref{lem:regularity}.\\[5pt]
$R=\ell^{12}$ & Contains the regularity windows and gives
capture failure $Cn^{13}(\log n)e^{-c\ell^{36/11}}$;
$36/11>2$ allows the final failure bound $Ce^{-c\ell^2}$.\\[5pt]
$\ell^{73}$ cloud pairs & Dominates the mean $C\ell^{72}$,
with failure at most $e^{-\ell^{73}/2}$ per box.\\[5pt]
$\ell^{90}$ near-maximisers & Dominates the Brownian-comparison
loss: the logarithmic exponent is
$-c\ell^{90}+C\ell^{14}(\ell^{90})^{5/6}
=-c\ell^{90}+C\ell^{89}$.\\[5pt]
$Q>\max\{1080,12/\bar\eta\}$ & At $N\ge\ell^Q$, ensures
$\ell^{90}=o(N^{1/12})$ for Brownian comparison and
$\ell^{12}\le N^{\bar\eta}$ for capture. It also implies the
weaker restrictions in regularity and stability.\\[5pt]
$\ell\le n^{1/(2Q)}$ & Puts the terminal cutoff
$\ell^Q\le n^{1/2}$ below the initial scale $N_0\asymp n$.
The constants needed in the proof are uniform in this range.\\[5pt]
$h=n^{-1/3}\ell^{-4}$ & Gives
$A_N/N^{1/3}\le\ell^{-1}+\ell^{2-Q/3}$ and starts the
contraction. Subdividing $[0,n^{-1/3}]$ into intervals of length
at most $h$ costs a factor of at most $\lceil\ell^4\rceil$;
Lemma~\ref{lem:timesubdivision}.\\[3pt]
\hline
\end{tabular}\end{center}

The proof uses two different counts. At each scale there are $O(n)$ mesh intervals and $O(n)$
bulk rows, hence $O(n^2)$ possible boxes. There are at most
$\log n$ scales, and all $O(n^2\log n)$ boxes enter the
probability bound in Lemma~\ref{lem:refinement}.
Only the retained parents enter \eqref{eq:branchcount}; each has
at most $3\ell^{163}$ children. The number of generations is
bounded by \eqref{eq:generations}, and the terminal intervals
cost at most $C\ell^{2Q/3}$ cells each. Multiplying the initial cover size, the branching factors
and the terminal cell cost, and then summing over $O(n)$ rows,
gives $n\exp\{C\Phi_n(\ell)\}$ in \eqref{eq:cellcount}.

For $\ell=\log n$, the logarithm of this loss is
$O((\log\log n)^2)$. For $\ell=n^a$, it is
$O(a\log(1/a)\log n)$. Equation~\eqref{eq:allowancechoice}
chooses $a$ in terms of the desired size exponent $\varepsilon$.
Thus the same local estimates and cover recursion yield
both regimes by choosing different values of $\ell$.
\par\endgroup

\printbibliography
\end{document}